\documentclass[11pt]{amsart}

\usepackage[T1]{fontenc}
\usepackage[cp1250]{inputenc}
\usepackage{ae,aecompl}

\usepackage[arrow, matrix, curve]{xy}
\usepackage{amsmath,amsthm,amssymb,dsfont}
\usepackage{amscd,graphics}

\usepackage{graphicx}
\usepackage{xcolor}

\newenvironment{Proof of}[1]{\textbf{Proof #1.}}{$\qquad \blacksquare$\par}

\usepackage{a4wide}
\DeclareMathOperator{\spane}{span}
\DeclareMathOperator{\clsp}{\overline{span}}

\DeclareMathOperator{\supp}{supp}
\DeclareMathOperator{\alg}{alg}

\DeclareMathOperator{\Aut}{Aut}

\DeclareMathOperator{\id}{id}

\DeclareMathOperator{\reg}{reg}
\DeclareMathOperator{\pos}{pos}

\DeclareMathOperator{\Prim}{Prim}

\newcommand{\B}{\mathcal B}

\newcommand{\M}{\mathcal M}

\newcommand{\NN}{\mathcal N}

\newcommand{\QQ}{\mathcal Q}

\newcommand{\OO}{\mathcal{O}}

\newcommand{\al}{\alpha}
\newcommand{\p}{\varphi}

\newcommand{\A}{\mathcal A}

\renewcommand{\AA}{\mathbb A}
\newcommand{\D}{\mathcal D}
\newcommand{\G}{\mathcal G}
\newcommand{\C}{\mathbb C}
\newcommand{\R}{\mathbb R}

\newcommand{\Z}{\mathbb Z}
\newcommand{\N}{\mathbb N}

\newtheorem{thm}{Theorem}[section]
\newtheorem{lem}[thm]{Lemma}
\newtheorem{prop}[thm]{Proposition}
\newtheorem{cor}[thm]{Corollary}

\newtheorem{thmx}{Theorem}

\theoremstyle{definition}

\newtheorem{defn}[thm]{Definition}
\newtheorem{ex}[thm]{Example}
\newtheorem{rem}[thm]{Remark}

\begin{document}
 \thispagestyle{empty}
   \title[$C^*$-algebras associated with transfer operators]{$C^*$-algebras associated to  transfer operators for countable-to-one maps}
  \author[Bardadyn]{Krzysztof Bardadyn}

\author[Kwa\'sniewski]{Bartosz K. Kwa\'sniewski}
	
	\author[Lebedev]{Andrei V. Lebedev}
  % \address{Institute of Mathematics,  University  of Bialystok\\ul. Akademicka 2,  PL-15-267  Bialystok,   Poland}
   %\email{bartoszk@math.uwb.edu.pl}\urladdr{http://math.uwb.edu.pl/~zaf/kwasniewski}
 %  \keywords{...}
   \thanks{This work was supported by the National Science Centre,
Poland, grant number 2019/35/B/ST1/02684}
    \begin{abstract} Our initial data is a  transfer operator $L$ for a continuous, countable-to-one  map $\varphi:\Delta \to X$
	defined on an open subset of a locally compact Hausdorff space $X$. Then $L$ may be identified with a `potential', i.e. a map $\varrho:\Delta\to X$ that need not be continuous unless $\varphi$ is a local homeomorphism. %We analyze $C^*$-algebras associated	to
	We define the crossed product  $C_0(X)\rtimes L$ as a universal $C^*$-algebra with explicit generators and relations, and give an explicit faithful representation
	of $C_0(X)\rtimes L$  under which it is generated by weighted composition operators.
We explain its relationship with Exel-Royer's crossed products, quiver $C^*$-algebras of Muhly and Tomforde, $C^*$-algebras associated to complex or self-similar dynamics
 by Kajiwara and Watatani, and groupoid $C^*$-algebras associated to Deaconu-Renault groupoids.

We describe spectra of core subalgebras of  $C_0(X)\rtimes L$, prove  uniqueness theorems for $C_0(X)\rtimes L$ and characterise simplicity of $C_0(X)\rtimes L$. We give efficient criteria for  $C_0(X)\rtimes L$ to be purely infinite simple and in particular a Kirchberg algebra.
		\end{abstract}
		
		\maketitle

%\setcounter{tocdepth}{1}
%\tableofcontents

\section*{Introduction.}

Since 1970's transfer operators are indispensable tools in thermodynamical formalism  and ergodic theory \cite{Bowen}, and even earlier such operators, named averaging operators, played an important role in the study of
Banach spaces $C(X)$ of continuous functions on a compact space $X$, see \cite{Pelczynski}.
They are also crucial in the study of spectrum of weighted composition operators, see \cite{t-entropy}, \cite{bar_kwa}.
 Transfer operators as a tool to construct $C^*$-algebras,  were explicitly used for the first time by Exel in \cite{exel3} to present Cuntz-Krieger algebras as crossed products associated to topological Markov chains.
 Since then a number of  generalisations and modifications of such crossed products were introduced, see for instance \cite{er}, \cite{exel_renault}, \cite{larsen}, \cite{BroRaeVit}, \cite{Brownlowe}. Their general structure  as Cuntz-Pimsner algebras  is now quite well-understood, see  \cite{br}, \cite{kwa_Exel}. However, the detailed analysis of the associated $C^*$-algebras is usually limited to the case where the underlying mapping is a local homeomorphism on a compact Hausdorff space,
see \cite{exel_vershik}, \cite{er}, \cite{CS}, \cite{BroRaeVit}, \cite{Brownlowe}. The exceptions  are $C^*$-algebras associated to rational maps \cite{Kajiwara_Watatani0} or maps whose inverse branches form a self-similar systems \cite{Kajiwara_Watatani}, \cite{Kajiwara_Watatani16}.
All these $C^*$-algebras can  be viewed as crossed products by transfer operators for \emph{finite-to-one} maps admitting at most finite number irregular points.
However, in many problems there is a natural need to study tranfer operators for partial continuous maps that are countable-to-one. This concerns
in particular infinite graph $C^*$-algebras \cite{Raeburn}, \cite{BroRaeVit}, \cite{Brownlowe}, \cite{kwa_Exel} or thermodynamic formalism for countable Markov shifts, interest in which has been growing in recent years, see \cite{Sarig}, \cite{exel_laca}, \cite{Bissacot-Exel-Frausino-Raszeja}, \cite{Bissacot-Exel-Frausino-Raszeja2}.  
In the present paper we give a general, comprehensive  account of the main structural results for crossed products by transfer operators for arbitrary partial continuous maps that are countable-to-one.

More specifically we consider a continuous map $\varphi:\Delta \to X$ defined on an open subset $\Delta$ of a  locally compact Hausdorff space $X$. We assume that   $\varphi^{-1}(y)$ is countable for all $y\in \Delta$. Then every bounded transfer operator for $\varphi$ is a map $L:C_0(\Delta)\to C_{0}(X)$ given by the formula
$$
L(a)(y)=\sum_{x\in\varphi^{-1}(y)}\varrho(x)a(x)
$$
where $\varrho: \Delta\to [0,\infty)$ is a map that we call a  \emph{potential}. 
%We emphasize that
%the possibility of defining a transfer operator $L$ by means of a (strictly positive) potential $\varrho$ is equivalent to the assumption that  $\varphi^{-1}(y)$
%is countable for all $y\in X$, as in general a transfer operator is defined  by means of a
%measure valued function $y \to \mu_y$, cf. \eqref{equ:transfer_operator_form} below.
A potential $\varrho$ is in general only upper semi-continuous and
the main role in our analysis is played by the following two sets:
$$
\Delta_{\pos}:=\{x\in \Delta: \varrho(x) >0\},\qquad \Delta_{\reg}:=\{x\in \Delta_{\pos}:  \varrho \text{ is continuous at }x\}.
$$
So  $\Delta_{\reg}\subseteq \Delta_{\pos}\subseteq \Delta\subseteq X$. As we show $\Delta_{\reg}$ is an open subset of $X$ and the restricted map $\varphi:\Delta_{\reg}\to X$ is a local homeomorphism.
We define the \emph{crossed product} $C_0(X)\rtimes L$ as a universal $C^*$-algebra generated
by the $C^*$-algebra $C_0(X)$ and weighted operators $at$,  for $a\in C_0(\Delta)$, subject to relations
$$
L(a)=tat^*, \,\, a\in C_0(\Delta), \qquad a\sum_{i=1}^n u_i^{K} tt^* u_i^{K}  =   a, \,\, a\in C_c(\Delta_{\reg})
$$
where $u_i^K$'s is a suitably normalized partition of unity on $K:=\supp(a)$, see \eqref{eq:quasi-basis} below.
Apart from the case of covering maps on compact spaces treated in  \cite{exel_vershik} this is the first general description of
the crossed product  $C_0(X)\rtimes L$ in terms of explicit relations coming from $L$. In other works the corresponding crossed product is
usually defined and analyzed as the  Cuntz-Pimsner algebra $\OO_{M_L}$ associated to a $C^*$-correspondence $M_L$. We prove that $C_0(X)\rtimes L$ is isomorphic to $\OO_{M_L}$ (Theorem \ref{thm:crossed_product_Cuntz-Pimsner}). We do not know whether in general $C_0(X)\rtimes L$ can be  naturally modelled by a topological quiver of Muhly and Tomforde \cite{mt}. 
One of our main structural result is the following the following version of (Cuntz-Krieger) uniqueness theorem (see Theorems \ref{thm:isomorphism}, \ref{thm:commutant_topological_freeness})
that generalizes  the corresponding results from \cite{exel_vershik}, \cite{er}, \cite{CS}, \cite{BroRaeVit}.
\begin{thmx}\label{Theorem_A} The following conditions are equivalent:
\begin{enumerate}
\item\label{enu:isomorphismA1}  Every representation of $C_0(X)\rtimes L$  is faithful  provided it is faithful on $C_0(X)$.

\item\label{enu:isomorphismA2} The orbit representation of $C_0(X)\rtimes L$ on $\ell^2(X)$ is faithful; this representation  sends function in $C_0(X)$ to operators of multiplication and the generator $t$ to the weighted composition operator
$Th:= \sqrt{\varrho} h\circ \varphi$,

\item\label{enu:isomorphismA3}  The map $\varphi: \Delta_{\reg}\to X$ is topologically free, that is the set of periodic points whose orbits are contained in  $\Delta_{\reg}$ has empty interior.

\end{enumerate}
If in addition $\Delta_{\pos}=\Delta_{\reg}$ the above conditions are further equivalent to
\begin{enumerate}\setcounter{enumi}{3}
\item\label{enu:isomorphismA4}  $C_0(X)$ is a maximal abelian $C^*$-subalgebra of $C_0(X)\rtimes L$.
%\item $C_0(X)$ is a Cartan subalgebra of $C_0(X)\rtimes L$, in the sense of \cite{Re}
\end{enumerate}

\end{thmx}
In general we characterise faithful representations of $C_0(X)\rtimes L$ in terms of a canonical  \emph{generalized expectation} $G$ for the inclusion $C_0(X)\subseteq  C_0(X)\rtimes L$ (Theorem \ref{cor:Expectation_Invariance}).
We construct $G$ using a \emph{regular representation}  of $C_0(X)\rtimes L$ on $\ell^2(X\rtimes\Z)$.   If $\Delta_{\pos}=\Delta_{\reg}$, then $G$  is a genuine conditional expectation and $C_0(X)\rtimes L$ is naturally isomorphic to the $C^*$-algebra of the Renault-Deaconu groupoid  for the partial local homeomorphism $\varphi:\Delta_{\reg}\to X$
(see Theorem \ref{thm:local_homeo_crossed_groupoid}). Then  \ref{enu:isomorphismA4} in Theorem \ref{Theorem_A} says that $C_0(X)$
 is a Cartan subalgebra of $C_0(X)\rtimes L$ in the sense of Renault \cite{Re}. We show by example that if $\Delta_{\pos}\neq\Delta_{\reg}$,  then topological freeness of  $\varphi: \Delta_{\reg}\to X$ is not sufficient for maximal abeliannes of $C_0(X)$  in $C_0(X)\rtimes L$.

We say that $L$ is \emph{minimal} if there are no non-trivial open subsets $U\subseteq X$ such that $\varphi(U\cap \Delta_{\pos})\subseteq U$ and $\varphi^{-1}(U)\cap \Delta_{\reg}\subseteq U$.
As a corollary to Theorem \ref{Theorem_A} we get the following characterisation of simplicity (see Theorem \ref{thm:test_of_simplicity}):

\begin{thmx}\label{Theorem_B}
If $\Delta_{\reg}$ is infinite, then $C_0(X)\rtimes L$ is simple if and only if $L$ is minimal.
\end{thmx}

Inspired by notions of locally contractive groupoids \cite{Anantharaman-Delaroche} and contractive topological graphs \cite{ka1} we define \emph{contractive transfer operators}, see Definition \ref{defn:contractive}. For such operators we get (see Theorem \ref{thm:purely_infinite} and Corollary \ref{cor:Kirchberg}):

\begin{thmx}\label{thmx:purely_infinite}
If $L$ is minimal and contractive, then  $C_0(X)\rtimes L$ is  purely infinite and simple.
If in addition $X$ is second countable, then $C_0(X)\rtimes L$ is a UCT-Kirchberg algebra (and so it is classifiable by its $K$-theory).
 \end{thmx}

We illustrate the power of Theorem \ref{thmx:purely_infinite} by showing that it covers and unifies all purely infinite results in  \cite{Kajiwara_Watatani0}, \cite{Kajiwara_Watatani}, \cite[Section 4]{Anantharaman-Delaroche}, \cite{Exel_Huef_Raeburn}  (Examples \ref{ex:pure_infinite1}, \ref{ex:pure_infinite2}, \ref{ex:pure_infinite3}).

Another fundamental  $C^*$-algebra associated to $L$ is the fixed point algebra of the canonical circle gauge action on $C_0(X)\rtimes L$. It is a direct limit $A_{\infty}=\overline{\bigcup_{n=0}^\infty A_n}$ of $C^*$-algebras
$$
A_n=\clsp\{ at^k t^{*k} b: a,b\in C_0(\Delta_k), k=0,...,n\}
$$
where $\Delta_k=\varphi^{-k}(\Delta)$ is the natural domain for $\varphi^k$. The algebras $A_n$ are interesting in their own right, see \cite{kumjian0}, and the $C^*$-algebra $A_\infty$  has important dynamical interpretations. For special self-similar maps $A_{\infty}$ was studied in  \cite{Kajiwara_Watatani16}.
When $\Delta_{\reg}=\Delta$, so that $\varphi$ is a local homeomorphism, then $A_{\infty}$ is a groupoid $C^*$-algebra of  a generalized approximately proper equivalence relation  on $X$.
This is a crucial tool in  the study of Gibbs states via the Radon-Nikodym problem  \cite{Renault05}, \cite{Bissacot-Exel-Frausino-Raszeja2}.
Also stable $C^*$-algebras for irreducible Smale spaces are naturally Morita equivalent to algebras of the form $A_{\infty}$ (see Remark \ref{rem:Wieler_solenoids} below). 
Putting $
\Delta_{\pos,n}:=\{x\in \Delta_n: \prod_{i=0}^{n-1}\varrho(\varphi^i(x)) \neq 0\}
$ we describe the spectra of $A_n$, $n\in \N$, and $A_{\infty}$, as follows (see Proposition \ref{prop:spectrum_of_K_n} and Theorems  \ref{thm:spectra_of_A_n}, \ref{thm:primitive_groupoid}):

\begin{thmx}\label{thmx:core_subalgebras}
 For each $n\in \N$ the algebra $A_n$ is postliminary (Type I) and up to unitary equivalence all its irreducible representations are subrepresentations of the orbit representation on $\ell^2(X)$.
Namely, we have a bijection
\begin{equation}\label{eq:spectra_of_A_n}
\widehat{A}_n  \cong \left(\bigsqcup_{k=0}^{n-1}  \varphi^{k}(\Delta_{\pos,k})\setminus\Delta_{\reg} \right)\sqcup  \varphi^{n}(\Delta_{\pos,n}),
\end{equation}
where a representation corresponding to $y\in \varphi^{k}(\Delta_{\pos,k})$ is the restriction of the orbit representation to the subspace $\ell^2(\varphi^{-k}(x))\subseteq \ell^2(X)$.
The Jacobson topology on $\widehat{A}_n$ in general is finer than the pushout topology on the right-hand side of \eqref{eq:spectra_of_A_n}. But
the two topologies coincide for instance when the potential $\varrho$ is continuous,
and if in addition $X$ is second countable, then the primitive ideal space of $A_{\infty}$ is homeomorphic to the quasi-orbit space:
$$
\Prim (A_\infty) \cong X/\sim
$$
 where  $x\sim y$ iff $\overline{\OO(x)}=\overline{\OO(y)}$ and  the orbit of $x\in X$ is $\OO(x):=\bigcup_{k=0, x\in \Delta_{\pos, k}}^\infty \varphi^{-k}(\varphi^{k}(x))$.

 \end{thmx}

\smallskip

The paper is organized as follows. In Section \ref{sec:transfer_operators} we discuss transfer operators for partial maps and
the properties of the associated potential $\varrho$. Covariant representations  for transfer operators are introduced in Section \ref{sec:covariant_representations}.
The crossed product $C_0(X)\rtimes L$ and its relationship with  previous constructions modeled by Cuntz-Pimsner algebras are discussed in Section \ref{sec:crossed_products}.

In Section \ref{sec:invariance_uniqueness_theorems}, based on the well known gauge-invariance uniqueness for Cuntz-Pimsner algebras  we prove faithulness of 
the regular representation of  $C_0(X)\rtimes L$, which in turn leads us to a generalized expectation-invariance uniqueness theorem. We use the latter in Section \ref{sec:groupoid_picture}
to prove that Reanult-Deaconu groupoid $C^*$-algebras assocaited to a a local homeomorphism $\varphi$ is  naturally isomorphic to $C_0(X)\rtimes L$.
Section \ref{sec:Spectra of the core subalgebras} is devoted to description of the spectrum of algebras $A_n$ and $A_\infty$ and it contains the proof of Theorem \ref{thmx:core_subalgebras}.
Section \ref{Sec:Topological_freeness} introduces topological freeness for transfer operators and  contains  proof of  Theorems \ref{Theorem_A} and \ref{Theorem_B}.
Finally in Section \ref{Sec:simplicity_pure_infiniteness} we  give criteria for pure infiniteness for $C_0(X)\rtimes L$ (we prove Theorem \ref{thmx:purely_infinite}).

\section{Transfer operators for partial maps and potentials}\label{sec:transfer_operators}
Throughout this paper  $\p:\Delta \to X$ is a continuous map defined on an open subset $\Delta$ of a locally compact space $X$.
We refer to $(X,\p)$ as to a \emph{partial dynamical system}. In addition we will fix a bounded transfer operator for $(X,\p)$,
which we will interpret  as a  \emph{potential} for the system $(X,\p)$.
Namely, let us denote by $C_0(X)$ the $C^*$-algebra of continuous functions on $X$ that vanish at infinity. We treat $C_0(\Delta)$  as an ideal in $C_0(X)$.
By a \emph{transfer operator} for $(X,\p)$ we mean a positive linear map $L:C_0(\Delta)\to C_{0}(X)$ satisfying
\begin{equation}\label{eq:transfer_identity}
L((a\circ \p)b)=aL(b), \qquad a\in C_{0}(X), b\in C_0(\Delta).
\end{equation}
%Recall that positive linear maps between $C^*$-algebras are necessarily bounded.
\begin{rem}\label{rem:definition_of_transfer}
We could allow the transfer operator  $L$ to attain values in the bounded continuous functions $C_{b}(X)$,
but then \eqref{eq:transfer_identity} forces $L$ to take values in $C_{0}(X)$ anyway. Indeed, if $b\in C_c(\Delta)$ is compactly supported with the support $K$ then taking $a\in C_c(X)$ such that $a|_{\varphi(K)}\equiv 1$ we get
$
L(b)=L((a\circ \p)b)=aL(b)\in C_c(X).
$
Thus  transfer operators map compactly supported functions to compactly supported ones.
 %The transfer identity \eqref{eq:transfer_identity} implies that $L(C_0(\Delta))$ is an ideal in $C_{0}(X)$ while positivity of $L$ implies that $L(C_0(\Delta))$ is a $*$-ideal.

\end{rem}
Transfer operator could be defined in  purely $C^*$-algebraic terms as follows. Let $I$ be an ideal in a $C^*$-algebra  $A$ (by which we always mean a closed two-sided ideal).
Let $\alpha:A\to M(I)$ be a non-degenerate $*$-homomorphism  from $A$ to the multiplier $C^*$-algebra  $M(I)$ of $I$.
 Such maps are called \emph{partial endomorphisms} of $A$ in \cite[Definition 1.1]{er}, \cite[Definition 3.12]{katsura1}. A (bounded) \emph{transfer operator for $\alpha$} is a positive linear
map $L:I\to A$ satisfying
\begin{equation}\label{e-1}
 L(\alpha(a)b)=aL(b), \qquad a\in A, b\in I.
\end{equation}
Positivity implies that $L$ is bounded and $*$-preserving. 
In addition the transfer equality \eqref{e-1} implies that $L(I)$ is an ideal.
Transfer operators introduced in \cite[Definition 1.2]{er} are defined on a not necessarily closed ideal in $I$, and thus
in general they are unbounded.

Having the triple $(A,\alpha, L)$ as above and assuming that $A=C_0(X)$, we necessarily have  $I=C_0(\Delta)$, for an open set $\Delta\subseteq X$,  and
 $$
\al(a)= a(\p(x)), \qquad  x\in \Delta, \qquad a\in A,
$$
for a continuous map $\p:\Delta\to X$. Accordingly, $M(I)=C_b(\Delta)$ consists of continuous bounded functions and $\alpha:C_0(X)\to C_b(\Delta)$.
In particular, $\alpha:C_0(X)\to C_0(\Delta)\subseteq C_0(X)$ is an endomorphism of $C_0(X)$ %(we  have $\alpha (C_0(X)) \subseteq C_0(\Delta)$)
if and only if the map $\p:\Delta\to X$ is \emph{proper}, i.e.  the preimage of every compact set in $X$ is compact in $\Delta$.
Furthermore, denoting by $\M(X)$  the space of finite regular borel measures on $X$ equipped with the weak$^*$ topology, a  transfer operator $L:C_0(\Delta)\to C_{0}(X)$ for $\alpha$ is of the form
\begin{equation}\label{equ:transfer_operator_form}
L(a)(y)=\int_{\varphi^{-1}(y)} a(x) d\mu_y(x), \qquad a\in C_0(\Delta), \, y\in X,
\end{equation}
where
 $X\ni y \longmapsto \mu_y \in \M(X)$ is a continuous map such that $\supp\mu_y\subseteq \varphi^{-1}(y)$ for every $y\in X$
 and $\sup_{y\in X}\mu_{y}(X)=\|L\|<\infty$, cf., for instance, \cite{t-entropy},
\cite{kwa-trans}, \cite{kwa_Exel}. If the preimages of $\varphi$ are countable, then
this measure valued function can be replaced by a number valued function. We assume this throughout the paper.

\textbf{Standing assumption:}
\begin{equation}\label{eq:countable_to_one}
 |\varphi^{-1}(y)|\leq \aleph_0 \qquad \text{ for all } y\in X.
\end{equation}
%The main object of our study is  a transfer operator    for the endomorphism $\alpha$.
Under this assumption,  the measures $\{\mu_y\}_{y\in X}\subseteq \M(X)$ appearing
 in \eqref{equ:transfer_operator_form} are discrete and putting
$
\varrho(x):=\mu_{\p(x)}(\{x\})$, $x\in \Delta,
$
we get that the corresponding transfer operator  is given by
\begin{equation}\label{eq:transfer operator}
L(a)(y)=\sum_{x\in\varphi^{-1}(y)}\varrho(x)a(x).
\end{equation}
We refer to the map $\varrho: X\to [0,\infty)$ as to the \emph{potential associated to $L$},  and we put
$$
 \Delta_{\pos}:=\Delta\setminus  \varrho^{-1}(0)=\{x\in \Delta: \varrho(x)>0\}.
$$
Obviously, every map admits a zero transfer operator (so that $\Delta_{\pos}=\emptyset$), but there is a large and important class of maps that admit a transfer operator
with $\Delta_{\pos}=\Delta$. This concerns essentially all local homeomorphisms, see Theorem  \ref{thm:local_homeo_crossed_groupoid} below,
and all open finite-to-one maps on compact spaces. This last claim follows from   
 a result of Pavlov and Troitsky \cite[Theorem 1.1]{Pavlov_Troitsky}
-- we thank Magnus Goffeng for pointing this to us:

\begin{thm}[Pavilov, Troisky, \cite{Pavlov_Troitsky}]
Let $\varphi:\Delta\to X$ be a continuous surjection where $\Delta$ is a compact open subset of $X$. There exists a transfer operator
$L:C(\Delta)\to C(X)$ with a strictly positive potential  $\varrho:\Delta\to (0,+\infty)$
if and only if $\varphi$ is an open map with $\sup_{x\in X} |\varphi^{-1}(x)|<\infty$.
\end{thm}
\begin{proof}
Under our assumptions the endomorphism $\alpha:C(X)\to C(\Delta)$, given by  composition with $\varphi$, is a unital monomorphism -- an inclusion.
Conditional expectations $E$ for the inclusion $\alpha$ are in bijective correspondence with transfer operators
$L$ for $\varphi$, given by  $E=\alpha \circ L$. Thus the assertion follows from \cite[Theorem 1.1]{Pavlov_Troitsky}
(in fact the `if part' follows from the proof of  \cite[Theorem 4.3]{Pavlov_Troitsky}).
\end{proof}
%Note that it is necessarily bounded as we have
We fix a transfer operator $L$ of the form \eqref{eq:transfer operator}. In general,  $\varrho$  has the following  properties.
\begin{prop}\label{prop:properties_or_rho}
The potential $\varrho$ is upper semi-continuous, and so $\varrho$ is continuous at every point in $\varrho^{-1}(0)$.
If $x_0\in \Delta_{\pos}=\Delta\setminus  \varrho^{-1}(0)$, then the following are equivalent:
%\begin{enumerate}
%\item\label{it:properties_or_rho1} $\p$ is locally open at  $x_0$, that is $\p(x_0)$ is in the interior of  $\p(U)$ for every  open   $U$ containing $x_0$;

%\item \label{it:properties_or_rho2} the following are equivalent:
\begin{enumerate}
\item\label{it:properties_or_rho1} $\varrho$ is continuous at $x_0$,

\item \label{it:properties_or_rho2}$\p$ is locally injective at $x_0$, i.e. there is open $U\subseteq \Delta$ with $x_0\in U$ such that $\p|_{U}:U\to X$ injective,
\item \label{it:properties_or_rho3}  $x_0$ is a local homeomorphism point for $\p$, i.e.  there is open $U$ with  $x_0\in U$ such that $\p:U\to \p(U)$ is a homeomorphism
 and $\p(U)$ is open in $X$.

\end{enumerate}
%\end{enumerate}
Moreover, $\varphi$ restricted to $\Delta_{\pos}$ is an open map.
 \end{prop}
%\begin{rem}
%It follows from Proposition \ref{prop:properties_or_rho}, that the set $\Delta\setminus  \varrho^{-1}(0)$ is open in $X$.
%Thus by passing to the restricted partial map $\varphi:\Delta\setminus  \varrho^{-1}(0) \to X$, we may assume without loss of generality that
%$\varrho>0$ on the whole of the domain of the partial map. Nevertheless, we do not always want to do that.
%\end{rem}
Before we get into the proof of Proposition \ref{prop:properties_or_rho}, we  first
%introduce a definition and
prove a couple of lemmas.

 \begin{lem}\label{lem:local_open_map}
Restriction of $\varphi$ to $\Delta_{\pos}$ is an open map.
%The map $\p$ is locally open at every $x_0\in \Delta\setminus  \varrho^{-1}(0)$, that is  for every  open   $U\subseteq \Delta$ containing $x_0$ the point  $\p(x_0)$ is in the interior of  $\p(U)$.
 \end{lem}
  \begin{proof}
	Every open set in  $\Delta_{\pos}$ is of the form $U\cap \Delta_{\pos}$ where $U\subseteq \Delta$ is open in $\Delta$.
	Let $y_0\in \varphi(U\cap \Delta_{\pos})$ so that $y_0=\varphi(x_0)$ for some $x_0\in U\cap \Delta_{\pos}$.
	 Take any continuous function $0 \leq a \leq 1 $  supported on  $U$  and such that $a(x_0)=1$. Then
  $
  \mu_{\p(x_0)}(a)\geq \varrho(x_0)>0$ and $\mu_y(a)=0$  for every $y \notin \p(U).
 $
  Since the map $X\ni y\to \mu_y(a)$ is  continuous  the set
	$V:=\{y\in X: \mu_y(a)>0\}
	$ is open in $X$. Clearly, $y_0=\p(x_0)\in V \subseteq \p(U\cap \Delta_{\pos})$.
 \end{proof}
\begin{lem}\label{continuity of measures lemma}
For any $x_0\in \Delta$ and $\varepsilon>0$ there is a neighbourhood $U_0$ of $x_0$ such that for any open $U\subseteq U_0$, with $x_0\in U$, there is a neighbourhood $V$  of $\varphi(x_0)$ such that
\begin{equation}\label{eq:continuity of measures lemma}
 \left|\sum_{x\in U\cap \varphi^{-1}(y)}{\varrho(x)} - \varrho(x_0)\right|<\varepsilon \qquad\text{for all } y\in V.
\end{equation}
\begin{proof}
Fix $\varepsilon >0$.
Since the measure $\mu_{\p(x_0)}$ is regular there is  a
neighbourhood  $U_1$ of $x_0$ such that $\mu_{\p(x_0)} (U_1)<\mu_{\p(x_0)} (\{x_0\})+\varepsilon$, which translates to
$$
\sum_{x\in U_1\cap \p^{-1}(\p(x_0))}{\varrho(x)} < \varrho(x_0) +\varepsilon.
$$
Let $U_0$ be any neighbourhood of $x_0$ such that $\overline{U_0}\subseteq U_1$. Now for any neighbourhood $U\subseteq U_0$ of $x_0$ take two continuous functions such that
$$
0\leq f_{1}, f_{2} \leq 1,\qquad
f_{1}(x)=
\begin{cases}
1,& x \in U\\
0, & x\notin U_0
\end{cases} ,\qquad
f_{2}(x)=
\begin{cases}
1,& x =x_0\\
0, & x\notin U
\end{cases}.
$$
Set $V=\{y:\mu_y(f_1)< \varrho(x_0)+\varepsilon \text{ and } \mu_y(f_2)> \varrho(x_0)-\varepsilon\}$.
Clearly, $\p(x_0)\in V$ and  for any $
y\in V$ we have %$O''(\al(x_0))$ is the neighborhood we need, since
$
 \varrho(x_0)-\varepsilon< \mu_y(f_2) \leq \sum_{x\in U\cap \p^{-1}(y)}{\varrho(x)}= \mu_y(U) \leq  \mu_x(f_1)< \varrho(x_0)+\varepsilon.
$
\end{proof}
\end{lem}
%\begin{cor}\label{continuity of rho in zero}
%The cocycle $\varrho$ is upper continuous. In particular,  $\varrho$ is continuous on the set $\varrho^{-1}(0)$.
%\end{cor}
%\begin{proof}
%\end{proof}
\begin{cor}\label{corollary to construct approximations of rho}
For any neighbourhood  $U$ of $x_0\in \Delta$ and any   $\varepsilon>0$ there exists a continuous function $0\leq h\leq1$ supported on $U$  such that $h(x)\equiv 1$ on an  neighbourhood of $x_0$ and
$\varrho(x_0)\leq  \max_{y\in X} \sum_{x\in\varphi^{-1}(y)}  \varrho(x)h(x) < \varrho(x_0) +\varepsilon.
$
\end{cor}
\begin{proof}
We may assume that $U\subseteq U_0$ where  $U_0$  is as in Lemma \ref{continuity of measures lemma} and then find  $V$ corresponding to $U$ in this lemma. Take any continuous function $0\leq h\leq 1$  supported on an open set contained with a boundary in $U \cap \p^{-1}(V)$ and such that $h(x)=1$ on an open neighbourhood of $x_0$. Then
 $$
 \varrho(x_0)\leq \sum_{x\in\p^{-1}(\p(x_0))} \varrho(x)h(x)\leq \max_{y\in X} \sum_{x\in \p^{-1}(y)}\varrho(x)h(x)=\|L(h)\|
 $$
 and
 $
\|L(h)\|=  \max_{y\in V} \sum_{x\in\p^{-1}(y)}  \varrho(x)h(x) \leq \max_{y\in V} \sum_{x\in U}\varrho(x)< \varrho(x_0) +\varepsilon
 $.
 \end{proof}

  \begin{proof}[Proof of Proposition \ref{prop:properties_or_rho}]
	Let $x_0\in X$ and $\varepsilon > 0$. Let $ U$ and $V$ be open sets as in Lemma \ref{continuity of measures lemma}. Then $U \cap \p^{-1}(V)$
	is an open   neighbourhood of $x_0$, and for any $x\in U \cap \p^{-1}(V)$  we have
$
\varrho(x)\leq \sum_{y\in U\cap\varphi^{-1}(\varphi(x))}{\varrho(y)}<\varrho(x_0)+\varepsilon.
$
Hence $\varrho$ is upper continuous at $x_0$.

Now let $x_0\in \Delta\setminus  \varrho^{-1}(0)$.

\ref{it:properties_or_rho1}$\Rightarrow$\ref{it:properties_or_rho2}.
 Suppose that $\varrho$ is lower continuous at $x_0$.
 Then for any  $\varepsilon < \varrho(x_0)/3$  there is  a neighbourhood $U$  of $x_0$ such that
\begin{equation}\label{from below boundary}
\varrho(x)> \varrho(x_0)- \varepsilon >0 \qquad \textrm{ for all } x\in U.
\end{equation}
By Lemma \ref{continuity of measures lemma} we may assume that there is an open neighbourhood $V$ of $\varphi(x_0)$ such that \eqref{eq:continuity of measures lemma}
holds. Then for $W:=U\cap\p^{-1}(V)$ we get
$$
\varrho(x_0)-\varepsilon < \mu_{\p(x)}(W)<\varrho(x_0)+\varepsilon\quad \text{ for all }x\in W.
$$
We claim, that $\p$ is injective on $W$. Indeed,  assume on the contrary that $W$ contains two distinct points  $x_1$, $x_2$ such that $\p(x_1)=\p(x_2)$. Then by \eqref{from below boundary} we get
$$
\varrho(x_0)+\varepsilon > \mu_{\p(x_1)}(W) \geq \varrho(x_1)+\varrho(x_2)> 2 (\varrho(x_0)- \varepsilon).
$$
which contradicts $\varepsilon < \varrho(x_0)/3$.

\ref{it:properties_or_rho2}$\Rightarrow$\ref{it:properties_or_rho3}.
This follows from Lemma \ref{lem:local_open_map}.

\ref{it:properties_or_rho3}$\Rightarrow$\ref{it:properties_or_rho1}. 	
	Suppose that $x_0$ is a local homeomorphism point, and let $U$ be a neighbourhood  of $x_0$ such that $\p:U\to \p(U)$ is a homeomorphism.
	Let $\varepsilon >0$. 	By Lemma \ref{continuity of measures lemma} we may assume that there is a neighbourhood $V$ of $\p(x_0)$ such that
	\eqref{eq:continuity of measures lemma} holds. But for any $x$ in $U\cap \p^{-1}(V)$  we have  have $\p^{-1}(\p(x))\cap U=\{x\}$ and thus
$$
\varrho(x_0)-\varepsilon < \varrho(x)=\mu_{\p(x)}(\{x\})=\mu_{\p(x)}(U)<\varrho(x_0)+\varepsilon.
$$
 Hence $\varrho$ is continuous at $x_0$.
   \end{proof}

	\section{Covariant representations and regular points}\label{sec:covariant_representations}

Throughout the paper,	we fix a  transfer operator $L:C_0(\Delta)\to C_0(X)$ of the form \eqref{eq:transfer operator} where $\p:\Delta\to X$ is a partial map and
	 $\varrho:\Delta\to [0,\infty)$ is the associated potential. We write $A:=C_0(X)$ and $I:=C_0(\Delta)$, and let $\alpha:C_0(X)\to C_b(\Delta)$ be  given by $\al(a)= a \circ \p$. %We put $I:=C_0(\Delta)$.
\begin{defn}\label{defn:representations}
A \emph{representation of the transfer operator} $L$ is a pair $(\pi, T)$ where $\pi:A\to B(H)$ is a non-degenerate representation
and $T\in B(H)$ satisfies
\begin{equation}\label{equ:main_relation}
\pi(L(a))=T^* \pi(a) T,\qquad a\in I=C_0(\Delta).
\end{equation}
We say that  $(\pi, T)$ is \emph{faithful} if $\pi$ is faithful.  We denote by
$$C^*(\pi,T):=C^*(\pi(A)\cup \pi(I)T )
$$
 the $C^*$-algebra generated by $\pi(A)\cup \pi(I)T$.
\end{defn}
\begin{rem}\label{rem:representation_of_L}
Without loss of generality, we could additionally assume in  Definition \ref{defn:representations} that $TH \subseteq \overline{\pi(I)H}$  (as composing $T$ with the projection onto $\overline{\pi(I)H}$ does not affect \eqref{equ:main_relation} and the $C^*$-algebra $C^*(\pi,T)$).
Assuming this we have $\|T\|\leq  \|L\|^{\frac{1}{2}}$, with the equality when $\pi$ is faithful. Indeed,
since $L:I\to A$ is positive, we have $\|L\|=\lim_{\lambda}\|L(\mu_\lambda)\|$ for an approximate unit $\{\mu_\lambda\}$ in $I$, see for instance,
\cite[Lemma 2.1]{kwa_Exel}.
Hence
$$
\|T\|^2=\|T^*T\|=\lim_{\lambda}\|T^*\pi(\mu_{\lambda})T\|=\lim_{\lambda}\|\pi(L(\mu_\lambda))\|\leq \lim_{\lambda}\|L(\mu_\lambda)\|=\|L\|
$$
and the inequality is equality when $\pi$ is injective. However, in what follows, we will not assume that $TH\subseteq \overline{\pi(I)H}$, as we will be mainly concerned with operators of the form $\pi(a)T$, for $a\in I$, and then we always have $\|\pi(a)T\|\leq \|a\| \|L\|^{\frac{1}{2}}$.
\end{rem}
\begin{rem}
The $C^*$-algebra $C^*(\pi,T)$ is not affected if we replace $\Delta$ by any open  set $U$ such that $\Delta_{\pos}\subseteq U\subseteq \Delta$,
as then  $\pi(C_0(\Delta))T=\pi(C_0(U))T$. Indeed,
$
\|\pi(a)T\|^2=\|\pi(L(a^*a))\|$ and $\|L(a^*a)\|=\sup_{y\in X}\sum_{x\in \varphi^{-1}(y)\cap \Delta_{\pos}} |a|^2(x)\varrho(x)$,
so the norm of $\pi(a)T$ depends only on values of $a$ on $\Delta_{\pos}$. In particular, we may always assume that
%\begin{equation}\label{eq:normalizing_condition}
$
\Delta=\varphi^{-1}(\varphi(\Delta_{\pos})),
$ %\end{equation}
as the set $\varphi^{-1}(\varphi(\Delta_{\pos}))$ is open because the map  $\varphi:\Delta_{\pos}\to X$ is open.
% Hence  $\varphi^{-1}(\varphi(\Delta_{\pos}))\subseteq \Delta$ is an open subset of $\Delta$ containing $\Delta_{\pos}$.
\end{rem}

\begin{lem}\label{lem:automatic_commutation_relation}
Let $(\pi, T)$ be a representation of $L$. We  have the following commutation relations
$$
\pi(b)T\pi(a)=\pi(b\alpha(a)) T, \qquad a\in A, b\in I.
$$
 If in addition $TH\subseteq \overline{\pi(I)H}$ and $\varphi$ is proper, then  $T\pi(a)=\pi(\alpha(a)) T$, $a\in A$.
%In particular, $\|aT\|^2\leq \|L\| \|a\|^2$ for all $a\in I$.
\end{lem}
\begin{proof} %\ref{enu:automatic_commutation_relation1}.
Putting $c:=\pi(b)T\pi(a)$ and $d:=\pi(b\alpha(a)) T$ one sees, that   each of the  expressions
$
c^*d$, $d^*d$, $c^*c$, $d^*c$ is equal to $\pi\big(L(\alpha(a^*)b^*b\alpha(a))\big)$.
Thus using the $C^*$-equality we get
$
\|c -d\|^2=\|\big(c^* -d^*)\big(c -d)\|=\|c^*d+d^*d+c^*c-d^*c\|= 0.
$

If $\varphi$ is proper, then $\alpha$ takes values in $I=C_0(\Delta)$ (rather than in $M(I)=C_b(\Delta)$) and hence we  may put $c:=T\pi(a)$ and $d:=\pi(\alpha(a)) T$ in the calculations  above. Then all the terms $
c^*d, d^*d, c^*c, d^*c
$  are equal to $\pi(a^*)TT^*\pi(a)$. For instance, if $\{\mu_\lambda\}$ is  an approximate unit  in $I$, then
\begin{align*}
c^*d&=\pi(a^*)T^*\pi(\alpha(a))T=s\text{-}\lim_{\lambda}\pi(a^*)T^*\pi(\mu_\lambda)\pi(\alpha(a))T
\\
&=s\text{-}\lim_{\lambda}\pi(a^*)\pi\big(L(\mu_\lambda)\alpha(a)\big)=s\text{-}\lim_{\lambda}\pi(a^*)\pi(L(\mu_\lambda))\pi(a)=\pi(a^*)TT^*\pi(a).
\end{align*}
Here $s\text{-}\lim$ stands for a limit in strong operator topology.
%Now, %\ref{enu:automatic_commutation_relation1}.
%since $L:I\to A$ is positive, we have $\|L\|=\lim_{\lambda}\|L(\mu_\lambda)\|$ for an approximate unit $\{\mu_\lambda\}$ in $I$, see for instance,
%\cite[Lemma 2.1]{kwa_Exel}.
%Hence
%$$
%\|T\|^2=\|T^*T\|=\lim_{\lambda}\|T^*\pi(\mu_{\lambda})T\|=\lim_{\lambda}\|\pi(L(\mu_\lambda))\|\leq \lim_{\lambda}\|L(\mu_\lambda)\|=\|L\|
%$$
%and the inequality is equality when $\pi$ is injective.
\end{proof}
%\begin{rem}
%For transfer operators $L:A\to A$ for full endomorphisms $\alpha:A\to A$,
%\end{rem}
	\begin{cor}\label{cor:algebra_and_ideal}
	If $(\pi, T)$ is a representation of $L$, then
	$$
	\overline{\pi(I)TT^*\pi(I)}=\clsp\{\pi(a)TT^*\pi(b): a,b \in I\}
	$$ is a $C^*$-algebra, and so
	$\pi(A)\cap \overline{\pi(I)TT^*\pi(I)}$ is an ideal in $\pi(A)$.
	\end{cor}
	\begin{proof}
	By Lemma \ref{lem:automatic_commutation_relation},
	$	\pi(a)TT^*\pi(b) \cdot  \pi(c)TT^*\pi(d)=\pi(a)TT^*\pi(\alpha(L(bc))d)$  for $a,b,c,d\in I$.
	Thus $\text{span}\{\pi(a)TT^*\pi(b): a,b \in I\}$ is a $*$-algebra.
	\end{proof}
	\begin{rem}
	In view of  Lemma \ref{lem:automatic_commutation_relation}, we have $\overline{\pi(I)TT^*\pi(I)}=\overline{\pi(I)T\pi(A)T^*\pi(I)}$,
	and if $TH\subseteq \pi(I)H$ and 	$\varphi$ is proper, then $\overline{\pi(I)TT^*\pi(I)}=\overline{\pi(A)TT^*\pi(A)}$.
	\end{rem}
	The spectrum of the ideal in Corollary \ref{cor:algebra_and_ideal} is  related to the set of regular points that we define as follows.
	\begin{defn}\label{defn:regular_points}
	The set of \emph{regular points} for $\varrho$ is
	$$
	\Delta_{\reg}:=\{x\in \Delta: \varrho(x)>0 \text{ and } \varrho \text{ is continuous at }x\}.
	$$
	Clearly,  $\Delta_{\reg}$ is an open set, and by Proposition \ref{prop:properties_or_rho}, a point  $x\in \Delta$  is regular if and only if $\varrho(x)>0$ and $x$ is  a local homeomorphism point
for $\varphi$.
	\end{defn}
\begin{rem}
We have a hierarchy of sets $\Delta_{\reg}\subseteq \Delta_{\pos} \subseteq \Delta $ where $\Delta_{\pos}$  need not be open nor closed in $X$.
The map $\varphi$ is open on $\Delta_{\pos}$ and in addition locally injective on $\Delta_{\reg}$.
\end{rem}

	\begin{prop}\label{prop:katsuras_ideal}
Let	$(\pi, T)$ be a  faithful representation of $L$. Then
$$
\pi(A)\cap \overline{\pi(I)TT^*\pi(I)}\subseteq \pi(C_0(\Delta_{\reg})).
$$
\end{prop}
\begin{proof}
To lighten the notation we will suppress $\pi$ and we will write $A\subseteq B(H)$.
Let us fix $a\in A$ such that $a\not\in C_0(\Delta_{\reg})$. That is, there is $x_0 \in X\setminus \Delta_{\reg}\neq 0$
with $a(x_0)\neq 0$. We need to show that $a\not\in \overline{ITT^*I}$ and to this end it suffices to show that  for any $a_i,b_i\in I$, $i=1,...,N\in\N$, we have
\begin{equation}\label{*condition}
\|a- \sum_{i=1}^{N}a_iTT^* b_i\|\geq|a(x_0)|.
\end{equation}
We first show a weaker inequality, which holds for an arbitrary $y\in \Delta$ though,
\begin{equation}\label{*condition2}
\|a- \sum_{i=1}^{N}a_iTT^* b_i\| \geq |a(y)| - \sqrt{\varrho(y)}\sum_{i=1}^{N}\|a_i\| \|b_i\|.
\end{equation}
Let $\varepsilon>0$ and put $U:=\{x\in \Delta:|a(x)-a(y)|<\varepsilon \}$.
By Corollary \ref{corollary to construct approximations of rho}  there is a continuous function $0\leq h\leq 1$
 supported on $U$  such that $h(x)=1$ on an open neighbourhood of $y$  and
$$
\varrho(y)\leq  \|L(h^2)\| < \varrho(y) +\varepsilon .
$$
Thus for any $b\in C_0(X)$ one has
$$
\|hb^*T\|^2=\|T^*bh\|^2=\|L(|b|^2h^2)\|\leq \|b\|^2 \|L(h^2)\| \leq  \|b\|^2 (\varrho(y) +\varepsilon).
$$
Using this we get
\begin{align*}
\|a- \sum_{i=1}^{N}a_iTT^* b_i\|& \geq  \|h \big(a- \sum_{i=1}^{N}a_iTT^* b_i\big)h\|= \| a h^2 - \sum_{i=1}^{N}ha_iT  T^*b_i h\|
\\
&\geq \| a h^2\|- \sum_{i=1}^{N}\|ha_iT\| \|T^*b_i h\| \\
& \geq |a(y)|  -  \sqrt{(\varrho(y) +\varepsilon)} \sum_{i=1}^{N}\|a_i\| \|b_i\|.
\end{align*}
Passing with $\varepsilon$ to zero, we get \eqref{*condition2}.
Now we consider  two cases.

I). Suppose first that for each $\delta>0$ every neighbourhood of $x_0$ contains a point $x$ with $\varrho(x)<\delta$.
Equivalently, there is a net $\{x_n\}\subseteq \Delta$ such that  $x_n\to  x_0$ and  $\varrho(x_n)\to 0$. Applying  \eqref{*condition2}
to  $y=x_n$ we have
$$
\|a- \sum_{i=1}^{N}a_iTT^* b_i\| \geq |a(x_n)| -  \sqrt{\varrho(x_n)} \sum_{i=1}^{N}\|a_i\| \|b_i\|,
$$
which by passing to the limit,  gives  \eqref{*condition}.

II). Finally, suppose that
  there is $\delta>0$ and an open neighbourhood $U$ of $x_0$
 such that
$$
\inf_{x\in U}\varrho(x)\geq \delta >0.
$$
Clearly it is enough to consider the case when $a(x_0)\neq 0$. Let  $\varepsilon >0$. We may assume that $U\subseteq \{x\in \Delta:|a(x)-a(x_0)|<\varepsilon \}$.
Also, since $x_0\not\in \Delta_{\reg}$,  $\varrho$ is not continuous at $x_0$. Therefore, by Proposition \ref{prop:properties_or_rho}, there exist two distinct points $x_1$, $x_2$ in $U$ such that $\p(x_1)=\p(x_2)$. Let $U_1$, $U_2\subseteq U$ be two open disjoint sets with $x_1\in U_1$ and $x_2\in U_2$.
By Corollary \ref{corollary to construct approximations of rho}, for each $i=1,2$, there are  continuous functions $0\leq h_i\leq 1$
supported on $U_i$  such that $h_i(x_i)=1$  and
$$
\varrho(x_i)\leq  \|L(h_i)\| < \varrho(x_i) +\varepsilon, \qquad {\rm\text{and}} \qquad
\varrho(x_i)\leq  \|L(h_i^2)\| < \varrho(x_i) +\varepsilon .
$$
Put
$
h:=h_1 -  \frac{\varrho(x_1)}{\varrho(x_2)} h_2.
$
Using that $h$ is supported on $U$ we get
$$
\|ahT\|\geq \|hT\|(|a(x_0)| -\varepsilon).
$$
%
%\begin{align*}
%\|ahT\|^2=\|L(|a|^2h^2)\|\leq
%\|L(h^2)\| (|a(x_0)| +\varepsilon)^2
%= \|hT\|^2 (|a(x_0)| +\varepsilon)^2.
%\end{align*}
%Thus $\|ahT\|\leq \|hT\|(|a(x_0)| +\varepsilon)$.
On the other hand,
$
\|hT\|^2=\|L(h^2)\|=\|L(h_1^2) +\left(\frac{\varrho(x_1)}{\varrho(x_2)}\right)^2 L(h_2^2)\|\geq \|L(h_1^2)\|\geq \varrho(x_1)\geq \delta.
$
Moreover, for any $b\in C_0(X)$  we have
\begin{align*}
\|T^*b hT\|&=\|L(hb)\|\leq  \|L(h)\|\cdot \|b\| =  \left\|L(h_1) -   \frac{\varrho(x_1)}{\varrho(x_2)} L(h_2)\right \|\cdot \|b\|
\\
&\leq \left((\varrho(x_1) +\varepsilon)
- \frac{\varrho(x_1)}{\varrho(x_2)} \varrho(x_2)\right) \cdot \|b\| = \varepsilon \cdot \|b\|.
\end{align*}
Using all this, we get
\begin{align*}
\|a- \sum_{i=1}^{N}a_iTT^* b_i\| & \geq \frac{ \| \Big(a- \sum_{i=1}^{N}a_iTT^* b_i\Big) (hT)\|}{\|hT\|}
\\
&\geq \frac{\| a h T\|}{\|hT\|}- \sum_{i=1}^{N}\frac{ \|a_iT\|\cdot \| T^*b_i  hT\|}{\|hT\|}
\\
& \geq (|a(x_0)| -\varepsilon)- \varepsilon \sum_{i=1}^{N}\frac{ \|a_iT\|\cdot \| b_i\|}{ \sqrt{\delta}}.
\end{align*}
Passing with $\varepsilon$ to zero, we get \eqref{*condition}.
\end{proof}

%\textcolor{red}{Now we introduce a class of representations playing the key role in the crossed product construction.}

\begin{defn}\label{defn:representations_covariant}
We say that a representation $(\pi, T)$ of $L$ is \emph{covariant} if
$$
\pi(C_0(\Delta_{\reg}))\subseteq  \overline{\pi(I)TT^*\pi(I)}.
$$
Thus a faithful  representation $(\pi, T)$ is covariant iff  $\pi(C_0(\Delta_{\reg}))=\overline{\pi(I)TT^*\pi(I)}$.
\end{defn}
\begin{rem}\label{rem:Toeplitz representation}
Every transfer operator admits a faithful covariant representation (see  Example~\ref{ex:orbit_representation}).
If  $\Delta_{\reg}$ is non-empty, then there are faithful representations that are not covariant. Indeed, if $(\pi,T)$ is any representation
of $L$ on a Hilbert space $H$, then putting $\widetilde{H}:=H\otimes \ell^2(\N)$, $\widetilde{\pi}:=\pi\otimes \id$ and $\widetilde{T}:=T\otimes U$
where $U$ is the unilateral shift on  $\ell^2(\N)$, we get a representation $(\widetilde{\pi}, \widetilde{T})$ of $L$ with
$
\widetilde{\pi}(A)\cap\overline{\widetilde{\pi}(I)\widetilde{T}\widetilde{T}^*\widetilde{\pi}(I)}= 0,
$
because $UU^*$ is a non-trivial projection.
\end{rem}
\subsection{Characterisations of covariant representations}
Let $K$  be a compact subset of $\Delta_{\reg}$. Then we may find a finite cover $\{U_i\}_{i=1}^n$ of $K$ such that
 $\bigcup_{i=1}^n U_i$ is contained in a  compact subset of  $\Delta_{\reg}$ and $\varphi|_{U_i}$ is injective for every $i=1,...,n$. Take a partition of unity $\{v_i\}_{i=1}^n\subseteq C_0(\Delta_{\reg})$ on $K$
 subordinated to $\{U_i\}_{i=1}^n$. Then
\begin{equation}\label{eq:quasi-basis}
u_i^K:=\sqrt{\frac{v_i}{\varrho}},\qquad  i=1,...,n,
\end{equation}
 are well defined functions in $C_c(\Delta_{\reg})$ because $\varrho$ is bounded away from zero on $\bigcup_{i=1}^n U_i$. We will use these functions
to characterise covariant representations.
\begin{prop}\label{prop:characterization_of_covariance} Let  $(\pi,T)$ be a
representation   of $L$. The following are equivalent:
\begin{enumerate}
\item\label{it:characterization_of_covariance1} $(\pi,T)$  is covariant.

\item\label{it:characterization_of_covariance2.5} for every $a\in C_c(\Delta_{\reg})$ supported on a set where $\varphi$ is injective we have
$
\pi(a)TT^*\pi(u)=\pi(a)
$ for some $u\in C_0(X)$.
\item\label{it:characterization_of_covariance3} for every $a\in C_c(\Delta_{\reg})$ supported on a set where $\varphi$ is injective we have
$
\pi(a)TH=\pi(a)H.
$
\item\label{it:characterization_of_covariance2}  For every  $a\in C_c(\Delta_{\reg})$ supported on $K$ we have
we have
\begin{equation}\label{eq:Cuntz_relation_K}
\pi(a) \sum_{i=1}^n \pi(u_i^{K}) TT^* \pi(u_i^{K})  =   \pi(a).
\end{equation}

%\item\label{it:characterization_of_covariance5} for every $x_0\in \Delta_{\reg}$  there is a neighbourhood $U\subseteq \Delta_{\reg}$ of $x_0$ such that for every
%$a,b\in C_0(U)$  we have
%$
%\pi(a)TT^*\pi(b)=\pi(\varrho ab)
%$, note that $\varrho a \in C_0(U)$.
\item\label{it:characterization_of_covariance4} for every $x_0\in \Delta_{\reg}$ and $\varepsilon >0$ there is a neighbourhood $U$ of $x_0$ such that for every
$a,b\in C_0(U)$ with $\|a\|, \|b\|\leq 1$ we have
$$
\|\pi(a)TT^*\pi(b)-\varrho(x_0)\pi(ab)\|< \varepsilon.
$$
\end{enumerate}
The above conditions hold whenever $\overline{\pi(I)TH}=H$ (which is equivalent to $\overline{\pi(A)TH}=H$ when $\varphi$ is proper).
\end{prop}
\begin{proof}
Clearly, \ref{it:characterization_of_covariance2.5} implies  \ref{it:characterization_of_covariance3}.
Since $C_c(\Delta_{\reg})$ is dense in $C_0(\Delta_{reg})$ and every element in $C_c(\Delta_{\reg})$ is
 a finite sum of functions supported on sets where $\varphi$ is injective,  we see that \ref{it:characterization_of_covariance2.5} also implies  \ref{it:characterization_of_covariance1}.
For converse implications, let $a\in C_c(\Delta_{\reg})$ have support  $K$ such that $\varphi|_K$ is injective and let $u\in C_c(\Delta_{\reg})$
 be such that $u|_K=(\varrho|_K)^{-1}$.
 For every $b\in C_0(\Delta)$ and $x\in \Delta_{\reg}$ we have $a(x)\alpha(L(ub))(x)=a(x)\sum_{t\in\varphi^{-1}(\varphi(x))}\varrho(t)u(t)b(t)=a(x)b(x)$.
Hence
%\begin{equation}\label{eq:technical_relation}
$$
\Big(\pi(a)TT^*\pi(u)\Big)\pi(b)T=\pi(a\alpha(L(ub)))T=\pi(a)\pi(b)T.
$$
%\end{equation}
Thus $\pi(a)TT^*\pi(u)=\pi(a)$  whenever  $\pi(a)$ is determined by its action on $\pi(I)TH$.
 Both \ref{it:characterization_of_covariance1} and \ref{it:characterization_of_covariance3} imply this.
Indeed, if \ref{it:characterization_of_covariance1} holds then $\pi(a)\in \overline{\pi(I)TT^*\pi(I)}$,
and if we assume  \ref{it:characterization_of_covariance3} we get
$$
\pi(a)H=\pi(a)TH=\pi(a)T\pi(A)H=\pi(a)\pi(\alpha(A))TH=\pi(a)\pi(I)TH.
$$
Hence \ref{it:characterization_of_covariance1}, \ref{it:characterization_of_covariance2.5}, \ref{it:characterization_of_covariance3} are equivalent and they follow from
the condition $\pi(I)TH=H$.
If $\varphi$ is proper, then $\pi(I)TH=\pi(A\alpha(A))TH=\pi(A)T\pi(A)H=\pi(A)TH$.

Since $C_c(\Delta_{\reg})$ is dense in $C_0(\Delta_{\reg})$,  \ref{it:characterization_of_covariance2} readily implies \ref{it:characterization_of_covariance1}.
Conversely, if we assume
\ref{it:characterization_of_covariance1}, then for every $a\in C_c(\Delta_{\reg})$ the operator  $\pi(a)\in \overline{\pi(I)TT^*\pi(I)}$  is determined by its action on $\pi(I)TH$. Moreover,
 for every $b\in I$ we have  $a \sum_{i=1}^n u_i^K \alpha(L(  u_i^K b))=a b$. Thus
\begin{align*}
\left(\pi(a) \sum_{i=1}^n \pi(u_i^K) TT^* \pi(u_i^K)\right) \pi(b) T&= \pi\left(a \sum_{i=1}^n u_i^K \alpha(L(  u_i^K b))\right) T.
\\
&=\pi(a) \pi(b) T.
\end{align*}
This implies
\eqref{eq:Cuntz_relation_K}. Hence \ref{it:characterization_of_covariance1}$\Rightarrow$\ref{it:characterization_of_covariance2}.

\ref{it:characterization_of_covariance3} $\Rightarrow$
\ref{it:characterization_of_covariance4}.
Let $U$ be a neighbourhood of $x_0\in \Delta_{\reg}$
such that $\varphi|_U$ is injective and $U\subseteq \{x\in \Delta_{\reg} : |\varrho(x)-\varrho(x_0)|<\varepsilon\}$.
Take any $a,b\in C_0(U)$ with $\|a\|, \|b\|\leq 1$.
Note that $\varrho a\in C_0(U)$ and  $\| \varrho ab  - \varrho(x_0) ab\|<\varepsilon$. The argument in the proof that \ref{it:characterization_of_covariance3}
implies \ref{it:characterization_of_covariance1}, shows that $\pi(a)TT^*\pi(b)= \pi(\varrho ab)$. Hence \ref{it:characterization_of_covariance4} holds.

%.............
%
%
%\begin{equation*}
%\label{eq:technical_relation}
%\Big(\pi(a)TT^*\pi(u)\Big)\pi(b)T=\pi(a\alpha(L(ub)))T=\pi(a)\pi(b)T.
%\end{equation*}
%
%...........
%
%
%Since $\varrho(x)<\varrho(x_0)+\varepsilon$ for every $x\in U$, we have $\|\pi(a)TT^*\pi(b)\pi(c)T-\varrho(x_0)\pi(ab)\pi(c)T\|<\varepsilon$.
%As $\pi(C_c(\Delta_{\reg}))\subseteq \pi(A)TT^*\pi(A)$ we get  $\|\pi(a)TT^*\pi(b)-\varrho(x_0)\pi(ab)\|<\varepsilon$.

\ref{it:characterization_of_covariance4} $\Rightarrow$ \ref{it:characterization_of_covariance3}.
Let   $a\in C_c(\Delta_{\reg})$  have support  $K$ such that $\varphi|_K$ is injective. Without loss of generality
we may assume that $\|a\|\leq 1$.
%Also to prove \ref{it:characterization_of_covariance3} we may assume that $K$ is contained in an arbitrary open set
% $U$ where $\varphi$ is injective, as then using a partition of unity on $K$ subordinate to a cover formed by such sets
By  \ref{it:characterization_of_covariance4} and compactness
of $K$ there is a partition of unity $\{u_i\}_{i=1}^n$ on $K$
subordinate to an open cover $\{U_{i}\}_{i=1}^n$ of $K$
such that for every $i=1,...,n$ there is a point $x_i\in U_i$
such that for every $b\in C_0(U_i)$, $\|b\|\leq 1$, we have
$\|\pi(u_ia)T^*T\pi(b)- \varrho(x_i)\pi(ab)\|< \varrho(x_i)/2$.
Clearly, $a$ satisfies \ref{it:characterization_of_covariance3}
iff each $u_i a$  satisfies \ref{it:characterization_of_covariance3}.
Hence we may assume that $a\in C_0(U)$ where $\varphi|_U$ is injective and
there is $x_0\in U$ such that
$$
\|\pi(a)T^*T\pi(b)- \varrho(x_0)\pi(ab)\|< \varrho(x_0)/2,
$$
for any $b\in C_0(U)$, $\|b\|\leq 1$.
Now let $\{\mu_\lambda\}$ be an approximate  unit in $C_0(U)$ and let  $P:=\text{s-}\lim \pi(\mu_\lambda)$
be the projection given by the strong limit. Then we have
$$
\| PT^*TP- \varrho(x_0)P\|\leq \varrho(x_0)/2.
$$
Thus $\|1/\varrho(x_0)PT^*TP- P\|\leq 1/2<1$ and therefore the operator
$1/\varrho(x_0)PT^*TP:PH\to PH$ is invertible.
In particular, $P_U TH = P_U H$ and this gives
$\pi(a) TH =\pi(a)P_U TH = \pi(a)P_U H =\pi(a)H$.
\end{proof}
\begin{rem}
Assume  $X=\Delta_{\reg}$. Equivalently, $\p:X\to X$ is a local homeomorphism and $\varrho>0$ is strictly positive. Then condition
\ref{it:characterization_of_covariance2} in Proposition \ref{prop:characterization_of_covariance} reduces to
$$
\sum_{i=1}^n \pi(u_i^X) TT^* \pi(u_i^X)  =1,
$$ which is the condition  identified by Exel and Vershik in \cite{exel_vershik}. Also
conditions in Proposition \ref{prop:characterization_of_covariance} are equivalent to the condition $\overline{\pi(A)TH}=H$,
which is called axiom (A3) in \cite{bar_kwa}.
\end{rem}
\begin{ex}[Orbit representation] \label{ex:orbit_representation}
There is a natural faithful covariant representation $(\pi_o, T_o)$ of $L$ on the Hilbert space $\ell^2(X)$. We will call it the \emph{orbit representation}.  Namely we define a faithful representation
$\pi_o:C_0(X) \to B(\ell^2(X))$  by
$$
(\pi_o(a) h)(x):= a(x) h(x),  \qquad a\in C_0(X), h\in \ell^2(X).
$$
%Hence  $A:=\pi_o(C(X))\cong C(X)$.
Let $\{\mathds{1}_{x}\}_{x\in X}$ be  the standard orthonormal basis of $\ell^2(X)$.
Since   $\sum_{x\in \varrho^{-1}(y)}\varrho(x)\leq \|L\|$, $y\in X$,   there is $T_o\in B(\ell^2(X))$
such that $T_o\mathds{1}_y:=\sum_{x\in\varphi^{-1}(y)}\sqrt{\varrho(x)}\mathds{1}_x$, $y\in X$.
Its  adjoint is given by
$T_o^*\mathds{1}_x=\sqrt{\varrho(x)}\mathds{1}_{\varphi(x)}$, for $x\in \Delta$, and $T_o^*\mathds{1}_x=0$ for $x\not\in \Delta$. Equivalently,
$$
(T_oh)(x)=\begin{cases}\sqrt{\varrho(x)}h(\varphi(x)), & x\in \Delta
\\
0 , & x\not\in \Delta,
\end{cases}
\qquad  (T_o^*h)(y)=\sum_{x\in\varphi^{-1}(y)} \sqrt{\varrho(x)}h(x),
$$
for  $h\in \ell^2(X)$.
Clearly $(\pi_o, T_o)$  is a faithul representation of $L$. To see that it is covariant we show condition \ref{it:characterization_of_covariance2.5} in Proposition \ref{prop:characterization_of_covariance}.
Let  $a\in C_c(\Delta_{\reg})$ be supported on a set $K$ such that $\varphi|_K$ is injective and let $u\in C_c(\Delta_{\reg})$ be such that $u|_K=(\varrho|_K)^{-1}$. Then
$
\left( \pi_o(a) T_oT_o^* \pi_o(u)h\right)(x)
=a(x)\sqrt{\varrho(x)}\left(\sum_{t\in\varphi^{-1}(\varphi(x))}\sqrt{\varrho(t)}u(t)h(t)\right)=a(x)h(x)
=\left(\pi_o(a) h\right)(x).
$
Hence $\pi_o(a)=\pi_o(a) T_oT_o^* \pi_o(u)$.

\end{ex}

\section{The crossed product}\label{sec:crossed_products}

%In this section we  construct the crossed product of $A$ by the transfer operator $L$ as a  universal $C^*$-algebra
%with explicit underlying algebraic relations described. We show that this construction is consistent with Exel-Royer crossed products \cite{er}
% and a number of other $C^*$-algebras defined as Cuntz-Pimnser algebras in  \cite{exel3, BroRaeVit, Kajiwara_Watatani0, Kajiwara_Watatani, Izumi_Kajiwara_Watatani}.

Recall that $L:C_0(\Delta)\to C_0(X)$ is  a transfer operator for a partial map $\varphi:\Delta\to X$,  $A=C_0(X)$, $I=C_0(\Delta)$, and $\alpha:A\to M(I)$ is given by composition with $\varphi$.
Let us consider  a universal $*$-algebra $\A(L)$ generated by $C_c(X)$ (viewed as a $*$-algebra) and an element $t$ subject to relations
\begin{equation}\label{eq:standard relations}
L(a)=tat^*, \qquad atb=a\alpha(b)t, \qquad \text{for all } a\in C_c(\Delta), b\in C_c(X),
\end{equation}
and for every compact $K\subseteq \Delta_{\reg}$ and  $a\in C_c(\Delta_{\reg})$ supported on $K$
\begin{equation}\label{eq:Cuntz_relation_K2}
a \sum_{i=1}^n u_i^K tt^* u_i^K  =   a
\end{equation}
where  $u_i^K$'s are given by \eqref{eq:quasi-basis}. Note that these relations are satisfied by operators coming from
covariant representations of $L$, see Lemma \ref{lem:automatic_commutation_relation} and Proposition \ref{prop:characterization_of_covariance}.
\begin{defn}
The \emph{algebraic crossed-product} $C_c(X)\rtimes_{\alg} L$ is the $*$-subalgebra of $\A(L)$ generated by $C_c(X)$ and $C_c(\Delta)t$.
\end{defn}
To describe the structure of $C_c(X)\rtimes_{\alg} L$ we need to iterate partial transfer operators. Let
$$
\Delta_n:=\varphi^{-n}(X), \qquad I_n:=C_0(\Delta_n), \qquad n\in \N.
$$
We put $\Delta_0:=X$ and $I_0:=A=C_0(X)$. So $\Delta_n$ is a natural domain for the partial map $\varphi^n$; the composition $\underbrace{\varphi\circ \dots \circ \varphi}_{\text{$n$ times}}$ makes sense on $\Delta_n$. Define $\alpha^n:A\to M(I_n)$ to be the partial endomorphism of $A$ given by composition with $\varphi^n:\Delta_n\to X$. Having the map $\varrho:\Delta\to [0,\infty)$,
 that defines $L$ via \eqref{eq:transfer operator}, for each $n\in\N$ we define $\varrho_n:\Delta_n\to [0,\infty)$ by
$$
\varrho_n(x):=\prod_{i=0}^{n-1}\varrho(\varphi^i(x)), \qquad x\in \Delta_n.
$$
We also put $\varrho_0\equiv 1$. % The maps $\{\varrho_n\}_{n=0}^\infty$ form a (partial) \emph{cocycle} generated by $\varrho$, cf. \cite{exel_renault}. 
Then the formula
$$
L^n(a)(y):=\sum_{x\in \varphi^{-n}(y)}\varrho_n(x) a(x),\qquad  a\in C_0(\Delta_n), y\in X,
$$
defines a transfer operator $L^n:I_n\to A$ for $\alpha^n:A\to M(I_n)$.
To describe the maps $L^n$ and $\alpha^n$ more algebraically, note that
\begin{equation}\label{eq:I_n_spanning_set}
I_n^0:=\spane\{a_1\alpha(a_2\alpha(...a_n)...):a_1,...,a_n\in C_c(\Delta)\}
\end{equation}
is a dense $*$-subalgebra of $I_n$ (we put $I_0^0:=C_c(X)$).
Thus $L^n$ and $\alpha^n$ are determined by the following formulas, for $a_1,...,a_n\in C_c(\Delta)$ and $a\in C_c(X)$:
\begin{align}
L^n\Big(a_1\alpha(a_2\alpha(...a_n)...)\Big)&=L(L(... L(L(a_1)a_2)a_3...)a_n)\label{eq:transfer_powers}
\\
a_1\alpha(a_2\alpha(...a_n)...)\cdot \alpha^n(a) &=a_1\alpha(a_2\alpha(...a_n\alpha (a)...).\label{eq:endomorphism_powers}
\end{align}
\begin{lem}\label{algebraic_structure}
The algebraic crossed-product is the following linear span
\begin{equation}\label{eq:spanning set}
C_c(X)\rtimes_{\alg} L=\spane\{at^n t^{*m}b: a\in I_n^0, b\in I_m^0, n,m\in \N_0\}.
\end{equation}
Moreover,
 for all $n,m,k,l\in \N_0$ and $a\in I_n^0,b\in I_m^0,c\in I_k^0,d\in I_l^0$, we have
\begin{equation}\label{eq:commutation relations}
(at^n t^{*m}b)\cdot  (ct^k t^{*l}d)
=\begin{cases}
at^n t^{*m-k+l}\alpha^l(L^k(bc)) d & m\geq k,
\\
a\alpha^n(L^m(bc))t^{k-m+n} t^{*l} d & m< k.
\end{cases}
\end{equation}
\end{lem}
\begin{proof} Using  \eqref{eq:transfer_powers} and \eqref{eq:endomorphism_powers} we get that \eqref{eq:standard relations} generalizes to
\begin{equation}\label{eq:standard_relations_iterated}
L^n(a)=t^{*n}at^{n}, \qquad at^nb=a\alpha^n(b)t^n, \qquad \text{for all } a\in I_n^0, b\in A, n\in \N.
\end{equation}
For instance, for $n=2$, and $a_1,a_2\in I$, we have
$$
L^2\Big(a_1\alpha(a_2)\Big)\stackrel{\eqref{eq:transfer_powers}}{=}L(L(a_1) a_2)\stackrel{\eqref{eq:standard relations}}{=}t^* (t^* a_1 ta_2)t\stackrel{\eqref{eq:standard relations}}{=}t^{*2}a_1\alpha(a_2) t^{2},
$$
$$
\Big(a_1\alpha(a_2)\Big)t^2b\stackrel{\eqref{eq:standard relations}}{=}a_1 t a_2tb\stackrel{\eqref{eq:standard relations}}{=} a_1t a_2\alpha(b)t=a_1 \alpha(a_2\alpha(b))t^2\stackrel{\eqref{eq:endomorphism_powers}}{=}  \Big(a_1\alpha(a_2)\Big)\cdot \alpha^2(b)t^2 .
$$
Using \eqref{eq:standard_relations_iterated} one readily gets \eqref{eq:commutation relations}. In turn
\eqref{eq:commutation relations} implies that the self-adjoint linear space $\spane\{at^n t^{*m}b: a\in I_n^0, b\in I_m^0, n,m\in \N_0\}$
is closed under multiplication. Hence it is a $*$-algebra, and clearly it is generated by $I_0^0\cup I_1^0$. This proves \eqref{eq:spanning set}.
\end{proof}
By universality every covariant representation $(\pi,T)$ of $L$ induces (uniquely)
a representation $\pi\rtimes T$ of the $*$-algebra $C_c(X)\rtimes_{\alg} L$ where $\pi\rtimes T(a)=\pi(a)$, $a\in C_c(X)$, and
$\pi\rtimes T(at)=aT$ for $a\in C_c(\Delta)$. Namely, $\pi\rtimes T (\sum_{i=1}^n a_it^{n_i} t^{*m_i}b_i)=\sum_{i=1}^n a_i T^{n_i} T^{*m_i}b_i$
 for $a_i\in I_{n_i}^0$, $b_i\in I_{m_i}^0$, $i=1,...,n$.
We put
$$
\|x\|_{\max}:=\sup\{ \|\pi\rtimes T(x)\|: (\pi,T)\text{ is a covariant representation for }L\}.
$$
It is easily verified that $\|\cdot \|_{\max}$ is a $C^*$-seminorm (a submultiplicative seminorm satisfying the $C^*$-equality). It is finite because
$
\|\sum_{i=1}^n a_it^{n_i} t^{*m_i}b_i \|_{\max}\leq  \sum_{i=1}^n \|a_i\|  \|b_i\|  (\|L^{n_i}\| \|L^{m_i}\|)^{\frac{1}{2}},
$
cf. Remark \ref{rem:representation_of_L}. Restriction $\|\cdot \|_{\max}$  to $C_c(X)$
coincides with the unique  $C^*$-norm on $A$, because there exists  a faithful covariant representation, see Example \ref{ex:orbit_representation}.
In other words, the (self-adjoint and two-sided) ideal
$$
\NN:=\{x\in C_c(X)\rtimes_{\alg} L: \|x\|_{\max}=0\}
$$
intersects $C_c(X)$ trivially.
\begin{defn}\label{def:crossed_product}
The \emph{crossed product} of $A$ by the transfer operator $L$ is the $C^*$-algebra $A\rtimes L$ obtained by the Hausdorff completion
of $C_c(X)\rtimes_{\alg} L$ in $\|\cdot \|_{\max}$:
$$
A\rtimes L=\overline{C_c(X)\rtimes_{\alg} L/\NN}^{\|\cdot \|_{\max}}.
$$
\end{defn}
\begin{rem}\label{rem:crossed_product_span}
Since   $C_c(X)\cap \NN=\{0\}$, we may and we will treat $C_c(X)$ as a $*$-subalgebra of $A\rtimes L$. The closure of $C_c(X)$ in  $A\rtimes L$ will be identified with $A$. We will also abuse the notation and write
$at^n$, $a\in I_n^0$,    for  their images in $A\rtimes L$.
In fact we extend this notation to any $a\in I_n=C_0(\Delta_n)$ by writing $at^n$ for the
limit in   $A\rtimes L$ of a sequence $a_nt^n$ where $\{a_n\}_{n=1}^\infty \subseteq I_n^0$ converges uniformly to $a$.
%($\{a_nt\}_{n=1}^\infty$ converges and its limit does not depend on the choice of $\{a_n\}_{n=1}^\infty$).
So by Lemma \ref{algebraic_structure} we have
$$
A\rtimes L=\clsp\{at^n t^{*m}b: a\in I_n,b\in I_m, n,m\in \N_0\}.
$$
\end{rem}
%There is a bijective correspondence between covariant representations of $L$ and non-degenerate representations
%of $A\rtimes L$. This property characterises the crossed product up to isomorphism:
\begin{prop}\label{prop:universal_crossed_product}
Assume that $A\rtimes L\subseteq B(H)$ is represented in a faithful and non-degenerate way on a Hilbert space $H$. The crossed product $A\rtimes L$ is the universal $C^*$-algebra for covariant representations of $L$
:
%and this universal property  determines $A\rtimes L$ uniquely up to an natural isomorphism.
%every transfer operator $L$ there exists a universal $C^*$-algebra $A\rtimes L$ for covariant representations of $L$.
%Namely,
\begin{enumerate}
\item \label{prop:universal_crossed_product1} $A\rtimes L$ contains $A$ as a $C^*$-subalgebra, and is generated by $A$ and  $It$ for $t\in B(H)$ such that
$L(a)=t^*a t$, $a\in I$ and  $C_0(\Delta_{\reg})\subseteq \overline{Itt^*I}$.
\item\label{prop:universal_crossed_product2} Every covariant representation $(\pi,T)$ of $L$  induces
a representation $\pi\rtimes T$ of $ A\rtimes L$ where $\pi\rtimes T(a)=\pi(a)$, $a\in A$, and
$\pi\rtimes T(at)=\pi(a)T$, $a\in I$.
%\item\label{prop:universal_crossed_product3}
\end{enumerate}
 Every $C^*$-algebra possessing properties \ref{prop:universal_crossed_product1}, \ref{prop:universal_crossed_product2} is  isomorphic to $A\rtimes L$ by an isomorphism which is identity on $A$.
\end{prop}
\begin{proof}
\ref{prop:universal_crossed_product1} and \ref{prop:universal_crossed_product2} follow by construction. To see the  last part, assume that  $C=C^*(A\cup Is)\subseteq B(K)$
is a $C^*$-algebra, represented on a Hilbert space $K$, that satisfy analogues of  \ref{prop:universal_crossed_product1}, \ref{prop:universal_crossed_product2}.
Then \ref{prop:universal_crossed_product2} for $A\rtimes L$ and $C$   give    $*$-epimorphisms $\Psi:A\rtimes L\to C$ and $\Phi:C\to A\rtimes L$ which clearly are inverse to each other.
\end{proof}
\begin{rem}
Proposition \ref{prop:universal_crossed_product} shows that $A\rtimes L$ depends only on  $L:I\to A$, or equivalently on $\varrho:\Delta\to X$
(it depends only on $\varrho$ up to continuous factors, see Corollary \ref{cor:independence} below).
\end{rem}

\subsection{Cuntz-Pimsner picture and other constructions}
Let $L:I\to A$ be a transfer operator for the partial endomorphism $\alpha:A\to M(I)$.
The $C^*$-correspondence $M_L$ associated to $L$, cf.  \cite{exel3, er},    is a Hausdorff completion of the $A$-bimodule $I$ where
$a\cdot  \xi \cdot b= ax\alpha(b)$, for $\xi\in I$, $a,b\in A$, in the $A$-valued pre-inner product given by
$
\langle \xi, \eta\rangle_A:=L(\xi^*\eta)
$, $\xi,\eta\in I$. A representation of $M_L$ is a pair $(\pi,\psi)$ where $\pi:A\to B(H)$ is a non-degenerate representation
 and $\psi:M_L\to B(H)$ is a (necessarily linear) map such that $\pi(a)\psi(\xi)\pi(b)=\psi(a\xi b)$ and $\psi(\xi)^*\psi(\eta)=\pi( \langle \xi, \eta\rangle_A)$
for $a,b \in A$, $\xi,\eta\in M_L$.
\begin{lem}\label{lem:representations_correspondence}
Every representation $(\pi, \psi)$ of $M_L$ comes from a representation $(\pi, T)$ of $L$ in the sense that
$
\psi(q(\xi))=\pi(\xi)T$, $\xi\in I$,
where $q:I\to M_L$ is the canonical quotient map. This gives a bijective correspondence between  representations
$(\pi, \psi)$ of $M_L$  and representations $(\pi, T)$ of $L$ satisfying $TH \subseteq \overline{\pi(I)H}$.
\end{lem}
\begin{proof} If  $(\pi,T)$ is a representation of $L$, then $
\psi(q(\xi)):=\pi(\xi)T$, $\xi\in I$, is well defined because $\|\pi(\xi)T\|^2=\|\pi(L(\xi^*\xi))\|\leq \|L(\xi^*\xi)\|=\|q(\xi)\|$, and clearly, $(\pi,\psi)$ is a representation of $M_L$.
Let $(\pi, \psi)$ be a representation of $M_L$ and let $\{\mu_\lambda\}$ be an approximate unit in $I=C_0(\Delta)$.
We  claim that the net of operators $T_\lambda:=\psi(q(\mu_{\lambda}))$ is strongly Cauchy. Indeed, let $h\in H$ and  $\lambda\leq \lambda'$, in the directed set $\Lambda$.  Then
 \begin{align*}
\|(T_\lambda-T_{\lambda'})h\|^2
&=\langle h,  L\left(\mu_\lambda-\mu_{\lambda'}\right)^2)h\rangle
\leq \langle h, L(\mu_\lambda-\mu_{\lambda'})  h \rangle.
\end{align*}
Since the net $\{L(\mu_\lambda)\}_{\lambda\in \Lambda}$ is strongly convergent  the last expression tends to zero.
Hence $T:=\textrm{s-}\lim_{\lambda \in \Lambda} T_\lambda$ defines a bounded operator. For every $a\in C_0(\Delta)$
we have
  $$T^* a T=\textrm{s-}\lim_{\lambda \in \Lambda} T_\lambda^* a T_\lambda=\lim_{\lambda \in \Lambda}L(\mu_{\lambda} a\mu_\lambda)=L(a).
	$$
%If $a$ is supported on a set $K$ such that $\varphi|_K$ is injective, then taking $u\in C_c(\Delta_{\reg})$  such that $u|_K=(\varrho|_K)^{-1}$
%	we get
%	$aT T^*u=\textrm{s-}\lim_{\lambda \in \Lambda} a T_\lambda T_\lambda u =\lim_{\lambda \in \Lambda} \mu_{\lambda} a=a.
%	$
	Thus $(\pi,T)$ is a representation of $L$ satisfying $TH \subseteq \overline{\pi(I)H}$.
\end{proof}
\begin{thm}\label{thm:crossed_product_Cuntz-Pimsner}
The crossed product $A\rtimes L$ is naturally isomorphic with Katsura's Cuntz-Pimsner algebra $\OO_{M_L}$. In particular, $A\rtimes L$ is always nuclear, and satisfies  the Universal Coefficient Theorem (UCT) if $A$ is separable (equivalently $X$ is second countable).
\end{thm}
\begin{proof}
By \cite[Propositions 3.3 and 4.9]{katsura} there is the largest ideal $J_{M_L}$ in $A$ such that for every faithful representation $(\pi,\psi)$ of $M_L$ we have
$$
\{a\in A: \pi(a)\in \overline{\psi(M_L)\psi(M_L)^*}\}\subseteq J_{M_L}.
$$
The faithful representation $(\pi,\psi)$ of $M_L$ is called \emph{covariant} if the above inclusion is an equality.
Hence by Lemma \ref{lem:representations_correspondence} and Propositions \ref{prop:katsuras_ideal} (and Example \ref{ex:orbit_representation}) we have $J_{M_L}=C_0(\Delta_{\reg})$
and we have a bijective correspondence between covariant representations
$(\pi, \psi)$ of $M_L$  and covariant representations $(\pi, T)$ of $L$ satisfying $TH \subseteq \overline{\pi(I)H}$.
%In fact by Lemma \ref{lem:representations_correspondence}
By definition $\OO_{M_L}$ is generated by the range of a universal covariant representation of $M_L$. By Proposition \ref{prop:universal_crossed_product} and Remark \ref{rem:representation_of_L}, $A\rtimes L$ is
generated by a universal  covariant representations $(\pi, T)$ of $L$ satisfying $TH \subseteq \overline{\pi(I)H}$. This gives a natural isomorphism  $A\rtimes L\cong\OO_{M_L}$,
cf. the last part of Proposition \ref{prop:universal_crossed_product}.

Since $A$ is commutative (and hence nuclear),  $A\rtimes L\cong\OO_{M_L}$ is nuclear by \cite[Corollary 7.4]{katsura}. If $A$ is separable,
then  satisfies the UCT by \cite[Proposition 8.8]{katsura}.
  \end{proof}
	\begin{rem}\label{rem:correspondence_description}
	We have seen in the above proof that Katsura's ideal $J_{M_L}$ for $M_L$ is $C_0(\Delta_{\reg})$. It also follows form the proof of Lemma \ref{lem:representations_correspondence}
	that the $C^*$-correspondence $M_L$ is naturally isomorphic to $\overline{It}$ with operations coming from the $C^*$-algebra  $A\rtimes L$.
	
	\end{rem}
\begin{cor}
The crossed product $A\rtimes L$ is naturally isomorphic to the crossed product $\OO(A,\alpha, L)$ by the partial endomorphism $\alpha$ defined
in \cite{er}.
\end{cor}
\begin{proof}
The crossed product $\OO(A,\alpha, L)$ in \cite{er} is defined to be  $\OO_{M_L}$.
\end{proof}
\begin{cor} If $\Delta=X$ and $\varrho>0$ on a dense subset of $X$, then
$A\rtimes L$ is naturally isomorphic to the Exel's crossed product $A\rtimes_{\alpha,L} \N$ \cite{exel3}, generalised to the non-unital case in
\cite{BroRaeVit}.
\end{cor}
\begin{proof}
The assumptions mean that $\alpha:A\to A$ is non-degenerate and  $L$ is faithful. Thus the assertion follows from \cite[Proposition 4.9]{kwa_Exel}.
\end{proof}
We naturally  associate to $L$ a \emph{topological correspondence} in the sense of \cite[Definition 2.1]{BHM}, see also \cite[Subection 9.3]{CKO}.
The underlying topological directed graph $(E^0,E^1,s,r)$ is the graph of $\varphi$:
$$
E^0:=X, \quad E^{1}:=\Delta,\quad  r(x):=x,\quad s(x):=\varphi(x).
$$
It is equipped with the continuous family  of measures $\mu=\{\mu_{y}\}_{y\in X}$ along fibers  of $\varphi$ given by $\mu_y(a):=L(a)(y)$, $a \in  C_c(X)$. Note that we only have $\supp \mu_y\subseteq s^{-1}(y)$, $y\in X$. Thus the  topological correspondence $\QQ:=(X, \Delta, id, \varphi,\mu)$ is a \emph{topological quiver} in the sense of \cite{mt} iff $\supp \mu_y= s^{-1}(y)$, $y\in X$   iff $\Delta=\Delta_{\pos}$ (note that we use a convention where  $s$ and $r$ play the opposite role in \cite{mt}).
\begin{lem}\label{lem:correspondence_coincidence}
The $C^*$-correspondence $M_{\QQ}$ associated to the topological correspondence $\QQ=(X, \Delta, id, \varphi,\mu)$ in \cite[Definition 2.4]{BHM}, cf. \cite[3.1]{mt}, coincides
 $M_L$. % and hence the associated Cuntz-Pimnser algebras are isomorphic.
\end{lem}
\begin{proof} This follows immediately from the constructions (definitions).
\end{proof}
\begin{cor}\label{cor:topological_quivers}
If $\Delta=\Delta_{\pos}$, so that $\QQ=(X, \Delta, id, \varphi,\mu)$ is a topological quiver, then the crossed product $A\rtimes L$ is naturally isomorphic to
the quiver $C^*$-algebra associated to $\QQ$ by Muhly and Tomforde in \cite{mt}.
\end{cor}
\begin{proof}
By definition the quiver $C^*$-algebra is the Cuntz-Pimsner algebra of $M_\QQ$, which by Lemma
\ref{lem:correspondence_coincidence} is equal to $M_L$. Hence the assertion follows from Theorem \ref{thm:crossed_product_Cuntz-Pimsner}.
\end{proof}
\begin{rem}\label{rem:topological_quiver_model}
Note that if $\Delta_{\pos}$ is locally compact, then $\QQ_{\pos}:=(X, \Delta_{\pos}, id, \varphi,\mu)$ is a topological quiver.
Moreover, if $\Delta_{\pos}\subseteq \Delta$ is  open, we may apply  Corollary \ref{cor:topological_quivers} to the restricted map $\varphi:\Delta_{\pos}\to X$, to conclude that  $A\rtimes L$ is the quiver algebra associated to  $\QQ_{\pos}$. If $\Delta_{\pos}$ is closed in $\Delta$ and $\Delta$ is normal,  one may show
that the $C^*$-correspondences $M_{L}$ and $M_{\QQ_{\pos}}$ are isomorphic and hence  $A\rtimes L$ is again the quiver algebra of $\QQ_{\pos}$.
We do not know whether  $A\rtimes L$ has a natural topological quiver model in general.
\end{rem}
\begin{ex}[Maps on Riemann surfaces]\label{ex:Riemann_maps}
Let $\varphi:\Delta\to X$ be a non-constant holomorphic  map defined on an open connected subset $\Delta$ of  a  Riemann surface $X$ (so that $\Delta$ is a Riemann surface as well). Let $x\in \Delta$. By branching lemma, $\varphi$  locally at $x$ looks like
$z\to z^d$, and then $m(x):=d \in \N$ is called the multiplicity of $\varphi$ at $x$. In particular, $\varphi^{-1}(y)$ is a discrete subset of $\Delta$, for every $y\in X$. Assume that $\varphi$ is proper.  Then it is surjective and the number  $d:= \sum_{x\in \varphi^{-1}(y)} m(x)$, called the \emph{degree} of $\varphi$, does not depend on $y\in X$ and is finite. In particular,
$$
L(a)(y):= \sum_{x\in \varphi^{-1}(y)} m(x)a(x), \qquad a\in C_0(\Delta)
$$
defines a transfer operator for $\varphi$, and $\|L\|=d$. If $\Delta=X=\widehat{\C}=\C\cup \{\infty\}$ is  the Riemann sphere, then $\varphi$ is a rational function $R:\widehat{\C}\to \widehat{\C}$ and the crossed product $C(\widehat{\C})\rtimes L$ is isomorphic to the $C^*$-algebra $\OO_R(\widehat{\C})$ associated to $R$ in \cite{Kajiwara_Watatani0} (which is  $\OO_{M_L}$ by definition).
If $R$ is of degree at least two and has an exceptional point, then $R$ is conjugated either to a polynomial or a map $z\to z^d$ for some $d\in \Z\setminus \{0\}$, \cite[Theorem 4.1.2]{Beardon}.
The rational map $R$ (and the transfer operator $L$) restricts to the\emph{ Julia set} $J_R$ and \emph{Fatou set} $F_{R}$ and the crossed products
$C(J_R)\rtimes L$ and $C_0(F_{R})\rtimes L$ to the $C^*$-algebras studied in \cite{Kajiwara_Watatani0}.
\end{ex}

\begin{ex}[Branched coverings with finite system of branches]\label{ex:branched_maps}
Let $\varphi:\Delta \to X$ be a continuous surjective partial map such that  $\varphi^{-1}$ has a \emph{finite system of branches}, i.e. there is a finite collection of partial maps  $\{\gamma_i\}_{i=1}^N$
such that each $\gamma_i:X\to \Delta$ is continuous injective and
$\varphi^{-1}(y)=\{\gamma_i(y):  i=1,...,N\}$. %(we may also assume that $|\varphi^{-1}(y)|=|\{ y\in \Delta_i, \, i=1,...,N\}|$).
Then $\varrho(x):=|\{i: x \in \gamma_i(X)\}|$ defines a potential for $\varphi$ as clearly
$$
L(a)(y)= \sum_{i=1}^N a(\gamma_i(y))=\sum_{x\in \varphi^{-1}(y)}\varrho(x) a(x), \qquad
a\in C_0(\Delta),
$$
 defines a transfer operator $L :C_0(\Delta)\to C_0(X)$ for $\varphi$.
If $X=\Delta$ is compact and $\gamma_i$ is a proper contraction, for all $i$,  the crossed product $A\rtimes L$
is naturally isomorphic to the the $C^*$-algebra associated to the \emph{self-similar set }$X$ in \cite{Kajiwara_Watatani}
(it is defined there as $\OO_{M_L}$). The model example is the \emph{tent map}  $\p:[0,1]\to[0,1]$ where  $\p(x)=1-|1-2x|$ and $L(a)(y)=a(\frac{y}{2})+a(1-\frac{y}{2})$.
\end{ex}
If a map has infinitely many branches one may define a transfer operator by using a scaling function that will make the sums converge:
\begin{ex}
Let $X=[0,1]$, $\Delta=(0,1]$ and $\varphi(x)=\sin\frac{1}{x}$. On may
define a transfer operator for $\varphi$ by the formula
$
L(a)(y)=2^{[y=\pm 1]} \sum_{x\in \varphi^{-1}(y)} e^{-1/x}a(x) .
$
\end{ex}

\section{Invariance uniqueness theorems and the regular representation}\label{sec:invariance_uniqueness_theorems}

%\subsection{Gauge-action and the core subalgebra}
An important consequence of universality of $A\rtimes L$ is that it is equipped with a \emph{circle gauge action} $\gamma:\mathbb{T}\to \Aut(A\rtimes L)$.
Namely, for each $\lambda \in \mathbb{T}$ the pair $(\id_A, \lambda t)$ may be treated as covariant representation of $L$. Hence by Proposition \ref{prop:universal_crossed_product}\ref{prop:universal_crossed_product2}
there is a $*$-epimorphism $\gamma_\lambda: A\rtimes L\to A\rtimes L$ such that
$$
{\gamma_\lambda}|_A=\id_A,\qquad\text{and} \qquad \gamma_\lambda (at)=\lambda at, \,\, a\in I.
$$
Moreover, we clearly have $\gamma_1=\id_{A\rtimes L }$ and  $\gamma_{\lambda_1}\circ \gamma_{\lambda_2}=\gamma_{\lambda_1 \lambda_2}$ for $\lambda_1,\lambda_2\in \mathbb{T}$. Thus  $\gamma:\mathbb{T}\to \Aut(A\rtimes L)$ is a group homomorphism. Its fixed points form a $C^*$-algebra
$
A_{\infty}:=\{x\in A\rtimes L: \gamma_\lambda(x)=x\text{ for all }\lambda \in \mathbb{T}\}.
$
We  call  $A_{\infty}$ the \emph{core $C^*$-subalgebra} of $A\rtimes L$. It is well known, see, for instance \cite[Proposition 3.2]{Raeburn},
that the formula
$
E(x):=\int_{\mathbb{T}} \gamma_\lambda(x)\, d\lambda
$
defines a faithful conditional expectation onto $A_\infty$. That is, $E$ is norm one projection onto $A_\infty$, which is   necessarily a completely positive $A_\infty$-bimodule map, see \cite[III, Theorem 3.4, IV, Corollary 3.4]{Tak}. Faithfulness here means that $E(a^*a)=0$ implies $a=0$ for all $a\in A\rtimes L$.
\begin{prop}\label{prop:conditional_expectation_description}
We have 
$$
A_{\infty}=\clsp\{at^n t^{*n}b: a,b\in I_n, n\in \N_0\},
$$
and the conditional expectation $E:A\rtimes L\to A_{\infty}$ is the unique contractive projection onto $A_{\infty}$ such that
$E(at^nt^{*m}b)=0$ for $n\neq m$ ($a\in I_n$, $b\in I_m$).
\end{prop}
\begin{proof}
We  have $E(at^n t^{*m}b)=at^n t^{*m}b \int_{\mathbb{T}} \lambda^{n-m}\, d\lambda$  which is zero when $n\neq m$ and
$at^n t^{*n}b$ when $n=m$. This determines $E$ and implies that $
A_{\infty}=\clsp\{at^n t^{*n}b: a,b\in I_n, n\in \N_0\}.
$
\end{proof}

The gauge-invariance uniqueness for Cuntz-Pimsner algebras implies the following
\begin{thm}\label{thm:faithfulness_on_the_core}
Let $(\pi,T)$ be  faithful covariant representation  of $L$ and let $C^*(\pi,T)$ be the $C^*$-algebra generated by $\pi(A)\cup \pi(I)T$.
Then $\pi\rtimes T$ is faithful on
the core subalgebra $A_\infty$ of  $A\rtimes L$, and the following conditions are equivalent:

\begin{enumerate}
\item\label{it:gauge_uniqueness1} $\pi\rtimes T$ is an isomorphism, i.e. $A\rtimes L\cong C^*(\pi,T)$;
\item\label{it:gauge_uniqueness2} $C^*(\pi,T)$ is equipped with a circle gauge-action, i.e.  a group homomorphism $\gamma:\mathbb{T}\to \Aut (C^*(\pi,T))$ where $\gamma_z|_{\pi(A)}=\id_{\pi(A)}$ and  $\gamma_z(\pi(a)T)= z\pi(a)T$, for $z\in \mathbb{T}$, $a\in I$;
\item\label{it:gauge_uniqueness3} There is a conditional expectation from $C^*(\pi,T)$ onto $(\pi\rtimes T)(A_\infty)\subseteq C^*(\pi,T)$  that annihilates all the operators of the form $\pi(a)T^mT^{*n}\pi(b)$ with $n\neq m$, $a\in I_m$, $b\in I_n$.
\end{enumerate}
\end{thm}
\begin{proof} Faithfulness of $\pi\rtimes T$ on $A_\infty$ follows from Theorem \ref{thm:crossed_product_Cuntz-Pimsner} and \cite[Theorem 6.4]{katsura}.
Implications \ref{it:gauge_uniqueness1}$\Rightarrow$\ref{it:gauge_uniqueness2}$\Rightarrow$\ref{it:gauge_uniqueness3}
are obvious, cf. Proposition \ref{prop:conditional_expectation_description}. Assume \ref{it:gauge_uniqueness3} and denote by $E_\pi$ the conditional expectation from $C^*(\pi,T)$ onto $(\pi\rtimes T)(A_\infty)$. Then $E_\pi\circ \pi\rtimes T=\pi\rtimes T\circ E$, and this composite map is faithful because $E$ is faithful and $\pi\rtimes T$ is faithful on the range of $E$. This implies that $\pi\rtimes T$ is faithful on the whole of $A\rtimes L$.
\end{proof}
\begin{cor}%[Iteration of the crossed product]
\label{cor:iteration} For each $n\in \N$, $A\rtimes L^n$ is naturally isomorphic to the $C^*$-subalgebra of $A\rtimes L$ generated by
$A\cup I_nt^n$.
\end{cor}
\begin{proof} By \eqref{eq:standard_relations_iterated} we see that $(\id,t^n)$ is a faithful representation of $L^n$ into $A\rtimes L$.
We use Proposition \ref{prop:characterization_of_covariance}\ref{it:characterization_of_covariance2.5} to show that the representation $(\id,t^n)$ is covariant.
To this end note that  the set of regular points for $\varrho_n$ is $\Delta_{\reg}^n=\{x\in \Delta_n: x,\varphi(x),...,\varphi^{n-1}(x)\in \Delta_{\reg}\}$.
Let  $a\in C_c(\Delta_{\reg}^n)$  with support $K$ contained in an open set $U\subseteq \Delta_{\reg}^n$ where $\varphi^n|_{U}$ is injective.
Put $a_0:=a$ and  for each $k=1,...,n-1$ let $a_k\in C_0(\varphi^k(U))$ be such that $a_k|_{\varphi^{k}(K)}\equiv 1$, so that then we have
$a=\prod_{k=0}^{n-1} \alpha^{k}(a_k)$. For each $k=0,...,n-1$, the map $\varphi|_{\varphi^k(U)}$ is injective. Hence by Proposition \ref{prop:characterization_of_covariance}\ref{it:characterization_of_covariance2.5}  there is $u_k\in C_0(X)$ such
that $a_k=a_ktt^*u_k$. In particular, since $a_0=a\in C_0(\Delta_n)$ we may assume that $u_0\in  C_0(\Delta_n)$. Then  $u:=\prod_{k=0}^{n-1} \alpha^k(u_k)$ is well defined and using \eqref{eq:standard_relations_iterated} we get
$$
a t^{n}t^{*n} u= a_0 (t a_1 t ... a_{n-1}tt^* u_{n-1} ... t^* u_1 t^*) u_0=\prod_{k=0}^{n-1} \alpha^{k}(a_k)=a.
$$
Hence $(\id,t^n)$ is a covariant representation of $L^n$. It is  equipped with a circle gauge-action. Indeed,  if $\gamma:\mathbb{T}\to \Aut (A\rtimes L )$
is the gauge-action $\gamma^n$ on $A\rtimes L$, then the desired gauge action on $\clsp\{at^{nk} t^{*nl}b: a\in I_{nk},b\in I_{nl}, n,m\in \N_0\}$ can be defined by the formula $\gamma^n_\lambda(b):=\gamma_z(b)$,  for any $\lambda\in \mathbb{T}$ and $z\in \mathbb{T}$ such that $z^n=\lambda$.
Hence we have the natural isomorphism $A\rtimes L^n\cong\clsp\{at^{nk} t^{*nl}b: a\in I_{nk},b\in I_{nl}, n,m\in \N_0\}$ by Theorem \ref{thm:faithfulness_on_the_core}.
\end{proof}

\subsection{Regular representation and  generalised expectations}\label{Sec:regular_representation}
The orbit representation $(\pi_o, T_o)$ defined in Example~\ref{ex:orbit_representation}
in general does give a faithful representation of $A\rtimes L$. 
%A dynamical criterion for $\pi_o\rtimes  T_o$ to be  faithful on  $A\rtimes L$ will be given further in Theorem \ref{thm:isomorphism}.
 Tensoring it with the  regular representation $\lambda$ of $\Z$ guarantees that:
\begin{defn}
The \emph{regular representation of the transfer operator $L$} is the pair $(\tilde{\pi}, \tilde{T})$ where
$\tilde{\pi}:C_0(X)\to B(H)$ and $\tilde{T}\in B(H)$ act on   $H:=\ell^2(X)\otimes \ell^2(\Z)\cong \ell^2(X\times \Z)$
by $\tilde{\pi}=\pi_o \otimes id_{ \ell^2(\Z)}$ and $\tilde{T}=T_o\otimes \lambda$.
Thus using the standard orthonormal basis $\{\mathds{1}_{x,n}\}_{x\in X, n\in \Z}$ of $H$ we have
$$
\tilde{\pi}(a)\mathds{1}_{x,n}= a(x)\mathds{1}_{x,n}, \qquad \tilde{T} \mathds{1}_{y,n} =\sum_{x\in\varphi^{-1}(y)}\sqrt{\varrho(x)}\mathds{1}_{x,n+1}.
$$
\end{defn}

\begin{prop}\label{prop:regular_is_faithful}
The regular representation $(\tilde{\pi},\tilde{T})$  is a faithful covariant representation of $L$ that
extends to a faithful representation $\tilde{\pi}\rtimes\tilde{T}$ of $A\rtimes L$, so
$
A\rtimes L\cong C^*(\tilde{\pi},\tilde{T}).
$
\end{prop}
\begin{proof}
Using that $(\pi_o, T_o)$ is a faithful covariant representation of $L$,
one readily concludes that  $(\tilde{\pi},\tilde{T})$ is also a faithful covariant representation of $L$.
By Theorem \ref{thm:faithfulness_on_the_core}, to prove that $\tilde{\pi}\rtimes\tilde{T}$
 is faithful it suffices to show  that $C^*(\tilde{\pi},\tilde{T})$ has the appropriate gauge action.
To this end, for each $z\in \mathbb{T}$  we define a unitary operator $U_z\in B(\ell^2(X\times \Z))$ by the formula $U_z \mathds{1}_{x,n}:=z^n \mathds{1}_{x,n}$, $ x\in X$, $n\in \Z$.
Putting $\gamma_z(b):=U_z b U_z^*$, $b\in C^*(\tilde{\pi},\tilde{T})$, we get   $\gamma_z|_{\pi(A)}=\id_{\pi(A)}$ and  $\gamma_z(\pi(a)\tilde{T})= z\pi(a)\tilde{T}$ for $z\in \mathbb{T}$ and $a\in I$.
Hence $\gamma:\mathbb{T}\to \Aut (C^*(\tilde{\pi},\tilde{T}))$ is the desired homomorphism.
\end{proof}
\begin{cor}%[Independence  of $C_0(X)\rtimes L$ on $\varrho$ up to a continuous factor]
\label{cor:independence}
Let $L$ and $L'$ be transfer operators for a fixed partial map $\varphi:\Delta\to X$ and let
$\varrho,\varrho':\Delta \to [0,+\infty)$ be the corresponding potentials.   Assume that
there is a continuous strictly positive map $\omega:\Delta:\to (0,\infty)$ such that $\varrho'=\varrho \omega$.
Then
$$
C^*(\tilde{\pi},\tilde{T})=C^*(\tilde{\pi},\tilde{T}'),
$$
where $(\tilde{\pi}, \tilde{T})$  and $(\tilde{\pi}, \tilde{T}')$ are regular representations of $L$ and $L'$ respectively.
Thus $C_0(X)\rtimes L$ and $C_0(X)\rtimes L'$ are naturally isomorphic.
\end{cor}
\begin{proof}
For any $a\in C_c(\Delta)$ we have  $a \omega^{\frac{1}{2}}, a \omega^{-\frac{1}{2}}\in C_c(\Delta)$,
$
\tilde{\pi}(a)\tilde{T}' =\tilde{\pi}(a \omega^{\frac{1}{2}})\tilde{T}$ and $\tilde{\pi}(a)\tilde{T} =\tilde{\pi}( a\omega^{-\frac{1}{2}})\tilde{T}'.
$
Hence  $\tilde{\pi}(C_c(\Delta))\tilde{T}'=\tilde{\pi}(C_c(\Delta))\tilde{T}$ which implies $\tilde{\pi}(C_0(\Delta))\tilde{T}'=\tilde{\pi}(C_0(\Delta))\tilde{T}$
and this gives the assertion.
\end{proof}

Using the regular representation we prove existence of a canonical faithful completely positive map  from $C_0(X)\rtimes L$
to  the $C^*$-algebra $\mathcal{B}(X)$  of all bounded Borel complex valued maps on $X$.  We denote by $\delta_{i,j}$ 
the Kronecker symbol.
\begin{lem}\label{lem:generalised_expectation}
There is a faithful completely positive map $G : C_0(X)\rtimes L\to  \mathcal{B}(X)$
such that
$$
G(at^kt^{*l}b)=\delta_{k,l}\cdot  ab\varrho_k
$$
for all $a\in I_k,b\in I_l$ and $k,l\in\N_0$. In particular, $G$ is a (genuine) conditional expectation from $C_0(X)\rtimes L$ onto $C_0(X)$ iff $\varrho:\Delta\to [0,+\infty)$ is continuous.
\end{lem}
\begin{proof} In view of Proposition \ref{prop:regular_is_faithful} we may identify $A\rtimes L$ with $C^*(\tilde{\pi}(A)\cup \tilde{\pi}(I)\tilde{T})$.
Let $P_{x,n}$ be the one-dimensional orthogonal projection onto the subspace spanned by
$\mathds{1}_{x,n}\in H:=\ell^2(X)\otimes\ell^2(\Z) $, for $(x,n)\in X\times \Z$. Since the projections $\{P_{x,n}\}_{(x,n)\in X\times \Z}$ are pairwise orthogonal and sum up, in the strong topology, to the identity operator, we get that the formula
$$
G(b):=\sum_{(x,n)\in X\times\Z}P_{x,n}bP_{x,n}, \qquad b\in B(H),
$$
defines a faithful, completely positive, contractive map (the series is strongly convergent).
For any $a\in I_k,b\in I_l$, $k,l\in\N_0$ $(x,n)\in X\times \Z$ we get
\begin{align*}
G(at^kt^{*l}b)\mathds{1}_{x,n}&=P_{x,n}at^kt^{*l}b\mathds{1}_{x,n}=P_{x,n}at^k b(x)\sqrt{\varrho_l(x)}\mathds{1}_{\varphi^l(x),n-l}
\\
&=P_{x,n}\sum_{t\in\varphi^{-k}(\varphi^l(x))}a(t)\sqrt{\varrho_k(t)} b(x)\sqrt{\varrho_l(x)} \mathds{1}_{t,n+k-l}\\
&=\delta_{k,l}\cdot a(x)b(x)\varrho_k(x) \mathds{1}_{x,n}
=\delta_{k,l}\cdot (ab\varrho_k) \mathds{1}_{x,n}.
\end{align*}
\end{proof}
\begin{rem} The above map $G$ is an identity on $A=C_0(X)\subseteq \B(X)$.
Therefore $G$ is a  \emph{generalised expectation} for the $C^*$-inclusion $A\subseteq A\rtimes L$
in the sense of  \cite[Definition 3.1]{Kwa-Meyer}. %In particular,  $G(ab)=aG(b)$ and $G(ba)=G(b)a$ for all $a\in A$, $b\in A\rtimes L$.
%Note that $G$ is a genuine conditional expectation onto $C_0(X)$ if and only if $\varrho$ is continuous.
\end{rem}
\begin{thm}\label{cor:Expectation_Invariance} Let $(\pi,T)$ be a faithful covariant representation  of $L$.
Then $\pi\rtimes T:A\rtimes L\to C^*(\pi,T) $ is faithful if and only if  there is a bounded linear map $F:C^*(\pi,T)\to \B(X)$  such that
$
F(\pi(a)T^kT^{*l}\pi(b))=\delta_{k,l}\cdot ab\varrho_k$,   for all $a\in I_k,b\in I_l, k,l\in\N_0$.
\end{thm}
\begin{proof} If $\pi\rtimes T$ is faithful, then $F$ exists by Lemma \ref{lem:generalised_expectation}.
Conversely, if $F$ exists, then for any  $b\in A\rtimes L$ with $\pi\rtimes T(b)=0$,
we have $G(b^*b)=F(\pi\rtimes T(b^*b))=F(\pi\rtimes T(b)^*\pi\rtimes T(b))=F(0)=0$, which by faithfulness of $G$ implies that $b=0$. Hence $\pi\rtimes T$ is faithful.
\end{proof}

\section{Local homeomorphisms and the groupoid model}
\label{sec:groupoid_picture}
In this section we assume that $\varphi:\Delta\to X$ is a  \emph{local homeomorphism}. %Then $(X,\varphi)$ is called  the singly generated dynamical system (SGDS)  in  \cite{Renault2000}.
We show  that the groupoid $C^*$-algebra associated in \cite{Renault2000} to  $(X,\varphi)$  is naturally isomorphic to the crossed product of $C_0(X)$ by a transfer operator.
We first discuss existence of a transfer operator for $(X,\varphi)$.

\begin{lem}[cf. {\cite[Lemma 2.1]{er}}]
Let $\varphi:\Delta\to X$ be a local homeomorphism. For any continuous function $\varrho:\Delta\to [0,+\infty)$  with
$\sup_{y\in X} \sum_{x\in \varphi^{-1}(y)}\varrho(x)<\infty$,
the formula
$L(a)(y)=\sum_{x\in\varphi^{-1}(y)}\varrho(x)a(x)$ defines a  transfer operator $L:C_0(\Delta)\to C_0(X)$ for $\varphi$.
Moreover every transfer operator for $\varphi$ is of the above form (even if we drop our standing assumption \eqref{eq:countable_to_one}).
\end{lem}
\begin{proof}
For each $a\in C_0(\Delta)$ and  $y\in Y$ we have $
|L(a)(y)|\leq  \sum_{x\in\varphi^{-1}(y)}|\varrho(x)a(x)|\leq \| a\|\cdot M$ where $M:=\sup_{y\in X} \sum_{x\in \varphi^{-1}(y)}\varrho(x)
$.
Hence $L$ is a well defined bounded linear operator from $C_0(\Delta)$ to the space of bounded functions on $X$.
Clearly, $L$ is positive and satisfies the transfer identity \eqref{eq:transfer_identity}. Thus
 it suffices to show that
 $L(a)$ is continuous on $X$ for any $a\in C_c(\Delta)$, see Remark \ref{rem:definition_of_transfer}(1).
 Let  $K$ be the compact support of $a$ and take any $y\in X$. If $y\notin \varphi(K)$, then $L(a)(y)=0$ and as $X\setminus\varphi(K)$ is open, $L(a)$ is continuous at $y$.
 Assume then that $y\in \varphi(K)$.
 Since $\varphi$ is a local homeomorphism,  $\varphi^{-1}(y)\cap K$ is finite, and
 we may find  pairwise disjoint, non-empty open sets $\{U_i\}_{i=1}^n$ covering  $\varphi^{-1}(y)\cap K$ and  such that
$\varphi|_{U_i}$ is injective for any $i=1,...n$.
%By shrinking $U_i$'s if necessary we may  assume that  $\varphi(U_i)=V$ for all $i=1,...,n$ and for a fixed neighbourhood $V$ of $y$
By \cite[Lemma 2.1 claim]{er} we may find open  $V\subseteq \bigcap_{i=1}^{n}\varphi(U_i)$ containing $y$
 and such that $\varphi^{-1}(V)\cap\big(K\setminus\bigcup_{i=1}^n U_i\big)=\emptyset$. So   $\varphi^{-1}(V)\cap K \subseteq \bigcup_{i=1}^n U_i$.
Using this we see that $L(a)|_V=\sum_{i=1}^n (\varrho\circ \varphi|_{U_i}^{-1})\cdot ( a\circ \varphi|_{U_i}^{-1})$.
Since the latter sum is finite and involves only continuous functions, $L(a)|_V$ is continuous.
This finishes the proof of the first part.

For the second part note that since $\varphi$ is a local homeomorphism, for every $y\in X$, $\varphi^{-1}(y)$ is discrete. Hence
the measures in \eqref{equ:transfer_operator_form} have to be discrete, here we do not need our standing assumption \eqref{eq:countable_to_one}.
Thus every transfer operator $L$ for $\varphi$ is of the form \eqref{eq:transfer operator} and the associated potential $\varrho$ is continuous by
Proposition \ref{prop:properties_or_rho}. In addition, $\sup_{y\in X} \sum_{x\in \varphi^{-1}(y)}\varrho(x)= \|L\|<\infty$.
\end{proof}

The important question is whether we can find a \emph{strictly positive} continuous  $\varrho:\Delta\to (0,+\infty)$  with
$\sup_{y\in X} \sum_{x\in \varphi^{-1}(y)}\varrho(x)<\infty$. Note that a necessary condition for this is our standing assumption   \eqref{eq:countable_to_one}.
We answer this question in the affirmative in two important cases.

\begin{ex}
If $\varphi:\Delta\to X$ is a proper local homeomorphism, then for each $y\in X$ the preimage $\varphi^{-1}(y)
$ is finite and in fact the map $X\ni y\mapsto |\varphi^{-1}(y)|\in \N_0$ is continuous (locally constant),  see \cite[Lemma 2.2]{BroRaeVit} where it is assumed that $\Delta=\varphi(\Delta)=X$ but the proof works in our  setting.  Thus putting $\varrho(x):=|\varphi^{-1}(\varphi(x))|^{-1}$, $x\in\Delta$,  we get a continuous strictly positive function $\varrho>0$ such that
$\sum_{x\in \varphi^{-1}(y)}\varrho(x)=1$ for every $y\in \varphi(\Delta)$. The corresponding transfer operator is
given by the formula
\begin{equation}\label{equ:classical_transfer}
L(a)(y)= \frac{1}{|\varphi^{-1}(y)|}\sum_{x\in \varphi^{-1}(y)} a(x), \qquad a\in C_0(\Delta).
\end{equation}
If  $\varphi$ is not proper,   \eqref{equ:classical_transfer} fails to
define a transfer operator  even if we assume $\sup_{y\in X}|\varphi^{-1}(y)|<\infty$ (the function $L(a)(y)$ may be discontinuous).
\end{ex}
\begin{ex}\label{ex:rho_countable} Let $\varphi:\Delta\to X$ be any  local homeomorphism, but assume
that there is a partition of unity $\{f_n\}_{n=1}^\infty$ subordinated to a countable cover $\{U_n\}_{n=1}^\infty$ of $\Delta$ such that
$\varphi|_{U_n}$ is injective. Such a partition exists if $\Delta$ is second countable or more generally if $\Delta$ is $\sigma$-compact.
Then
$$
\varrho(x)=\sum_{n=1}^\infty \frac{1}{2^n}f_n(x), \qquad x\in \Delta,
$$
defines a continuous strictly positive function $\varrho:\Delta \to (0,+\infty)$ such that for every $y\in X$ we have  $\sum_{x\in \varphi^{-1}(y)}\varrho(x) \leq 1$. Thus
$\varrho$ yields a transfer operator.
\end{ex}
For an introduction to the theory of  \'etale, locally compact, Hausdorff groupoids  we recommend \cite{Sims}.
For any such a groupoid  $\G$ the \emph{groupoid $C^*$-algebra} $C^*(\G)$ is the maximal $C^*$-completion of
 the $*$-algebra $C_c(\G)$ with
operations:
$$
(f *g)(\gamma)=\sum_{\gamma_1\gamma_2=\gamma} f(\gamma_1)g(\gamma_2)\qquad  \text{ and } \qquad  f^*(\gamma)=\overline{f(\gamma^{-1})},
$$ where $\gamma,\gamma_1, \gamma_2\in \G$,
 $f,g\in C_c(\G)$.
Then the embedding $C_c(\G)\subseteq  C_0(\G)$ extends to a contractive embedding $C^*(\G)\subseteq  C_0(\G)$, so that we may view elements of
$C^*(\G)$ as functions on $\G$ and the formulas for algebraic operations remain valid. Also,  identifying $X$ with $\G^0:=\{(x,0,x):x\in X\}$,   $C_0(X)\subseteq C^*(\G)$ is a non-degenerate $C^*$-subalgebra and there is a
conditional expectation $F$ from $C^*(\G)$ onto $C_0(X)$ given by restriction of functions. This conditional expectation $F$ is faithful if $\G$ is amenable.

The  \emph{transformation groupoid} or \emph{Renault-Deaconu groupoid} associated to $(X,\varphi)$ is an  \'etale, amenable, locally compact, Hausdorff groupoid, see \cite{Renault2000}, where
$$
\G := \{(x,n-m,y): n,m \in \N_0, x\in \Delta_n, y\in \Delta_m, \varphi^n(x)=\varphi^{m}(y)\},
$$
the groupoid structure is given by $
(x,n,y) (y,m,z):= (x,n+m,z)$,  $(x,n,y)^{-1}:=(y,-n,x)
$, and the topology is defined by the basic open
sets $ \{(x,n-m,y):  (x,y)\in U\times V, \varphi^n(x) = \varphi^m(y)\}$
where $U\subseteq \Delta_n$, $V\subseteq \Delta_m$ are  open sets such that $\varphi^m|_U$
and  $\varphi^n|_V$ are injective.
For full local homeomorphisms on compact spaces the isomorphism in the following theorem is well known, see \cite{exel_vershik}, \cite{exel_renault}, \cite{bar_kwa}.
\begin{thm}\label{thm:local_homeo_crossed_groupoid}
Assume  $\varphi:\Delta\to X$ is a local homeomorphism and  let $\varrho:\Delta\to (0,+\infty)$ be any strictly positive continuous map  with
$\sup_{y\in X} \sum_{x\in \varphi^{-1}(y)}\varrho(x)<\infty$ (such a map always exists when $\varphi$ is proper or $\Delta$ is $\sigma$-compact).
%, and then \eqref{eq:transfer_identity} is automatically satisfied).
Then
$L(a)(y)=\sum_{x\in\varphi^{-1}(y)}\varrho(x)a(x)$ is a well defined transfer operator $L:C_0(\Delta)\to C_0(X)$ for $\varphi$, and
we have an isomorphism
$$
 C^*(\G)\cong C_0(X)\rtimes L
$$
where $\G$ is the Renault-Deaconu groupoid associated to $\varphi$. This isomorphism is determined by the formula
$$
\Phi(a_n\otimes b_m):= a_n {\varrho_n}^{-\frac{1}{2}} t^nt^{*m}  {\varrho_m}^{-\frac{1}{2}}  b_m , \qquad a_n\in C_c(\Delta_n), \, b\in C_c(\Delta_m), n,m\in \N_0,
$$
where $(a_n\otimes b_m)(x,k,y)=\delta_{k,n-m}\cdot a_n(x) b_m(y)$.
\end{thm}
\begin{proof}

%We claim that the function $T\in C_b(\G)$ given by
%$
%T(x,1,\varphi(x))=\varrho(x)^{\frac{1}{2}}$   and  $T(x,n,y)=0$ for $n\neq 1$,
%is a multplier for $C^*(\G)$

Let us assume that $ C^*(\G)\subseteq B(H)$ is represented in a faithful and non-degenerate way on some Hilbert space $H$.
 Let $\{\mu_\lambda\}_{\lambda\in \Lambda}\subseteq C_c(\Delta)$ be an approximate unit
  in $I$ and consider the net of functions $\{T_\lambda\}_{\lambda\in \Lambda}\in C_c(\G)$ given by
$
T_\lambda(x,1,\varphi(x))=\mu_\lambda(x)\varrho(x)^{\frac{1}{2}}$   and  $T_\lambda(x,n,y)=0$ if $(n,y)\neq (1,\varphi(x))$.
We  claim that $\{T_\lambda\}_{\lambda\in \Lambda}$ is strongly Cauchy. Indeed, let $a\in A$, $h\in H$ and  $\lambda\leq \lambda'$, in the directed set $\Lambda$.
We have $T_\lambda^* a T_{\lambda'}=L(\mu_\lambda a \mu_{\lambda'})$ in the $*$-algebra $C_c(\G)$. Thus
 \begin{align*}
\|(T_\lambda-T_{\lambda'})ah\|^2
&=\langle h,  L\left(\alpha(a^*)  (\mu_\lambda-\mu_{\lambda'})^2\alpha( a)\right)h\rangle
\\
&\leq \langle h, a^*L(\mu_\lambda-\mu_{\lambda'}) a h \rangle
\\
&=\langle ah, L (\mu_\lambda-\mu_{\lambda'})   ah \rangle.
\end{align*}
Since the net $\{L(\mu_\lambda)\}_{\lambda\in \Lambda}$ is strongly convergent  the last expression tends to zero.
Hence $T:=\textrm{s-}\lim_{\lambda \in \Lambda} T_\lambda$ defines a bounded operator. For every $a\in C_0(\Delta)$
we have
  $$T^* a T=\textrm{s-}\lim_{\lambda \in \Lambda} T_\lambda^* a T_\lambda=\lim_{\lambda \in \Lambda}L(\mu_{\lambda} a\mu_\lambda)=L(a).
	$$
	If $a$ is supported on a set $K$ such that $\varphi|_K$ is injective, then taking $u\in C_c(\Delta_{\reg})$ such that $u|_K=(\varrho|_K)^{-1}$
	we get
	$aT T^*u=\textrm{s-}\lim_{\lambda \in \Lambda} a T_\lambda T_\lambda u =\lim_{\lambda \in \Lambda} \mu_{\lambda} a=a.
	$
	Hence $(\id, T)$ is a covariant representation of $L$ by Proposition \ref{prop:characterization_of_covariance}.
	Thus we have a $*$-homomorphism $\id\times T:C_0(X)\rtimes L \to B(H)$. It takes values in $C^*(\G)$ because if $a\in C_c(\Delta)$, then $aT\in C_c(\G)$
	where  $aT(x,k,y)=\delta_{(k,y), (1,\varphi(x))}\cdot a(x)\varrho(x)^{\frac{1}{2}}$.
	More generally, one readily checks that for $a_n\in C_c(\Delta_n)$, $b\in C_c(\Delta_m)$
	we have $a_n {\varrho_n}^{-\frac{1}{2}} T^n T^{*m}{\varrho_m}^{-\frac{1}{2}}  b_m =a_n\otimes b_m\in C_c(\G)$.
	Since functions $a_n\otimes b_m$ span $C_c(\G)$ we conclude that
	$\id\times T:C_0(X)\rtimes L \to C^*(\G)$ is a surjective $*$-homomorphism that  intertwines the conditional expectations
$G:C_0(X)\rtimes L \to C_0(X)$ and  $F:C^*(\G) \to C_0(X)$. Hence $\id\times T$  is an isomorphism by Corollary  \ref{cor:Expectation_Invariance}. Its inverse is as described in the assertion.
\end{proof}

\begin{ex}[Deaconu-Muhly $C^*$-algebras associated with branched coverings]
We consider a slightly more general situation than in \cite{Deaconu-Muhly} and by a \emph{branched self-covering} we mean a  continuous open and surjective map $\varphi:X\to X$ of a locally compact, $\sigma$-compact space, for which there is a closed set $S\subseteq X$ such that $\varphi|_{X\setminus S}$
is a local homeomorphism. The $C^*$-algebra $DM(X,\varphi)$ associated to $\sigma$ in \cite{Deaconu-Muhly} is by definition the $C^*$-algebra of the Renault-Deaconu groupoid
associated to the partial local homeomorphism $\varphi:X\setminus S\to X$.
Thus by Theorem \ref{thm:local_homeo_crossed_groupoid} we have
$$
DM(X,\varphi)\cong C_0(X)\rtimes L.
$$
where $L:C_0(X)\to  C_0(X)$ is any transfer operator for $\varphi:X \to X$ given by continuous $\varrho:X\to [0,+\infty)$   with  $S=\varrho^{-1}(0)$.
If in addition, $S$ has empty interior, then $C_0(X)\rtimes L$ is Exel's crossed product.
\end{ex}
\begin{ex}[Graph $C^*$-algebras]\label{ex:Graph_algebras}
Let $E = (E^0,E^1, r, s)$ be a countable directed graph  ($r,s:E^{1}\to E^0$ are range and source  maps).
 The \emph{boundary space} $\partial E= E^\infty\cup E^*_{s}\cup E^*_{inf} $ of $E$, cf. \cite{Webster}, \cite[Subsection 4.1]{Brownlowe} or \cite{kwa_Exel}, as a
set consist of all infinite paths  and of finite paths that start in sources  or in infinite emitters.
It is a locally compact Hausdorff space with topology generated by  cylinder sets  and their complements.
The one-sided \emph{topological Markov shift}  associated to $E$ is the map $\sigma:\partial E\setminus E^0 \to \partial E$  defined, for $\mu=\mu_1\mu_2...\in \partial E\setminus E^0$,  by the formulas
$$
\sigma(\mu):=\mu_2\mu_3...\,\, \textrm{ if }\,\,\mu \notin E^1, \quad \textrm{ and} \quad \sigma(\mu):=s(\mu_1)\,\, \textrm{ if }\,\,\mu=\mu_1 \in E^1.
$$
This is a countable-to-one local homeomorphism.
So we have a partial endomorphism $\alpha:C_0(\partial E) \to  M(C_0(\partial E\setminus E^0))$.
One may always find strictly positive numbers
 $\lambda=\{\lambda_e\}_{e\in E^1}$, such that the formula
%\begin{equation}\label{general formula for L of a graph}
$$ 
L(a)(\mu)=\sum\limits_{e \in E^1,\, e\mu \in \partial E} \lambda_{e}\, a(e\mu)
$$
%\end{equation}
defines a bounded map $ L:C_0(\partial E\setminus E^0)\to C_0(\partial E)$  (\cite[Proposition 5.4]{kwa_Exel}
 characterises when this happens), and then $L$ is a  transfer operator for $\sigma$. By \cite[Theorem 5.6]{kwa_Exel}  we then also have
$$
C^*(E)\cong C_0(\partial E)\rtimes L,
$$
where $C^*(E)$ is the graph $C^*$-algebra - the universal $C^*$-algebra generated by partial isometries $\{s_e: e\in E^1\}$ and mutually  orthogonal projections $\{p_v: v\in E^1\}$ such that  $s_e^*s_e=p_{s(e)}$, $s_e s_e^*\leq p_{r(e)}$   and $p_v=\sum_{r(e)=v} s_e s_e^*$ whenever the sum is finite. % (i.e. $v$ is a  finite receiver).
\end{ex}

%\subsection{Self-similar sets}

\begin{ex}[Exel-Laca $C^*$-algebras]
Let $I$ be any set and let $\AA =
\{A(i,j)_{i,j \in I} \}$ be a $\{0,1\}$-matrix
over $I$ with no identically zero rows.  The Exel-Laca
algebra $\OO_\AA$ is the universal
$C^*$-algebra generated by partial isometries $\{ s_i : i \in
I \}$ with commuting initial projections and mutually
orthogonal range projections satisfying $s_i^* s_i s_j s_j^*
= A(i,j) s_js_j^*$ and
%\begin{equation}
%\label{eq:exel-laca_relation}
$$
\prod_{i \in E} s_i^*s_i \prod_{j \in F} (1-s_j^*s_j) =
\sum_{k \in I} \prod_{i \in E} A(i,k)
\prod_{j \in F} (1-A(j,k)) s_ks_k^*
$$%\end{equation}
whenever $E, F\subseteq I$ are finite sets   such that  $\prod_{i \in E} A(i,k)
\prod_{j \in F} (1-A(j,k))$ is non zero only for finitely many $k\in I$.
 %We put $q_i:=s_i^*s_i$ for $i \in I$, and
For any word $\alpha=\alpha_1...\alpha_n$ in $I$  admissible by $\AA$ we put
$s_{\alpha}=s_{\alpha_1}...s_{\alpha_n}$. Then
$$
\D_\AA:=\clsp\{s_{\alpha} \big(\prod_{i \in E} s_x^*s_x\big)  s_{\alpha}^*: E\subseteq I \text{ is a finite set}, \alpha \text{ is a finite word} \}
$$
is a commutative $C^*$-subalgebra of $\OO_\AA$. The spectrum $X$ of this algebra is a second countable totally disconnected space described in \cite{exel_laca}, as a certain
subset of $\{0,1\}^\mathbb{F}$ where $\mathbb{F}$ is a free group generated by $I$. It is also
described in \cite{Renault2000} as a spectrum of a certain Boolean algebra that model a Markov shift.
In particular, there is a naturally associated  partial local homeomorphism  $\varphi:\Delta\to X$ defined on an  open dense subset $\Delta\subseteq X$.
The space of infinite admissible words
$$
X_\AA:=\{ \omega\in I^\N:   A(\omega_{n},\omega_{n+1})=1 \text{ for }n\in \N\}
$$
embeds naturally into $X$ (and is dense in $X$ when  $\AA$ is irreducible), in  a way that
$$
\varphi(\omega_1\omega_2...)=\omega_2\omega_3..., \qquad \text{ for } \omega \in X_\AA\cap \Delta.
$$
Moreover,  by  \cite[Proposition 4.8]{Renault2000}  $\OO_\AA$ is isomorphic to the    $C^*$-algebra $C^*(\G)$ of the Renault-Deaconu groupoid associated to $\varphi$.
Thus by Theorem \ref{thm:local_homeo_crossed_groupoid}, $\OO_\AA$ is isomorphic to  the crossed product $C_0(X)\rtimes L$ for a certain transfer operator $L$ for $\varphi$.
Such an isomorphism  is described in  \cite[Proposition 2.13]{er} for an unbounded transfer operator, and one can make the operator bounded by choosing appropriate potential $\varrho$.
For instance,  as in Example \ref{ex:Graph_algebras}, it suffices to choose positive numbers $\lambda=\{\lambda_i\}_{i\in I}$, such that the formula
$$
T_\lambda:=\sum_{i\in I} \sqrt{\lambda_i} s_i
$$
converges strictly in $\OO_\AA$, cf. \cite[Proposition 5.4]{kwa_Exel}. Then
$
L(a):=T_\lambda (a) T_\lambda^*
$
defines a bounded transfer operator for $\varphi$ and  $\OO_\AA\cong C_0(X)\rtimes L$.
\end{ex}

 \section{Spectra of the core subalgebras}\label{sec:Spectra of the core subalgebras}

%We fix a transfer operator $L:C_0(\Delta)\to C_0(X)$ for a partial dynamical system $(X,\varphi)$.
 We now proceed to
the analysis of the internal structure the core subalgebra $A_{\infty}$ of $A\rtimes L$. 
The fundamental fact is that $A_{\infty}=\overline{\bigcup_{n\in \N_0}A_n}$ is a direct limit of algebras that can be further decomposed into `liminary pieces'.
Namely, for  each $n\in \N_0$ we put
$$
K_n:=\overline{I_nt^{n}t^{*n}I_n}, \quad  A_n:=K_0+K_1+\dots +K_n=\clsp\{at^k t^{*k}b: a,b\in I_k, k=0,...,n\}.
$$
%Then $A_{\infty}=\overline{\bigcup_{n\in \N_0}A_n}$
%   
\begin{prop}\label{prop:reduced_to_katsura}
 For each $n\in \N_0$, $A_n$ and $K_n$ are $C^*$-sualgebras of  $A_{\infty}$. 
Moreover, $K_{n}K_m=K_m$ for $n\leq m$ and $A_n\cap K_{n+1}=K_{n}\cap K_{n+1}=\overline{I_n t^n C_0(\Delta_{\reg}) t^{*n} I_n}$.
\end{prop}
\begin{proof} By Remark \ref{rem:correspondence_description}, $J_{M_L}=C_0(\Delta_{\reg})$ and  we have an isomorphism of $C^*$-correspondences $M_{L}\cong \overline{It}$. 
By  Corollary \ref{cor:iteration}, this implies that for any $n\in \N$ we have $M_{L^n}\cong \overline{I_nt^{n}}$, and so also $M_{L^n}\cong M_{L}^{\otimes n}$, see  \eqref{eq:I_n_spanning_set}, \eqref{eq:transfer_powers} and \eqref{eq:endomorphism_powers}. In particular, $K_n$ is naturally isomorphic with compact operators on  $M_{L^n}$, and the assertion follows from the corresponding facts for Cuntz-Pimsner algebras, see, for instance, \cite[Lemma 5.4, Propositions 5.9, 5.11]{katsura}.
\end{proof}
Recall that $I_n=C_0(\Delta_n)$, where $\Delta_n$ is the domain of $\varphi^n$, and $I_0=A=C_0(X)$.
 We put
$$
\Delta_{\pos,n}:=\Delta_{n}\setminus  \varrho_n^{-1}(0)=\{x\in \Delta_n: \varrho_n(x)>0\}, \qquad n\in \N_0,
$$
which is the natural domain for the $n$-th iterate of the partial map $\varphi|_{\Delta_{\pos}}$ where $\Delta_{\pos}:=\{x\in \Delta: \varrho(x)>0\}$.
Using the transfer identity, we see that the closure of $L^n(I_n)$ is an ideal  in $A$.
Its spectrum is
$$
\widehat{L^n(I_n)}=\{y\in X:\varphi^{-n}(y)\setminus \varrho_n^{-1}(0)\neq \emptyset \}=\varphi^n(\Delta_{\pos,n}),
$$
and in particular, this set is open in $X$.
For any positive function $\rho:\Omega \to (0,\infty)$ we denote by $\ell^2(\Omega, \rho)$ the weighted $\ell^2$-space consisting
 of those functions $f:\Omega\to \C$ for which $\|f\|_2:=\left(\sum_{x\in \Omega} |f(x)|^2\rho(x)\right)^{1/2} <\infty$. This is a Hilbert space unitarily isomorphic to
$\ell^2(\Omega)$ via the map $\ell^2(\Omega, \rho)\ni \mathds{1}_{x} \mapsto \sqrt{\rho(x)}\mathds{1}_{x}\in \ell^2(\Omega)$.
\begin{prop}\label{prop:spectrum_of_K_n}
 For each $n\in \N$ the algebra $K_n$ is liminary (in fact it has a continuous  trace) and up to unitary equivalence all its irreducible representations are subrepresentations of the orbit representation. Moreover, we have a homeomorphism
$$
\widehat{K}_n\cong \varphi^{n}(\Delta_{\pos,n})
$$
under which the representation
$\pi_y^n$ of $K_n$ corresponding to $y\in \varphi^{n}(\Delta_{\pos,n})$
acts on $H_{y}^n:=\ell^2(\varphi^{-n}(y)\setminus \varrho_n^{-1}(0), \varrho_n)$ and  is defined   by
$$
 \pi_y^n(at^{n}t^{*n}b)h=a \cdot \left(\sum_{x\in \varphi^{-n}(y)}\varrho_n(x)b(x)h(x)\right),
$$
for $a,b\in I_n$ and $h\in H_{y}^n$.
%where $1\in H_{y}^{n}$ is the function constantly equal one.
The map $U\mapsto \overline{I_nt^{n}C_0(U)t^{*n}I_n}$ is a bijection between open subsets of $\varphi^{n}(\Delta_{\pos,n})$ and ideals  in $K_n$.
\end{prop}
\begin{proof} 
%To lighten the notation we will assume that $n=1$ and consider $K_1=\overline{Itt^*I}$ where $I=C_0(\Delta)$. We do not loose generality in this way because
%$A\rtimes L^n\subseteq A\rtimes L$, by Corollary \ref{cor:iteration}, and therefore $K_n$ for $A\rtimes L$ is the same as  $K_1$ for $A\rtimes L^n$.

As in the proof of Proposition \ref{prop:reduced_to_katsura}, we may identify $M_{L^n}$ with $\overline{I_nt^n}$ and then $K_n$ is identified with the algebra of compact operators on $M_{L^n}$.
Hence $\overline{I_nt^{n}}$ is a Morita-Rieffel equivalence bimodule between $K_n=\overline{I_nt^nt^{*n}I_n}$ and $\overline{t^{*n}I_nt^n}=\overline{L^n(I_n)}=C_0(\varphi^n(\Delta_{\pos,n}))$.
Such an equivalence preserves spectra and a number of other properties, see \cite{RaeWill}. In particular $K_n$ has a   continuous trace, because $C_0(\varphi^n(\Delta_{\pos,n}))$ has it.
The ideal in $K_n$ corresponding via the equivalence $M_{L^n}=\overline{I_nt^{n}}$ to an ideal  $C_0(U)$ in  $C_0(\varphi^n(\Delta_{\pos,n})$  is  $\overline{\langle M_{L^n}, C_0(U)M_{L^n}\rangle}=\overline{I_nt^{n}C_0(U)t^{*n}I_n}$.
This correspondence extends to a homeomorphism  $\widehat{K}_n\cong \varphi^n(\Delta_{\pos,n})$ where the representation $\pi_y$ of $K_n$ corresponding to  $y\in \varphi^n(\Delta_{\pos,n})$
acts on the Hilbert space $H_y:=M_{L^n}\otimes_{\text{ev}_{y}} \C$ which by construction is the Hausdorff completion of the algebraic tensor product $C_0(\Delta_n)\otimes \C$ in  seminorm
coming from the sesqui-linear form determined by   
$$
\langle c_1\otimes \lambda_1, c_2\otimes \lambda_2 \rangle= \overline{\lambda_1} L(c_1^* c_2)(y)\lambda_2=\sum_{x\in \varphi^{-n}(y)}\varrho_n(x) \overline{\lambda_1 c_1(x)} \lambda_2c_2(x).
$$
Then $\pi_y$ is determined by the formula  $\pi_y(at^nt^*b)[c\otimes \lambda]=[a L^n(bc) \otimes \lambda]$, for $a,b,c \in C_0(\Delta_n)$, $\lambda \in \C$.
Using this one readily sees that the map $[c\otimes 1]\mapsto c|_{\varphi^{-n}(y)\setminus \varrho_n^{-1}(0)}$   determines a unitary $H_{y}\cong H_{y}^n= \ell^2(\varphi^{-n}(y)\setminus \varrho_n^{-1}(0), \varrho_n)$ that intertwines $\pi_y$ and $\pi_y^n$. 
Moreover,  the subspace $G_{y}:=\ell^2\left(\varphi^{-n}(y)\setminus \varrho_n^{-1}(0)\right)$ of $\ell^2(X)$ is invariant
under the action of %the $C^*$-algebra
$\pi_o\rtimes T_o(K_n)$, because for $a\in A$ and $x\in \varphi^{-n}(y)\setminus \varrho_n^{-1}(0)$ we have
$$
\pi_o(a)\mathds{1}_x=a(x)\mathds{1}_x, \qquad  T_o^nT_o^{*n}\mathds{1}_x=\sum_{x'\in \varphi^{-n}(y)\setminus \varrho_n^{-1}(0)} \sqrt{\varrho_n(x')\varrho_n(x)} \mathds{1}_{x'}.
$$
Thus  we have  a subrepresentation $\sigma_y:K_n\to B(G_{y})$  of $\pi_o\rtimes T_o|_{K_n}$ where $\sigma_y(at^*tb)=\pi_o(a)T_oT_o^{*}\pi_o(b)|_{G_{y}}$. The canonical isomorphism $G_y\cong H_y^n$ is a unitary equivalence between $\sigma_y$ and $\pi_y^n$.
\end{proof}
\begin{rem}
If $\Delta=X$ is compact,  then $t\in A\rtimes L$ and the unique extension of $\pi_y^n$  to $A+K_n$ is defined by the formulas
$
\pi_y^n(a)h=a \cdot h$,  $\pi_y^n(t^{n}t^{*n})h=\left(\sum_{x\in \varphi^{-n}(y)}\varrho_n(x)h(x)\right) \cdot 1,
$
for $a\in A$ and $h\in H_{y}^n$. Thus $\pi_y^n(a)$ is a multiplication operator and $\pi_y^n(t^{n}t^{*n})$ is a rank one operator whose range consists of constant functions.
\end{rem}

Having a continuous map $f:U\to Y$ defined on an open subset $U$ of a  topological space $X$, we may  \emph{attach $X$ to $Y$ along $f$} to get the space
$
X\cup_f Y:=(X\sqcup Y)/(x\sim f(x) \text{ for all }x\in U)
$
equipped with the quotient topology. This is the \emph{pushout} of $f:U\to Y$ and the inclusion map $U\subseteq X$.
We may always identify $X\cup_f Y$ with the disjoint union
$
X\cup_f Y:=(X\setminus U)\sqcup Y
$
where the second summand $Y$ is open in  $X\cup_f Y$ and if the map $f$ is open, then
the open sets in $X\cup_f Y$ can be identified with pairs  of open sets $V\subseteq  X$, $W\subseteq Y$ satisfying $\varphi^{-1}(W)=V\cap U$  (then the corresponding open set in $X\cup_f Y$ is
 $V\setminus U \sqcup W$).
We use this construction to describe the spectrum of the  $C^*$-algebras $A_n$, as $A_{n+1}=A_{n} + K_n$ may be viewed as a pushout of $A_n$ and $K_{n+1}$ via
the $C^*$-algebra $A_{n}\cap K_{n+1}$.
\begin{lem}\label{lem:pushout_K_n}
For each $n\in \N$,  we have continuous bijection form $
\widehat{K_{n}+K_{n+1}}$ onto the pushout of  $ \varphi^{n}(\Delta_{\pos,n})$  and  $\varphi^{n+1}(\Delta_{\pos,n+1})$
along the partial homeomorphism  $\varphi: \varphi^{n}(\Delta_{\pos,n})\cap \Delta_{\reg} \to  \varphi^{n+1}(\Delta_{\pos,n+1}) $. We have a continuous  bijection
\begin{equation}\label{eq:pushout_bijection}
\widehat{K_{n}+K_{n+1}}\stackrel{\cong}{\longrightarrow}\varphi^{n}(\Delta_{\pos,n})\setminus \Delta_{\reg} \sqcup \varphi^{n+1}(\Delta_{\pos,n+1}),
\end{equation}
where the topology on the right hand side  consists of sets $U_{n}\setminus \Delta_{\reg}\sqcup U_{n+1}$ where $U_n\subseteq  \varphi^{n}(\Delta_{\pos,n})$, $U_{n+1}\subseteq \varphi^{n+1}(\Delta_{\pos,n+1}) $  are open and $\varphi^{-1}(U_{n+1})=U_n\cap \Delta_{\reg}$.
%In particular, \eqref{eq:pushout_bijection} is a homeomorphism whenever the pushout topology is Hausdorff.

%under the homeomorphism $\widehat{K}_n\cong \widehat{L^n(I_n)}$ from Proposition \ref{prop:spectrum_of_K_n}  coincides with $\varphi: \widehat{L^n(I_n)}\cap \Delta_{\reg} \to \widehat{L^{n+1}(I_{n+1})}$.
%The inverse of the above homeomorphism is
\end{lem}
\begin{proof}
Since $K_{n+1}$ is an ideal in $K_{n}+K_{n+1}$ we may identify $\widehat{K}_{n+1}\cong \varphi^{n+1}(\Delta_{\pos,n+1})$ with an open
subset of  $\widehat{K_{n}+K_{n+1}}$. Its complement is naturally identified with the spectrum of the
quotient  $K_n/(K_{n+1}\cap K_n)\cong (K_{n}+K_{n+1})/K_{n+1}$. By  Proposition \ref{prop:reduced_to_katsura},
$K_{n}\cap K_{n+1}=\overline{I_nt^n C_0(\Delta_{\reg}) t^{*n} I_n}$ is an ideal in $K_n$.
Hence using the homeomorphisms from  Proposition \ref{prop:spectrum_of_K_n} we may identify $
\widehat{K}_{n+1}$ with $\varphi^{n+1}(\Delta_{\pos,n+1}) $ and  $K_{n}\cap K_{n+1}$ with $\varphi^{n}(\Delta_{\pos,n})\cap \Delta_{\reg}$.
Accordingly, we get the  bijection \eqref{eq:pushout_bijection},
which restricts to homeomorphisms
$
\varphi^{n+1}(\Delta_{\pos,n+1})$ and $
\widehat{K_{n}+K_{n+1}}\setminus \widehat{K}_{n+1}\cong \varphi^{n}(\Delta_{\pos,n})\setminus \Delta_{\reg}
$.
Any representation $\pi$ that is in $\widehat{K_{n}+K_{n+1}}\setminus \widehat{K}_{n+1}$ is a representation of $K_{n}$ that vanishes on $K_{n+1}$.
Every ideal in $K_{n}$ is of the form $
\overline{It^{n}C_0(U)t^{*n}I}$, and the ideal in $K_n+K_{n+1}$ generated by the latter is
 $$
\overline{It^{n}C_0(U)t^{*n}I} + \overline{It^{n+1}C_0(\varphi(U\cap \Delta_{\reg})t^{*n+1}I}.
$$
Hence the bijection \eqref{eq:pushout_bijection} becomes continuous if $\widehat{L^{n}(I_{n})}\setminus \Delta_{\reg} \sqcup  \widehat{L^{n+1}(I_{n+1})}$
is equipped with  the pushout topology.
\end{proof}
The pushout topology on the right hand side of \eqref{eq:pushout_bijection}  is always $T_0$, and the continuous bijection \eqref{eq:pushout_bijection} might be a  homeomorphism even when this topology is non-Hausdorff, see \cite{Kajiwara_Watatani16} and Example \ref{ex:tent_homeomorphism_works}  below.
However, in general the pushout topology is weaker than the topology of the spectrum $\widehat{K_{n}+K_{n+1}}$, and   a general description of the topology  of the latter requires more than just the pushout data:
\begin{ex}\label{ex:tent_homeomorphism_fails}
Let us consider  $A_1=A+K_1=K_0+K_1$ associated to the  transfer operator $L(a)(y)=a(\frac{y}{2})$ for the tent map  $\p:[0,1]\to[0,1]$, $\p(x)=1-|1-2x|$.
 Then $\varrho=\mathds{1}_{[0,\frac{1}{2}]}$, $\Delta_{\pos}=[0,\frac{1}{2}]$ and $\Delta_{\reg}=[0,1/2)$.
So
%identyfying $\varphi(\Delta_{\pos})=[0,1]$ with $[\frac{1}{2},1]$, via the map $t\mapsto \frac{2-t}{2}$,
as sets we have
$$
\widehat{A}_1\cong X\setminus \Delta_{\reg}\sqcup \varphi(\Delta_{\pos})= [1/2,1]\sqcup [0,1].
$$
The pushout topology on the right hand side is the usual one with the only exception that  neighbourhoods of $1/2$ in the first summand
contain sets of the form $[ 1/2, 1/2+\varepsilon)\sqcup (1 -\varepsilon, 1)$. So in particular the pushout topology is not Hausdorff in this case (it is $T_0$ though). The topology on $\widehat{A}_1$
is larger and in fact $\widehat{A}_1$ is homeomorphic to the direct union of two closed intervals. Indeed,  the operator $tt^{*}$ in the regular representation
 becomes the multiplication operator by the characteristic function $\mathds{1}_{[0,\frac{1}{2}]}$.
So  $A_1$ is generated by $A=C[0,1]$ and  $tt^*=\mathds{1}_{[0,\frac{1}{2}]}$ and $A_1=Att^*  \oplus A(1-tt^*) =C[0,1/2]\oplus C[1/2,1]$.
The extra open set in  $\widehat{A}_1$ (not seen by the pushout topology) comes from the ideal generated by
the element $1-tt^*$ which is neither in $K_0=A$ nor in $K_1$.
Thus the precise description of  $\widehat{A}_1$ seem to require some additional algebraic data that is difficult to pin down.
\end{ex}

\begin{thm}\label{thm:spectra_of_A_n}
Let $\varrho:\Delta\to [0,+\infty)$ be a potential associated to a transfer operator $L:C_0(\Delta)\to C_0(X)$ for $\varphi:\Delta\to X$.
For each $n\in \N$ the algebra $A_n$ is postliminary
and we have a natural bijection
\begin{equation}\label{eq:pushout_bijection2}
\widehat{A}_n  \stackrel{\cong}{\longrightarrow} \left(\bigsqcup_{k=0}^{n-1}  \varphi^{k}(\Delta_{\pos,k})\setminus\Delta_{\reg} \right)\sqcup  \varphi^{n}(\Delta_{\pos,n}).
\end{equation}
More specifically, for every irreducible representation $\pi$ of $A_n$ there is a maximal $k\leq n$ with $\pi(K_k)\neq 0$ and a unique $y\in  \varphi^{k}(\Delta_{\pos,k})$
  ($y\in  \varphi^{k}(\Delta_{\pos,k})\setminus\Delta_{\reg}$ if $k<n$) such that $\pi\cong \pi_y^k$ where $\pi_y^k$ is a representation of $A_n$ on $\ell^2(\varphi^{-k}(y)\setminus \varrho_k^{-1}(0), \varrho_k)$ determined by
$$
\pi_y^k(at^{i}t^{*i}b)h=a \cdot \left(\sum_{x\in \varphi^{-i}(y)}\varrho_i(x)b(x)h(x)\right),  \qquad a,b\in I_i, \, i=1,...,k,
$$
and $\pi_y^k(K_i)=0$ for all $k<i\leq n$. If we equip the right hand side of \eqref{eq:pushout_bijection2} with the
the topology  that consists
of sets  $\left(\bigsqcup_{k=0}^{n-1}  U_k\setminus\Delta_{\reg} \right)\sqcup  U_n$ where
$U_k$ is an open subset of $ \varphi^{k}(\Delta_{\pos,k})$, for  $k=0,...,n$,  and $U_k\cap \Delta_{\reg}=\varphi^{-1}(U_{k+1})$ for $k<n$,
then \eqref{eq:pushout_bijection2}
 is continuous and its inverse is continuous when restricted to each direct summand.
\end{thm}
\begin{proof} 
We prove this by induction. The assertion holds for $n=1$ by Lemma \ref{lem:pushout_K_n}.
Assume that for certain $n$  we have $\widehat{A}_n \cong \left(\bigsqcup_{k=0}^{n-1}   \varphi^{k}(\Delta_{\pos,k})\setminus\Delta_{\reg} \right)\sqcup   \varphi^{n}(\Delta_{\pos,n})$  as in the assertion.  Here $\bigsqcup_{k=0}^{n-1}  \varphi^{k}(\Delta_{\pos,k})\setminus\Delta_{\reg} $ corresponds to the closed set $ \widehat{A}_n\setminus \widehat{K}_{n}$ and
$\widehat{K}_{n}\cong \varphi^{n}(\Delta_{\pos,n})$ is the homeomorphism from Proposition \ref{prop:spectrum_of_K_n}.

By Proposition \ref{prop:reduced_to_katsura},  $K_{n}+K_{n+1}$ is an ideal in $A_{n+1}$.  The corresponding open subset of  $\widehat{A}_{n+1}$ is
$\widehat{K_{n}+K_{n+1}}\cong\widehat{L^{n}(I_{n})}\setminus \Delta_{\reg} \sqcup  \widehat{L^{n+1}(I_{n+1})}$ as described in  Lemma \ref{lem:pushout_K_n}.
Its complement    $\widehat{A}_{n+1}\setminus \widehat{K_{n}+K_{n+1}}\cong \widehat{A}_n\setminus \widehat{K}_{n}\cong \left(\bigsqcup_{k=0}^{n-1}  \widehat{L^k(I_k)}\setminus\Delta_{\reg} \right)$.
 Since  $A_{n}\cap K_{n+1}=K_{n}\cap K_{n+1}$, see Proposition  \ref{prop:reduced_to_katsura},  we conclude that  the topology on
$\widehat{A}_{n+1}\cong \left(\bigsqcup_{k=0}^{n}  \widehat{L^k(I_k)}\setminus\Delta_{\reg} \right)\sqcup  \widehat{L^{n+1}(I_{n+1})}$
is as described in the assertion. The ranges of all representations in $\widehat{A}_n$ contain compact operators. Hence $A_n$ is postliminary.
\end{proof}
\begin{rem} If $\pi$ is an irreducible representation  of $A_n$, there is a `dynamical procedure' of determining $y$ and $k$ for which $\pi\cong \pi_y^k$. Namely,
the set $Z:=\{x\in X: a(x)\neq 0 \text{ implies } \pi(a)\neq 0\text{ for all }a\in A\}$ is closed and there is $k\leq n$ such that $\varphi^{k}(Z)$ is a singleton. If there is a  minimal $k< n$ such that
$\varphi^{k}(Z)=\{y\}\notin \Delta_{\reg}$, then $\pi\cong \pi_y^k$. Otherwise $\pi\cong \pi_y^n$  where $\varphi^{n}(Z)=\{y\}$.
\end{rem}
Example \ref{ex:tent_homeomorphism_fails} shows that the continuous bijection \eqref{eq:pushout_bijection2} in general fails to be a homeomorphism. Obviously, it is a homeomorphism
when the pushout topology on the right hand side of \eqref{eq:pushout_bijection2} is Hausdorff, and less obviously, when $\varrho$ is continuous, see Theorem \ref{thm:primitive_groupoid} below.
This may also happen in a non-continuous and non-Hausdorff case:
\begin{ex}\label{ex:tent_homeomorphism_works}
Let us consider  the standard transfer operator $L(a)(y)=\frac{1}{2}[a(\frac{y}{2})+a(1-\frac{y}{2})]$ for the tent map  $\p:[0,1]\to[0,1]$, $\p(x)=1-|1-2x|$. Then $
\varrho=\frac{1}{2}\mathds{1}_{X\setminus \{\frac{1}{2}\}}+\mathds{1}_{\{\frac{1}{2}\}}
$, $\Delta_{\pos}=X=[0,1]$ and  $\Delta_{\reg}=X\setminus \{\frac{1}{2}\}=[0,\frac{1}{2})\cup(\frac{1}{2},1]$. Accordingly,
$$
\widehat{A}_n\cong \bigcup_{k=0}^{n-1}\{\pi_{1/2}^k\}\cup\{\pi_{x}^n:x\in [0,1]\}
$$
where the pushout topology on the right hand side can be described as follows:
$\{\pi_{x}^n:x\in [0,1]\}$ is an open set homeomorphic to $[0,1]$ and each $\pi_{1/2}^k$ has a basis of
neighbourhoods of the form $\{\pi_{1/2}^k\}\cup\{\pi_{x}^n:x\in (0,\varepsilon)\}$, if $k<n-1$, and
$\{\pi_{1/2}^{n-1}\}\cup\{\pi_{x}^n:x\in (1-\varepsilon,1)\}$, if $k=n-1$ (so  $\pi_{1/2}^k$, $k<n-1$, cannot be separated from
$\pi_{0}^n$ and  $\pi_{1/2}^{n-1}$ can not be separated from
$\pi_{1}^n$). This topology coincides  with the standard topology of $\widehat{A}_n$,
as  using the regular representation one can see that $A_{n}$
is naturally isomorphic  the $C^*$-subalgebra of $C([0,1],M_{2^{n}}(\C))$ consisting continuous matrix valued functions $a$ satisfying
$$
  a(1)\in M_{2^{n-1}}(\C)\oplus M_{2^{n-1}}(\C),\qquad
  a(0)\in M_{2^{n-1}+1}(\C)\oplus M_{2^{n-2}}(\C)  \oplus ...\oplus M_2(\C) \oplus \C.
$$%	\end{align*}
In particular, representations
$\pi_1^n$ and $\pi_{1/2}^{n-1}$ are of dimension $2^{n-1}=|\varphi^{-n}(1)|$, and $\pi_0^n$ is of dimension  $2^{n-1}+1=|\varphi^{-n}(0)|$.
This example is covered by the main result of \cite{Kajiwara_Watatani16}.
\end{ex}
The algebra $A_\infty$ as a rule is not  postliminary (the example of Glimm algebras shows that  $A_{\infty}$ will usually be antiliminary, cf. \cite[Theorem 6.5.7]{Pedersen}).
Accordingly, one can not hope to describe the spectrum $\widehat{A}_{\infty}$ completely in a reasonable way. However,
the inductive limit of spaces $\widehat{A}_{n}$ will give a dense subset of $\widehat{A}_{\infty}$, and  when the continuous maps \eqref{eq:pushout_bijection2}
are opens one can use them to describe the \emph{primitive ideal space}  $\Prim(A_{\infty})$   of $A_{\infty}$. 
We show how to do it in the case when $\varrho$ is continuous.

%\begin{ex}[$C^*$-algebras generated by weighted shifts on directed trees]
%Let $\varphi:\Delta\to X$ be a partial map on countable discrete set $X$ and assume that
%$\varphi$ has no periodic points. This can be identified with a forest of directed trees in the
%sense of \cite{Stochel}. Namely, the directed graph $(E^0,E^1,r,s)$ where $E^0=X$, $E^1:=\{(\varphi(x),x):x\in X\}$,
%$s(\varphi(x),x)=x$ and $r(\varphi(x),x)=\varphi(x)$, decomposes into connected components  $E^0=\bigsqcup_{i} X_i$, in the sense of \cite{Stochel}. Each
% of the components $X_i$ is a directed tree.
%.......
%\end{ex}

\subsection{The case of a continuous potential}
When $\varrho$ is continuous then $A_{\infty}$ has a natural \'etale groupoid model.
The groupoid in question is a motivating example in the theory of approximately proper equivalence relations
\cite{Renault05}, which recently has been generalized to cover partial local homeomorphisms \cite{Bissacot-Exel-Frausino-Raszeja2}.
Namely, assume that $\varphi:\Delta\to X$ is a local homeomorphism.
For each $n\in \N$  consider the equivalence relation
$$
R_n:=\{(x,y)\in \Delta_{n}\times \Delta_n: \varphi^{n}(x)=\varphi^{n}(y)\}
$$
as an \'etale groupoid with the product topology inherited from $\Delta_{n}\times \Delta_n$.
Then we get a  generalized approximately proper (in short \textsc{gap}) equivalence relation on $X\times X$
$$
R:=\bigcup_{n=0}^\infty R_n,
$$
equipped with the \emph{inductive limit topology}, i.e. $U\subseteq R$ is open iff $U\cap R_n$ is open for every $n\in \N$,
see \cite[Proposition 5.6]{Bissacot-Exel-Frausino-Raszeja2}. Similarly, each finite union $\bigcup_{k=0}^n R_k$ becomes an \'etale groupoid.
All these groupoids are amenable (see the proof of \cite[Proposition 2.4]{Renault2000}).
\begin{prop}\label{prop:local_homeo_core_groupoid}
 Assume $\varphi:\Delta\to X$ is a local homeomorphism and let $\varrho:\Delta\to (0,+\infty)$ be any strictly positive continuous map  with
$\sup_{y\in X} \sum_{x\in \varphi^{-1}(y)}\varrho(x)<\infty$.
%, and then \eqref{eq:transfer_identity} is automatically satisfied).
Equivalently, fix a  transfer operator $L:C_0(\Delta)\to C_0(X)$ with continuous potential $\varrho >0$. We have a natural isomorphisms
$$
K_{n}\cong C^*(R_n),\quad   A_{n}\cong C^*(\bigcup_{k=0}^n R_k), \quad n\in \N, \qquad A_{\infty}\cong C^*(R),
$$
where $K_n:=\overline{I_nt^{n}t^{*n}I_n}$, $A_n:=K_0+...+K_n$,  $A_{\infty}=\overline{\bigcup_{n=0}^\infty A_n}$ are core subalgebras of $A\rtimes L$, and
$R$ is the \textsc{gap} relation associated to $\varphi$. 
\end{prop}
\begin{proof}
We  view  $R$ and $\bigcup_{k=0}^n R_k$ as open subgroupoids of $\G$ with the same unit space $X$.
Then the inclusions $C_c(\bigcup_{k=0}^n R_k)\subseteq  C_c(R)\subseteq C_c(\G)$ extend to $*$-homomorphisms
$
C^*(\bigcup_{k=0}^n R_k)\to  C^*(R)\to C^*(\G)$
that intertwine the canonical faithful conditional expectations onto $C_0(X)$ (the groupoids in question are amenable).
Hence the aforementioned  $*$-homomorphisms are faithful and we may write $C^*(\bigcup_{k=0}^n R_k)\subseteq C^*(R)\subseteq C^*(\G)$.
Now it is immediate that the isomorphism from Theorem \ref{thm:local_homeo_crossed_groupoid} restricts
to isomorphisms  $ C^*(\bigcup_{k=0}^n R_k)\cong A_{n}$,   $C^*(R)\cong A_{\infty}$.

The groupoid $R_n$ can be viewed as the restriction of $\bigcup_{k=0}^n R_k$ to an open invariant subset $\Delta_n\subseteq X$.
Hence $C^*(R_n)$ can be identified with an ideal in  $C^*(\bigcup_{k=0}^n R_k)$ generated by $C_0(\Delta)$, see \cite[Proposition 4.3.2]{Sims}.
Restriction of the isomorphism $ C^*(\bigcup_{k=0}^n R_k)\cong A_{n}$  gives $ C^*(R_n)\cong K_{n}$.
\end{proof}
\begin{rem}\label{rem:Wieler_solenoids}
By \cite{Wieler}, all  irreducible Smale spaces $(\tilde{X},\tilde{\varphi})$ with
totally disconnected stable sets are  inverse limits for certain finite-to-one  continuous surjections $\varphi:X\to X$ satisfying Wieler's axioms. By Proposition \ref{prop:local_homeo_core_groupoid} and \cite[Theorem 5.6]{Deeley},
if $\varphi$ is an open Wieler map,  then the stable algebra $S$ and the stable Ruelle algebra $R_{s}$ of the Smale space $(\tilde{X},\tilde{\varphi})$, cf. \cite{Putnam-Spielberg},  are  Morita-equivalent  to the algebras $A_\infty$ and $A\rtimes L$, respectively.
\end{rem}
The groupoids $R_n$,  $\bigcup_{k=0}^n R_k$, $R$ are not only amenable but also \emph{principle} (all the isotropy groups are trivial). In particular,
inclusions $C_0(X)\subseteq C^*(\bigcup_{k=0}^n R_k)\cong A_{n},  C^*(R)\cong A_{\infty}$ are \emph{$C^*$-diagonals} in the sense of Kumjian \cite{kumjian},
and the ideals in  $A_{n}$ and $A_{\infty}$ correspond to open invariant sets in $\bigcup_{k=0}^n R_k$ and $R$, respectively (see
\cite[Theorem 4.3.3]{Sims} or \cite[Corollary 3.12]{BL}). Also  the primitive ideal spaces can be identified with quasi-orbit spaces, see \cite[Corollary 3.19]{BL} or \cite[Theorem 7.17]{Kwa-Meyer1}.
This allows us to improve   Theorem \ref{thm:spectra_of_A_n} in the case when $\varrho$ is continuous,  as follows.

\begin{thm}\label{thm:primitive_groupoid} If  $\varrho:\Delta\to [0,+\infty)$ is continuous, then the continuous bijection in \eqref{eq:pushout_bijection2}
is a homeomorphism, and we have  natural homeomorphisms
$$
\widehat{A}_n \cong \Prim (A_n) \cong X/\sim_n .
$$
where $x\sim_n y$ iff  there is $k\leq n$ such that $x,y\in \Delta_{\pos,k}$ and $\varphi^{k}(x)=\varphi^k(y)$.
If in addition $X$ is second countable, we have a homeomorphism
$$
\Prim (A_\infty) \cong X/\sim
$$
where $
x\sim y
$ iff $\overline{\OO_{R}(x)}=\overline{\OO_{R}(y)}$
 and $\OO_{R}(x):=\bigcup_{k=0, x\in \Delta_{\pos, k}}^\infty \varphi^{-k}(\varphi^{k}(x))$ is the orbit of $x\in X$.
\end{thm}
\begin{proof} 
Since $\varrho$ is continuous we may  assume that $\Delta$ is equal to $\Delta_{\pos}=\Delta_{\reg}$, so that $\varrho>0$ and we may apply
isomorphisms from Proposition \ref{prop:local_homeo_core_groupoid}.
The orbits of $x$ for the groupoid $R_{[0,n]}:=\bigcup_{k=0}^n R_k$ are given by $\OO_n(x)= \bigcup_{k=0, x\in \Delta_{\pos, k}}^n \varphi^{-k}(\varphi^{k}(x))$.
We have a bijection
$$
\left(\bigsqcup_{k=0}^{n-1}  \varphi^{k}(\Delta_{\pos,k})\setminus\Delta_{\reg} \right)\sqcup  \varphi^{n}(\Delta_{\pos,n})\stackrel{\cong}{\longrightarrow} X/\sim_n
$$
that sends a point $y$ in $k$-th summand to $\varphi^{-k}(y)$, which is an $R_{[0,n]}$-orbit. Using this bijection one checks that open $R_{[0,n]}$-invariant subsets of $X$ correspond to open sets in the pushout topology  of the right-hand side of \eqref{eq:pushout_bijection2}.
Since ideals in $A_{n}\cong C^*(\bigcup_{k=0}^n R_k)$ correspond bijectively to open $R_{[0,n]}$-invariant sets, this gives the first part of the assertion (we have $\widehat{A}_n \cong \Prim (A_n)$ because $A_n$ is postliminary, in fact Type $I_0$).

The second part follows because
 $\OO_{R}(x)$ is the orbit of $x$ under the groupoid  $R$  and  the primitive ideal space $\Prim (A_\infty)\cong \Prim (C^*(R))$ is homeomorphic to the quasi-orbit space for $R$ by  \cite[Corollary 3.19]{BL}.
\end{proof}

\section{Topological free transfer operators and simplicity}\label{Sec:Topological_freeness}

A full map on locally compact Hausdorff space is called topologically free if the set of its periodic points has empty interior. 
%The importance of this dynamical characteristics in the analysis of the corresponding crossed products  has been known probably from the work by Zeller-Meier \cite{Zeller-Meier}.
We will introduce topological freeness for a partial  map $\varphi:\Delta\to X$ by reducing it to the case of a full map.
Namely, we will restrict $\varphi$ to its \emph{essential domain} $\Delta_{\infty}:=\bigcap_{n=1}^\infty \Delta_n \cap \varphi^n(\Delta_n)$, $\Delta_n=\varphi^{-n}(\Delta)$, $n\in \N$, which gives
a full  map $\varphi:\Delta_{\infty}\to \Delta_{\infty}$, see \cite[Definition 3.1]{kwa-leb}.
As a starting step of defining topological freeness for transfer operators we analyse this notion for open partial maps.

\begin{defn}
A partial continuous open map  $\varphi:\Delta\to X$  is \emph{topologically free} if the set of periodic points for $\varphi:\Delta_{\infty}\to \Delta_{\infty}$
has empty interior in $\Delta_\infty$.
\end{defn}

\begin{lem}\label{lem:topol_free_local_homeo}
Suppose that $\varphi:\Delta\to X$ is   a partial continuous open  map of $X$. The following
conditions are equivalent:
\begin{enumerate}
\item\label{enu:topol_free_local_homeo1} $\varphi$ is topologically free; % the set periodic points for $\varphi:\Delta_{\infty}\to \Delta_{\infty}$ has empty interior in $\Delta_\infty$;
\item\label{enu:topol_free_local_homeo2} for every $n\in \N$ the set $ \{x\in \Delta_n: \varphi^n(x)=x\}$  has empty interior in $X$;
\item\label{enu:topol_free_local_homeo3}  for  every  $k,l\in\N_0$ with $l< k $,  $
\{x\in \Delta_{k}: \varphi^k(x)=\varphi^l(x)\}
$
has empty interior in $X$.

\end{enumerate}

\end{lem}
\begin{proof} Note that $\{x\in \Delta_n: \varphi^n(x)=x\}\subseteq \Delta_\infty$ is  closed in $\Delta_{\infty}$, becasue $\Delta_{\infty}$ is Hausdorff and $\varphi$ is continuous. Moreover, $\Delta_{\infty}$ is a Baire space, as it is a $G_\delta$ subset of the locally compact Hausdorff space $X$.
Thus $\{x\in \Delta_{\infty}:\exists_{n\in \N}\, \varphi^n(x)=x\}=\bigcap_{n\in \N}\{x\in \Delta_n: \varphi^n(x)=x\}$ has empty interior if and only if each of the intersected sets has empty interior. This proves
\ref{enu:topol_free_local_homeo1}$\Leftrightarrow$\ref{enu:topol_free_local_homeo2}.

The implication \ref{enu:topol_free_local_homeo2}$\Rightarrow$\ref{enu:topol_free_local_homeo3} is immediate. For the converse assume that there is a
non-empty open set $U\subseteq \{x\in \Delta_{k}: \varphi^k(x)=\varphi^l(x)\}
$ where $l< k $. Then $V:=\varphi^l(U)$ is non-empty open set contained in $\{x\in \Delta_n: \varphi^n(x)=x\}$ where $n:=k-l\in \N$.
\end{proof}

If we consider a transfer operator $L$ associated to a partial map $\varphi: \Delta\to X$, then the natural `domain of openness'
 for $\varphi$ is $\Delta_{\pos}$, see  Proposition~\ref{prop:properties_or_rho}. The next lemma shows that in the definition of topological freeness we can equally-well use the smaller set $\Delta_{\reg}$.

%\begin{rem}\label{rem:topolo_for_transfers}
%Let  $\varrho:\Delta\to [0,+\infty)$ be a potential associated to a transfer operator $L$. By Lemma \ref{lem:topol_free_local_homeo},
%the transfer operator $L$ is topologically free if for every $n\in\N$ the set $\{x\in \Delta_n: x,\varphi(x),...,\varphi^{n-1}(x)\subseteq \Delta_{\reg} \text{ and } \varphi^n(x)=x\}$ has empty interior.
%\end{rem}
\begin{lem}\label{lem:reg_vs_positive_freeness}
If  $L$ is a transfer operator for a map $\varphi:\Delta\to X$, then  $\varphi: \Delta_{\reg}\to X$ is topologically free if and only if $\varphi:\Delta_{\pos}\to X$ is topologically free. % in the sense that for every $n\in\N$ the set $\{x\in \Delta_{\pos,n}: \varphi^n(x)=x\}$ has empty interior.
\end{lem}
\begin{proof}
Since $\Delta_{\reg}\subseteq \Delta_{\pos}$, topological freeness of $\varphi:\Delta_{\pos}\to X$ implies topological freeness
of $\varphi: \Delta_{\reg}\to X$.
 Assume now  that $\varphi:\Delta_{\pos}\to X$ is not topologically free. So  there is a non-empty open set
$U\subseteq \Delta_{\pos,n}=\Delta_n\setminus \varrho_n^{-1}(0)$, for some $n\in \N$, such that for each  $x\in U$  we have $x=\varphi^{n}(x)$.
Then $\varphi$ is injective on each of the sets $U$, $\varphi(U)$, ...,  $\varphi^{n-1}(U)$
 and since they are contained in
$\Delta_{\pos}=\Delta\setminus \varrho^{-1}(0)$ it follows from Proposition \ref{prop:properties_or_rho} that they are in fact contained in $\Delta_{\reg}$.
Therefore $\varphi: \Delta_{\reg}\to X$ is not topologically free.
\end{proof}
The foregoing observations naturally lead to the following definition, which  agrees with the version of topological freeness suggested in \cite[Example 9.14]{CKO}.
\begin{defn}
 We say that the \emph{transfer operator $L$ is topologically free} if the restricted open map
$\varphi: \Delta_{\reg}\to X$ is topologically free.
\end{defn}

\begin{thm}%[Uniqueness theorem]
\label{thm:isomorphism}
Let $L$ be a transfer operator for a partial map $\varphi:\Delta\to X$. The following conditions are equivalent:
\begin{enumerate}
\item\label{enu:isomorphism1}  Every faithful covariant  representation $(\pi,T)$ of $L$ extends to a faithful representation
$\pi\rtimes T$ of the crossed product  $A\rtimes L$.

\item\label{enu:isomorphism1.5}  $A$ detects ideals in $A\rtimes L$, i.e. $A\cap N\neq \{0\}$ for any non-zero ideal $N$ in $A\rtimes L$.

\item\label{enu:isomorphism2}  The orbit representation $(\pi_o, T_o)$ introduced in Example~\ref{ex:orbit_representation} extends to a faithful representation
$\pi_o\rtimes T_o$ of the crossed product  $A\rtimes L$.

\item\label{enu:isomorphism3}  The map $\varphi: \Delta_{\reg}\to X$ is topologically free.
\end{enumerate}
\end{thm}
\begin{proof}
The equivalence \ref{enu:isomorphism1}$\Leftrightarrow$\ref{enu:isomorphism1.5} is straightforward and implication \ref{enu:isomorphism1}$\Rightarrow$\ref{enu:isomorphism2} is trivial. To prove
\ref{enu:isomorphism2}$\Rightarrow$\ref{enu:isomorphism3} assume $\varphi: \Delta_{\reg}\to X$ is not  topologically free. Then there is a non-empty open set  $U\subseteq  \Delta_n$, such that $\varphi^n|_U=id|_U$ and $\varphi^{k}(U)\subseteq X_{\reg}$
for $k=0,...,n-1$. In particular, $\varrho_n$ is continuous and non-zero at every point in $U$. Thus  for any non-zero $a\in C_c(U)\subseteq A$ we have $a\sqrt{\varrho_n}\in A$. Using the regular representation one readily calculates that $at^n-a\sqrt{\varrho_n}$ is a non-zero element of $A\rtimes L\cong C^*(\widetilde{\pi}(A)\cup \widetilde{\pi}(I)\widetilde{T})$, but
$(\pi_o\rtimes T_o)(at^n-a\sqrt{\varrho_n})=0.
$ Hence  $\pi_o\rtimes T_o$ is not faithful.

Implication  \ref{enu:isomorphism3}$\Rightarrow$\ref{enu:isomorphism1}  follows from \cite[Example 9.14]{CKO} as by Lemma \ref{lem:reg_vs_positive_freeness} condition described there 
is equivalent to topologicall freeness as we define it.
\end{proof}

% \begin{ex}
% We put $X=[0,1)$ $(mod \,\,1)$, $\p(x)=2x$ $(mod\,\,1)$ and
% $$\varrho(x)=2|x-1/2|.
%$$
% \end{ex}
Topological freeness for homeomorphisms appeared already in the work of
Zeller-Meier, see \cite[Proposition 4.14]{Zeller-Meier},
  who used it to characterise when $C_0(X)$ is maximal abelian in the associated crossed product.
This result was generalised to crossed products by local homeomorphisms by Carlsen and Silvestrov in \cite{CS}.
However, it seems that there is no obvious generalisation of this fact if we allow irregular points (discontinuity points of $\varrho$).
%\subsection{Commutant of $C_0(X)$ in $C_0(X)\rtimes L$}
More specifically, let $A'$  denote the commutant of $A=C_0(X)$ in $A\rtimes L$. So $A$ is maximal abelian in $A\rtimes L$  iff $A=A'$. Using the generalised expectation $G$ introduced in Lemma~\ref{lem:generalised_expectation} we clearly have
$$
C_0(X)\subseteq \{b\in A\rtimes L: G(b)=b\}\subseteq C_0(X)'.
$$
It turns out that when $\varrho$ is discontinuous already the first inclusion might be proper.
\begin{ex}
For the tent map  $\p:[0,1]\to[0,1]$ where  $\p(x)=1-|1-2x|$ and  $\varrho=\mathds{1}_{[0,\frac{1}{2}]}$ we have the transfer operator $L(a)(y)=a(\frac{y}{2})$.
Using the regular representation one readily calculates that
$$
G(t^nt^{*n})=t^nt^{*n}=\mathds{1}_{[0,\frac{1}{2^n}]}.
$$
Accordingly, $C_0(X)\subsetneq \{b\in A\rtimes L: G(b)=b\}$ as the latter contains some functions that
are discontinuous at points $\frac{1}{2^n}$, $n\in \N$. Hence $C_0(X)\neq C_0(X)'$ even though the map $\p$ is topologically free.
\end{ex}

For the sake of completeness we will generalise  the main result of \cite{CS} (to partial, not necessarily surjective maps on locally compact spaces,  and arbitrary continuous $\varrho$).
The proof is based on Renault's characterisation of Cartan subalgebras \cite{Re}, see also \cite[7.2]{KwaMeyer}.

\begin{thm}\label{thm:commutant_topological_freeness}
Suppose that the transfer operator $L$ is given by a continuous  $\varrho$.  Then the equivalent conditions  in Theorem \ref{thm:isomorphism}
are further equivalent to
each of the following:
\begin{enumerate}
\item The $C^*$-algebra $C_0(X)$ is maximal abelian in $C_0(X)\rtimes L$.
\item $C_0(X)$ is a Cartan subalgebra of $C_0(X)\rtimes L$, in the sense of \cite{Re}.
%\item $\varphi: \Delta\to X$ is $\varrho$-topologically free
\item The partial map $\varphi:\Delta_{\reg}=\Delta_{\pos}\to X$ is topologically free.

\end{enumerate}
\end{thm}
\begin{proof}
The transfer operator $L:I\to A$ for $\varphi:\Delta\to X$ restricts to the transfer operator $L_{\reg}:C_0(\Delta_{\reg})\to A$ for
the partial homeomorphism $\varphi:\Delta_{\reg}\to X$. Since  we assume $\varrho$ is continuous we get
$\Delta_{\reg}=\Delta\setminus \varrho^{-1}(0)$ and the crossed products $A\rtimes L$ and $A\rtimes L_{\reg}$ are naturally isomorphic (their regular representations
coincide, see Proposition \ref{prop:regular_is_faithful}). Thus by Theorem \ref{thm:local_homeo_crossed_groupoid} we may identify
$A\rtimes L$ with the groupoid $C^*$-algebra  $C^*(\G)$ of  the Renault-Deaconu groupoid \(\G\) associated to $\varphi:\Delta_{\reg}\to X$.
By the work of Renault's \cite{Re}, $C_0(X)$ is maximal abelian in $C^*(\G)$ if and only if $C_0(X)$ is a Cartan subalgebra of \(\G\) if and only if the groupoid
\(\G\) is effective (Renault considered second countable groupoids but his theory works without this assumption, see \cite[7.2]{KwaMeyer})
By \cite[Corollary 3.3]{Re} and \cite[Proposition 2.3]{Renault2000}  the groupoid \(\G\) is effective if and only if the map $\varphi:\Delta_{\reg}\to X$  is topologically free
(local homeomorphisms satisfying \ref{enu:topol_free_local_homeo3} in Lemma \ref{lem:topol_free_local_homeo} are called essentially free in
\cite{Renault2000}). This proves the desired equivalence.
\end{proof}

\subsection{Simplicity}
The following definition and lemma are compatible with \cite[Definition 3.1 and Proposition 3.2]{er} where partial local homeomorphisms are considered.

\begin{defn}\label{defn:invariant_sets} Let $U\subseteq X$. We say $U$  is \emph{positively invariant} if $\varphi(U\cap \Delta_{\pos})\subseteq U$,  $U$ is \emph{negatively invariant}
if $\varphi^{-1}(U)\cap \Delta_{\reg}\subseteq U$,  and $U$ is \emph{invariant} if it is both positively and negatively invariant.
We say that $L$ is \emph{minimal} if there are no non-trivial open invariant sets.
\end{defn}
\begin{lem}\label{lem:restricted_ideals} Let $U$ be an open subset of $X$ and put $J:=C_0(U)$ and recall that $I=C_0(\Delta)$.% be the corresponding ideal in $A=C_0(X)$. The following conditions are equivalent

\begin{enumerate}
\item\label{item:restricted_ideals1} $U$ is positively invariant iff $L(J\cap I)\subseteq J$;
\item\label{item:restricted_ideals2}   $U$ is negatively invariant iff $L^{-1}(J)\cap C_c(\Delta_{\reg})\subseteq J$ iff $L^{-1}(J)\cap C_0(\Delta_{\reg})\subseteq J$;
\end{enumerate}
\end{lem}
\begin{proof}
\ref{item:restricted_ideals1}. If $U$ is positively invariant, then for any $y\not\in U$ we have $\varphi^{-1}(y)\cap \Delta_{\pos}=\emptyset$, which implies
$L(a)(y)=0$ for any $a\in J\cap I$. Thus $L(J\cap I)\subseteq J$.
If $U$ is not positively invariant, then there is $x\in U\setminus\varrho^{-1}(0)$ such that $\varphi(x)\notin U$.
Taking any positive  $a\in I\cap J$ with $a(x)\neq 0$, we get
 $L(a)(\varphi(x))=\sum_{t\in\varphi^{-1}(\varphi(x))}\varrho(t)a(t)>0$ which  shows that $L(J\cap I)\not\subseteq J$.

\ref{item:restricted_ideals2}. This follows from the equalities $C_c(\varphi^{-1}(U)\cap \Delta_{\reg})=L^{-1}(J)\cap C_c(\Delta_{\reg})$ and $C_0(\varphi^{-1}(U)\cap \Delta_{\reg})=L^{-1}(J)\cap C_0(\Delta_{\reg})$, which are readily verified.
\end{proof}
\begin{ex}[Directed graphs with one circuit]\label{ex:one_circuit_graph} Suppose that $\varphi:X\to X$ is a \emph{directed graph with one circuit}, i.e. $X$ is a countable  discrete set and there is a point $x\in X$ such that for any $y\in X$ we have $\varphi^{n}(y)=x$ for some $n\geq 1$, cf. \cite{BJJS}. In particular, $x$ is a periodic point and $\varphi$ is a local homeomorphism. Take any $\varrho:X \to (0,+\infty)$ 
with $\sup_{y\in X}\sum_{x\in \varphi^{-1}(y)}\varrho(x)<\infty$, so that it defines a transfer operator $L$ for $\varphi$.
Then $X=\Delta_{\pos}=\Delta_{\reg}$ and the transfer operator $L$ is minimal but $\varphi$ is not topologically free.
Hence, $A\rtimes L$ is not simple   by Theorem \ref{thm:isomorphism}.

\end{ex}
\begin{thm}\label{thm:test_of_simplicity} The following conditions are equivalent:
\begin{enumerate}
\item\label{thm:test_of_simplicity1} the crossed product $A\rtimes L$ is simple;
\item\label{thm:test_of_simplicity2} $L$ is minimal and $\varphi: \Delta_{\reg}\to X$ is topologically free;
\item\label{thm:test_of_simplicity3} $L$ is minimal and $\varphi: \Delta_{\reg}\to X$ is  not a directed graph with one circuit. %, see Example~\ref{ex:one_circuit_graph}.
%\item\label{thm:test_of_simplicity4} $L$ is minimal and there is no isolated periodic orbit in $\Delta_\reg$.
\end{enumerate}
% if and only if
% and there are no non-trivial open subsetes of $X$ which are invariant in the sense of Definition \ref{defn:invariant_sets}.
\end{thm}
\begin{proof} We first show that \ref{thm:test_of_simplicity1}$\Leftrightarrow$\ref{thm:test_of_simplicity2}.
If  $\varphi: \Delta_{\reg}\to X$ is not topologically free, then  $A\rtimes L$ is not simple by Theorem \ref{thm:isomorphism}.
So let us assume that $\varphi: \Delta_{\reg}\to X$ is  topologically free. Then for any non-zero ideal $N$ in $A\rtimes L$ the ideal  $J:=A\cap N$ in $A$ is non-zero. Hence
$J=C_0(U)$ for some non-empty open set $U$. Note that
$
L(J\cap I)= t^* I NI t\subseteq K
$, so $L(J\cap I)\subseteq J$. Also  for any $a\in L^{-1}(J)\cap C_c(\Delta_{\reg})$, using \eqref{eq:Cuntz_relation_K2}, we obtain
\begin{align*}
a&=\sum_{i,j=1}^nu_i^Ktt^*  u_i^K  a u_j^K tt^*  u_j^K =\sum_{i,j=1}^n  u_i^K  t  L(u_i^Kau_j^K)   t^*   u_j^K \\
&=\sum_{i,j=1}^nu_i^K  t L\big(\alpha(u_i^K\circ\varphi|_{U_i}^{-1})a\alpha(u_j^K\circ\varphi|_{U_j}^{-1})\big)t^*u_j^K\\
&=\sum_{i,j=1}^nu_i^K  t (u_i^K\circ\varphi|_{U_i}^{-1})  L(a ) (u_j^K\circ\varphi|_{U_j}^{-1}) t^*u_j^K\in \overline {I t J t^* I}\subseteq  N,
\end{align*}
 so $a\in J$. Hence $U$ is an invariant set by Lemma \ref{lem:restricted_ideals}. If $U\neq X$ then $N\neq A\rtimes L$ and $A\rtimes L$ is not simple.
Conversely, if $N\neq A\rtimes L$ then $U\neq X$ because otherwise $N$ would contain $A$ and $A\rtimes L= A (A\rtimes L)$ would be $N$.

Implication \ref{thm:test_of_simplicity2}$\Rightarrow$\ref{thm:test_of_simplicity3} is  clear and to prove  the converse %\ref{thm:test_of_simplicity4}$\Rightarrow$\ref{thm:test_of_simplicity2}
assume that $L$ is minimal but $\varphi: \Delta_{\reg}\to X$
is not topologically free. Then there is a non-empty open set $U\subseteq \Delta_{\reg, n}$ for some $n\geq 1$ such that $\varphi^n|_{U}=id$ and $U, \varphi(U), ..., \varphi^{n-1}(U)$ are pairwise disjoint. For any disjoint open sets $V_1$, $V_2\subseteq U$ the sets $U_i:=\bigcup_{m\in \N, k=0,...,n}\varphi^{-m}(\varphi^{k}(V_i))$ are disjoint open and invariant. Thus minimality of $L$ forces
$U=\{x\}$ to be a singleton and $\varphi: \Delta_{\reg}\to X$ to be a directed graph with one circuit.
\end{proof}
\begin{rem}\label{rem:regular_graphs_minimality} 
If $\Delta=\Delta_{\reg}$, then $E:=(X,\Delta, \id, \varphi)$ is a topological graph in the sense of Katsura \cite{ka1} and Theorem  \ref{thm:test_of_simplicity} 
 could be deduced from \cite[Theorem 8.12]{ka3}.  In this regular case, one could also get simplicity criteria for  $A\rtimes L$  using the Renault-Deaconu groupoid model, see \cite{Kwa-Meyer} and references therein.
\end{rem}

\section{Locally contractive transfer operators and pure infiniteness}\label{Sec:simplicity_pure_infiniteness}

%In this final section we analyse simplicity and pure infiniteness of the crossed product $A\rtimes L$ for
%the fixed transfer operator $L$ for a partial map $\varphi:\Delta\to X$.

The following definition is inspired by  \cite[Definition 2.7]{ka4}.
\begin{defn}\label{defn:contractive}
We say that an open  set $V\subseteq X$ is \emph{contracting }if there are pairwise disjoint, non-empty open sets $U_k\subseteq \Delta_{\reg, n_k}\cap V$ for $k=1,...,m$,  $n_k\geq 1$, such that
$$
V\not\subseteq \overline{\bigcup_{k=1}^m U_k}\quad \text{and}\quad
\overline{V}\subseteq \bigcup_{k=1}^m \varphi^{n_k}(U_k).
$$
We say that $L$ is \emph{contracting} if $\Delta=\Delta_{\pos}$ and there is $x_0\in \Delta$ such that
every neighbourhood of $x_0$  contains a  contracting open set and $\overline{\bigcup_{n=0}^{\infty}\varphi|_{\Delta_{\reg}}^{-n}(x_0)}=X$.
\end{defn}
%\begin{rem}
If $\Delta=\Delta_{\reg}$, then  $L$ is  contractive iff
the topological  graph $E:=(X,\Delta, \id, \varphi)$ is contractive in the sense of \cite{ka4}, and we could use \cite[Theorem A]{ka4} to show that  $A\rtimes L$  is purely infinite simple whenever $L$ is minimal and  contractive.
A formally weaker result could be obtained using a groupoid model and \cite[Proposition 2.4]{Anantharaman-Delaroche}, as  $L$ is contractive
if the Renault-Deaconu groupoid $\G$ is minimal and locally contractive in the sense of \cite{Anantharaman-Delaroche}, but the converse is not clear.
 We will now prove a general result allowing irregular points, that is admitting the case when  $\Delta_{\pos}\neq \Delta_{\reg}$.
%\end{rem}
\begin{lem}\label{lem:scaling_elements}
 If  there is a  contracting precompact open set $V$, then there are $b,c \in A\rtimes L\setminus\{0\}$ satisfying $b^*bb=b$, $b^*bc=c$, $b^*c=0$ and $a b=b $ for all
$a \in C_0(X)$ which is $1$ on $V$.
\end{lem}
\begin{proof} Since $U_k\subseteq \Delta_{\reg, n_k}$ and $\overline{V}\subseteq \bigcup_{k=1}^m \varphi^{n_k}(U_k)$ we may find $a_k\in C_c(U_k)^+$, for $k=1,...,m$, such that $g=\sum_{k=1}^m L^{n_k}(a_k)\in C_0(X)$ is $1$ on $V$. Indeed, each $U_k=\bigcup_{i=1}^{m_k} U_k^i$ is a union of open sets where $\varphi^{n_k}|_{U_k^i}$ is a homeomorphism onto its range. Taking a partition of unity
$\{h_k^i\}_{k=1, i=1}^{m,m_k}\subseteq C_0(X)$ on $\overline{V}$
 subordinated to $\{\varphi^{n_k}(U_k^i)\}_{k=1, i=1}^{m,m_k}$ we may put $a_k:=\sum_{i=1}^{m_k} \varrho_{n_k}^{-1}\cdot h_{k}^i\circ (\varphi^{n_k}|_{U_k^i})^{-1}$.
Now define $b=\sum_{k=1}^m \sqrt{a_k} t^{n_k}$. Then for $a \in C_0(X)$ which is $1$ on $V$ we have $a b= b$.
Since the sets $\{U_k\}_{k=1}^m$ are pairwise disjoint we get
$
b^*b=\sum_{k,l=1}^m t^{* n_k} \sqrt{a_k} \sqrt{a_l} t^{n_l}=\sum_{k=1}^m t^{* n_k} a_k t^{n_k}=g.
$
Since $g$ is $1$ on $\overline{V}$, we have $b^*bb=b$. Take any non-zero $c\in C_c(V\setminus \overline{\bigcup_{k=1}^m U_k})$, then clearly,
$b^*bc=c$ and $b^*c=0$.
\end{proof}

\begin{prop}\label{prop:infinite_projection}
 If  $L$ is minimal and there is a  contracting open set, then $A\rtimes L$ is simple and contains an infinite projection.
\end{prop}
\begin{proof}
If there is a contracting set, then $\varphi: \Delta_{\reg}\to X$ can not be a directed graph with one circuit.
Hence $A\rtimes L$ is simple by Theorem \ref{thm:test_of_simplicity}. By Lemma \ref{lem:scaling_elements} there is
$b \in A\rtimes L$ with $b^*bb=b$, $b^*b\neq bb^*$. Such elements are called \emph{scaling}, and
a simple $C^*$-algebra has an infinite projection if and only if it has a scaling element,  see  \cite[Proposition 4.2]{katsura0}.
\end{proof}

\begin{lem}\label{lem:cutting_elements}
 Assume that $\varphi:\Delta_{\reg}\to X$ is topologically free, $\Delta=\Delta_{\pos}$  and there is $x_0\in \Delta$ such that $\overline{\bigcup_{n=0}^{\infty}\varphi^{-n}(x_0)\cap \Delta_{\reg,n}}=X$ .
For any non-zero positive $b\in A\rtimes L$  there is $d\in A\rtimes L$ and $a\in A$ which is $1$ on some neighbourhood of $x_0$ and $\|d^* b d - a\|< 1/2$.
\end{lem}
\begin{proof}
Put $\varepsilon:=\|E(b)\|/5$ where $E$ is the conditional expectation onto the core $A_{\infty}$. Choose a positive $b_0\in \text{span}\{at^n t^{*m}c: a\in I_n,c\in I_m, n,m\in \N_0\}$  such that $\|b- b_0\|<\varepsilon$.
Since  $\Delta=\Delta_{\pos}$, in view of Lemma \ref{lem:reg_vs_positive_freeness}, we see that $\varphi:\Delta\to X$ is topologically free, and this implies that the topological quiver $\QQ=(X, \Delta, id, \varphi,\mu)$ satisfies condition (L).
 By Corollary \ref{cor:topological_quivers} the $C^*$-algebra associated to $\QQ$ is isomorphic to $A\rtimes L$. Hence we may apply \cite[Proposition 6.14]{mt}
to conclude that there is $d_0\in A\rtimes L$ and $a_0\in C_0(X)^+$ such that  $\|d_0\|\leq 1$, $\|a_0\|=\|E(b_0)\|$ and $\|d_0^* b_0 d_0 - a_0\|<\varepsilon$.
Since $\|a_0\|=\|E(b_0)\|>\|E(b)\|-\varepsilon=4\varepsilon$, the open set
$$
U:=\{x\in X: a_0(x)> 4 \varepsilon\}
$$
is non-empty. As $\bigcup_{n=0}^{\infty}\varphi^{-n}(x_0)\cap \Delta_{\reg,n}$ is dense in
$X$ there is $n\in\N$ and an open subset $U_0\subseteq U\cap \Delta_{\reg,n}$ such that $\varphi^n|_{U_0}$ is a local homeomorphism onto
an open neighbourhood $V_0$ of $x_0$. Take $c_0\in C_c(V_0)^+$, $\|c_0\|\leq 1$, such that $c_0$ is $1$ on an open neighbourhood $V\subseteq V_0$ of $x_0$.
Then $c:= (a_0\varrho_n)^{-1} \cdot c_0\circ (\varphi^n|_{U_0})^{-1}\in C_c(U_0)$ is such that $a:=L^{n}(ca_0c)$ is $1$ on  $V$ and
$\|c t^n \|^2=\|L^n(c^2)\|\leq \max_{x\in U_0} a_0(x)^{-1}< (4 \varepsilon)^{-1}$.
Thus putting
$d:=d_0 c t^n$ we get
\begin{align*}
\|d^*bd- a\|&=\|t^{*n} c d_0^* b d_0 c t^n-  t^{*n} ca_0c t^n \|\leq \|d_0^* b d_0 -  a_0\| \cdot \|c t^n \|^2
\\
&\leq (\|d_0^* (b-b_0) d_0 \| +\|d_0^* b_0 d_0 - a_0\|)\cdot \|c t^n \|^2<(\varepsilon +\varepsilon)(4 \varepsilon)^{-1}<1/2.
\end{align*}
\end{proof}

\begin{thm}\label{thm:purely_infinite} If $L$ is minimal and contractive, then  $A\rtimes L$ is  purely infinite and simple.
\end{thm}
\begin{proof} By Proposition \ref{prop:infinite_projection}, $A\rtimes L$ is simple and contains an infinite projection $p$. Hence it suffices to show
that for each non-zero positive $b_0\in A\rtimes L$, the hereditary $C^*$-subalgebra $\overline{b_0 (A\rtimes L) b_0}$ generated by $b_0$ contains a
projection equivalent to the infinite projection $p$. To this end, note that by Theorem \ref{thm:test_of_simplicity}, $\varphi:\Delta_{\reg}\to X$ is topologically
free. Let $x_0\in \Delta$ be such that $\overline{\bigcup_{n=0}^{\infty}\varphi|_{\Delta_{\reg}}^{-n}(x_0)}=X$ and
every neighbourhood of $x_0$  contains a non-empty contracting open set. By Lemma \ref{lem:cutting_elements}  there is $d\in A\rtimes L$ and $a\in A$ which is $1$ on a neighbourhood $V$ of $x_0$ and $\|d^* b_0 d - a\|< 1/2$. We may take $V$ to be a precompact open contracting set. By Lemma \ref{lem:scaling_elements}
there are non-zero  $b,c \in A\rtimes L$ such that
\begin{equation}\label{eq:scaling_relations}
b^*bb=b, \quad b^*bc=c, \quad b^*c=0\quad \text{and}\quad a b=b .
\end{equation} % for all
%$a \in C_0(X)$ which is $1$ on $V$.
Since  $A\rtimes L$ is simple and
$c^*c\neq 0$  there are $b_1,...,b_l\in A\rtimes L$ such that
$
p=\sum_{k=1}^l b_k^*c^*c b_k,
$
see \cite[Lemma 4.1]{katsura0}. Set $e:=\sum_{k=1}^l b^k c b_k$. Then using \eqref{eq:scaling_relations} we get
$$
e^* a e= e^*e=\sum_{k,i=1}^l  b_k c^*  b^{k*}   b^{i}c b_i=\sum_{k=1}^l  b_k c^* c b_k=p.
$$
In particular, $\|e\|=1$. Using all these we get
$$
\|e^*d^* b_0 d e - p\|=\|e^*(d^* b_0 d e - a) e\|< 1/2.
$$
Let $f$ be the characteristic function of the interval $(\frac{1}{2},\frac{3}{2})$. Then $p_0:=f(e^*d^* b_0 d e)$ is
a well defined projection with $\|p_0-e^*d^* b_0 d e\|\leq 1/2$, cf. \cite[Lemma 2.2.4]{RLL}.
Hence  $\|p_0 - p\|<1$ and therefore $p_0$ and $p$ are equivalent, cf. \cite[Proposition 2.2.5]{RLL}.
Then $q:=f(\sqrt{b_0}d e  e^*d^* \sqrt{b_0})$ is a projection in $\overline{b_0 (A\rtimes L) b_0}$ which is equivalent to $p_0$ and hence to the
infinite projection $p$.
\end{proof}
\begin{cor}\label{cor:Kirchberg}
If $X$ is second countable and $L$ is minimal and contractive, then $A\rtimes L$ is a Kirchberg algebra, i.e. a simple, separable, nuclear, purely infinite $C^*$-algebra satisfying the UCT.
\end{cor}
\begin{proof}
Combine Theorems \ref{thm:purely_infinite} and \ref{thm:crossed_product_Cuntz-Pimsner}.
\end{proof}

\begin{ex}[$C^*$-algebras of rational maps]\label{ex:pure_infinite1}
Let $R:\widehat{\C}\to \widehat{\C}$ be a rational map of degree at least two, and let $X=\Delta=J_R$ be the Julia set for $R$. The  transfer operator $L:C(J_R)\to C(J_R)$ for $\varphi:J_R\to J_R$, considered in Example \ref{ex:Riemann_maps}, is minimal and contractive.
Indeed, $\Delta_{\pos}=J_R$ is uncountable and $\Delta_{\pos}\setminus\Delta_{\reg}$ is finite by \cite[Corollary 2.7.2,  Theorem 4.2.4]{Beardon}.
 Moreover, for any open $V\subseteq J_R$ there is $n\in \N$ such that $R^n(V)=J_R$, by \cite[Theorem 4.2.5]{Beardon},  and  $\overline{\bigcup_{n=0}^{\infty}R^{-n}(z)}=J_{R}$  for every $z\in J_R$, by \cite[Theorem 4.2.7]{Beardon}.
So one may find $z_0\in J_R$ whose inverse orbit $\bigcup_{n=0}^{\infty}R^{-n}(z)$
does not contain any critical point, and any open neighbourhood of $z_0$
 contains a contractive open set.
Hence $C(J_R)\rtimes L$ is simple and purely infinite. This recovers \cite[Theorem 3.8]{Kajiwara_Watatani0} as a special case of Theorem \ref{thm:purely_infinite}.
Note that  the Fatou set $F_R=\widehat{\C}\setminus J_R$ is open and invariant for $R:\widehat{\C}\to \widehat{\C}$. Thus
$C(\widehat{\C})\rtimes L$ is simple if and only if $\widehat{\C}=J_R$.

As a by product we also recover the main result of  \cite{Hamada}. Namely, let $\mu^L$ be the Lyubich measure. It is a $\varphi$-invariant regular probability measure whose support is $J_R$. Denoting by $T_\varphi$ the composition operator on $L_2(\mu^L)$ and identifying $C(J_R)$ with operators of multiplication on $L_2(\mu^L)$ we get $L(a)= d\cdot T_{\varphi}^* a T_{\varphi}$ where $d$ is degree of $R$, see  \cite[Lemma, p. 366]{Lyubich}. Thus using Proposition \ref{prop:characterization_of_covariance} one gets that $(id, d^{1/2}T_{\varphi})$ is covariant representation of $L:C(J_R)\to C(J_R)$ and therefore the simple $C^*$-algebra $C(J_R)\rtimes L$ is isomorphic to the $C^*$-subalgebra of $B(L_2(\mu^L))$ generated by  $C(J_R)$ and the composition operator $T_\varphi$.
\end{ex}

\begin{ex}[Branched expansive coverings]\label{ex:pure_infinite2} Consider the transfer operator from  Example \ref{ex:branched_maps},
where $\varphi:X \to X$ is a continuous map on a compact metric space $X$ whose  inverse has a finite number of continuous
branches  $\gamma=\{\gamma_i\}_{i=1}^N$ that are proper contractions and  $X$ is self-similar for $\gamma$. In other words,
$X$ is covered by compact sets $\Delta_i$, $i=1,...,N$, such that $\varphi:\Delta_i\to X$ is an expansive homeomorphism ($\gamma_i=\varphi|_{\Delta_i}^{-1}$). %: there exist $c_1,c_2$
 As in \cite{Kajiwara_Watatani} we assume \emph{the open set condition} for $\gamma$, which in terms of $\varphi$ says that there is a non-empty open set $V\subseteq X$, such that
$\varphi^{-1}(V)\subseteq V$ and $\varphi^{-1}(V)\cap \Delta_{i}\cap \Delta_j=\emptyset$ for $i\neq j$.
Then $\varphi^{-1}(V)$ is necessarily an open dense set in $X$ not intersecting the set of branching points
$
B=\bigcup_{i\neq j} \{x\in \Delta_i\cap \Delta_j\}
$
and we have $\Delta_{\reg}=X\setminus B$, cf.  \cite[Proposition 2.6]{Kajiwara_Watatani}.
Using this one infers that each of the sets $\varphi^{n}(B)$ has empty interior.
Thus  $X\setminus \bigcup_{n=0}^\infty \varphi^{n}(B)$ is dense in $X$ by Baire theorem.
For any $x\in X\setminus \bigcup_{n=0}^\infty \varphi^{n}(B)$ its negative orbit $\bigcup_{n=0}^{\infty}\varphi^{-n}(x)$ lies entirely in $\Delta_{\reg}$.
Every  negative orbit is dense in $X$. Indeed,  $A:=\overline{\bigcup_{n=0}^{\infty}\varphi^{-n}(x)}$ is a closed set with $\varphi^{-1}(A)\subseteq A$, which implies  $A=X$ by the uniqueness of the self-similar set $X$, see  \cite{Huchinson}. Using  expansiveness of $\varphi$ we conclude that every neighbourhood of $x_0\in X\setminus \bigcup_{n=0}^\infty \varphi^{n}(B)$  contains a non-empty contracting open set. Minimality is clear. Hence
$C(X)\rtimes L$ is a unital Kirchberg algebra by Corollary \ref{cor:Kirchberg}.  This recovers \cite[Theorem 3.8]{Kajiwara_Watatani} when the systems of contractive maps  form inverse branches of a continuous map.

Recall that the \emph{Hutchinson measure} $\mu^H$ is the  unique regular  probability measure such that $\mu^H(A)=1/N\sum_{i=1}^N \mu^H(\gamma_i(A))$ for all Borel $A\subseteq X$. It support is $X$ and
so we may identify $C(X)$ with operators of multiplication on $L_2(\mu^H)$. If $\mu^H(B)=0$ (which is automatic when $X\subseteq \R^d$ and $\gamma_i$'s are similitudes, see \cite{Schief}), then  the composition operator $T_\varphi$ is an isometry on $L_2(\mu^L)$ satisfying  $L(a)= 1/N\cdot T_{\varphi}^* a T_{\varphi}$, see  \cite{Hamada2}. Thus using Proposition \ref{prop:characterization_of_covariance} one sees that $(id, N^{-1/2}T_{\varphi})$ is a covariant representation of $L:C(J_R)\to C(J_R)$ and therefore  $C(X)\rtimes L$ is isomorphic to the $C^*$-subalgebra of $B(L_2(\mu^H))$ generated by  $C(X)$ and the composition operator $T_\varphi$. This recovers the main result of \cite{Hamada2}.
\end{ex}
\begin{ex}[Expanding local homeomorphisms]\label{ex:pure_infinite3}
Assume $\varphi:X\to X$ is an open continuous expanding map on a compact metric space, cf. \cite{Anantharaman-Delaroche}, \cite{bar_kwa} and references therein.
Then any continuous $\varrho:X\to (0,\infty)$ defines a transfer operator $L$ for $\varphi$ and  $C(X)\rtimes L\cong C^*(\G)$, by Theorem \ref{thm:local_homeo_crossed_groupoid}.
By \cite[Lemma 7.4]{bar_kwa}, $\varphi$ is topologically free if and only if $X$ has no isolated periodic points. Clearly, $L$ is minimal if and only if $\varphi$ is \emph{minimal}, i.e. there is no non-trivial open set $U$ with $\varphi^{-1}(U)=U$.
Thus assuming  $X$ is infinite, by Theorem \ref{thm:test_of_simplicity}, we get
that
$$
C(X)\rtimes L\text{  is simple if and only if $\varphi$ is minimal}.
$$
Assume now that there are no wandering points in $X$, or equivalently that periodic points are dense in $X$.
Then by spectral decomposition, cf.  \cite[Theorem 2.5]{bar_kwa}, $\varphi$ is minimal iff  $\varphi$ is topologically transitive iff
for every non-empty open $U\subseteq X$ there is $N\in \N$ such that $\bigcup_{k=1}^N\varphi^k(U)=X$.
Thus if $\varphi$ is minimal, then every negative orbit $\bigcup_{n=0}^{\infty}\varphi^{-n}(x)$ is dense in $X$ and every non-trivial open subset $V\subsetneq X$ is contracting,  so in particular $L$ is contracting.
Hence by Corollary \ref{cor:Kirchberg} we get
$$
\text{$C(X)\rtimes L$ is a Kirchberg algebra, if $\varphi$ is minimal and has no wandering points}.
$$
This last statement improves  \cite[Proposition 4.2]{Anantharaman-Delaroche} (in the minimal case) and  implies  \cite[Proposition 4.2]{Exel_Huef_Raeburn}.
If there are no wandering points, then by \cite[Proposition 3.8]{bar_kwa}   there exists a  $\varphi$-invariant Borel probability measure $\mu$ with support $X$  such that identifying $C(X)$ with operators of multiplication on $L_2(\mu)$ we have $L(1)^{-1}L(a)=T_\varphi^* a T_\varphi$, for $a\in C(X)$, where $T_\varphi\in B(L_2(\mu))$ is the composition operator with $\varphi$.
Also $(id, L(1)^{-1/2}T_{\varphi})$ is a covariant representation of $L$. Thus is if $X$ has no isolated periodic points, then  $C(X)\rtimes L$ is isomorphic to the $C^*$-subalgebra of $B(L_2(\mu))$ generated by  $C(X)$ and the composition operator $T_\varphi$.
If in addition $\varphi$ is minimal (topologically transitive) and $\ln \varrho$ is H\"older continuous, then the above measure $\mu$ is unique and it is the \emph{Gibbs measure} for $\varphi$ and $\ln\varrho$.

\end{ex}

\end{document}